\documentclass{amsart}
\usepackage{amssymb}
\usepackage{graphicx}
\usepackage{hyperref}
\usepackage{xcolor} % A package to add color.

\usepackage{tensor}
\usepackage{enumerate}

\newcommand{\LieBr}[2]{[ #1, #2 ]}
\newcommand{\LieGrp}{\mathfrak{G}}
\newcommand{\LieAlg}{\mathfrak{g}}
\newcommand{\covD}{\bfD}

\newcommand{\covDlow}{\underline{\covD}}
\newcommand{\Alow}{\underline{A}}
\newcommand{\Flow}{\underline{F}}
\newcommand{\wlow}{\underline{w}}

\newcommand{\bfAlow}{\underline{\bfA}}

\newcommand{\Aini}{\overline{A}}

\newcommand{\Eini}{\overline{E}}
\newcommand{\Fini}{\overline{F}}
\newcommand{\Vini}{\overline{V}}
\newcommand{\calIini}{\overline{\calI}}

\newcommand{\sbr}{\overline{s}}

\newcommand{\subr}{\underline{s}}

\newcommand{\calAlow}{\underline{\calA}}

\newcommand{\calIFs}{{}^{(F_{s})} \calI}
\newcommand{\calIAlow}{{}^{(\Alow)} \calI}

\newcommand{\Atemp}{A^{\dagger}}

\makeatletter
\newtheorem*{rep@theorem}{\rep@title}
\newcommand{\newreptheorem}[2]{%
\newenvironment{rep#1}[1]{%
 \def\rep@title{#2 \ref{##1}}%
 \begin{rep@theorem}}%
 {\end{rep@theorem}}}
\makeatother
 
\theoremstyle{plain}
\newtheorem*{MainTheorem}{Main Theorem}
\newtheorem*{corrPrinciple}{Correspondence Principle}
\newtheorem{bigTheorem}{Theorem}

\newtheorem{theorem}{Theorem}[section]
\newtheorem{corollary}[theorem]{Corollary}
\newtheorem{lemma}[theorem]{Lemma}
\newtheorem{proposition}[theorem]{Proposition}

\newreptheorem{bigTheorem}{Theorem}
\newreptheorem{proposition}{Proposition}

\theoremstyle{definition}
\newtheorem{definition}[theorem]{Definition}

\theoremstyle{remark}
\newtheorem{remark}[theorem]{Remark}
\numberwithin{equation}{section}

\definecolor{green}{rgb}{0,0.8,0} % Redefines the color green.
\newcommand{\nrm}[1]{\Vert#1\Vert}
\newcommand{\abs}[1]{\vert#1\vert}

\newcommand{\set}[1]{\{#1\}}

\newcommand{\tr}{\textrm{tr}}

\newcommand{\lap}{\triangle}

\newcommand{\ud}{\mathrm{d}}
\newcommand{\rd}{\partial}
\newcommand{\nb}{\nabla}

\newcommand{\bb}{\Big}

\newcommand{\alp}{\alpha}

\newcommand{\dlt}{\delta}

\newcommand{\eps}{\epsilon}

\newcommand{\lmb}{\lambda}

\newcommand{\sgm}{\sigma}

\newcommand{\bfa}{{\bf a}}
\newcommand{\bfb}{{\bf b}}
\newcommand{\bfc}{{\bf c}}

\newcommand{\bfA}{{\bf A}}
\newcommand{\bfB}{{\bf B}}

\newcommand{\bfD}{{\bf D}}
\newcommand{\bfE}{{\bf E}}
\newcommand{\bfF}{{\bf F}}

\newcommand{\bbH}{\mathbb H}

\newcommand{\bbR}{\mathbb R}

\newcommand{\calA}{\mathcal A}
\newcommand{\calB}{\mathcal B}

\newcommand{\calD}{\mathcal D}
\newcommand{\calE}{\mathcal E}
\newcommand{\calF}{\mathcal F}

\newcommand{\calH}{\mathcal H}
\newcommand{\calI}{\mathcal I}

\newcommand{\calL}{\mathcal L}

\newcommand{\calN}{\mathcal N}
\newcommand{\calO}{\mathcal O}

\newcommand{\calQ}{\mathcal Q}

\newcommand{\calX}{\mathcal X}

\newcommand{\pfstep}[1]{\vspace{0.1in} {\it #1.}}
\begin{document}

\title[]{Finite energy global well-posedness of the Yang-Mills equations on $\bbR^{1+3}$: An approach using the Yang-Mills heat flow.}%: Title of the article
\author{Sung-Jin Oh}%
\address{Department of Mathematics, UC Berkeley, Berkeley, CA, 94720}%
\email{sjoh@math.berkeley.edu}%

\thanks{}%

%\date{\today}%
\dedicatory{}%
\commby{}%
% ----------------------------------------------------------------
\begin{abstract}
In this work, along with the companion work \cite{Oh:6stz7nRe}, we propose a novel approach to the problem of gauge choice for the \emph{Yang-Mills equations} on the Minkowski space $\bbR^{1+3}$. A crucial ingredient is the associated \emph{Yang-Mills heat flow}. As this approach does not possess the drawbacks of the previous approaches (as in \cite{Klainerman:1995hz}, \cite{Tao:2000vba}), it is expected to be more robust and easily adaptable to other settings.

Building on the results proved in the companion article \cite{Oh:6stz7nRe}, we prove, as the first application of our approach, finite energy global well-posedness of the Yang-Mills equations on $\bbR^{1+3}$. This is a classical result first proved by S. Klainerman and M. Machedon \cite{Klainerman:1995hz} using local Coulomb gauges. As opposed to their method, the present approach avoids the use of Uhlenbeck's lemma \cite{Uhlenbeck:1982vna}, and hence does \emph{not} involve localization in space-time.
\end{abstract}
\maketitle
% ----------------------------------------------------------------

\section{Introduction} \label{sec:introduction}
%In this paper, we continue the investigation in \cite{Oh:6stz7nRe} of the problem of gauge choice for the non-abelian \emph{Yang-Mills equations}
%\begin{equation*}
%	\covD^{\mu} F_{\nu \mu} =0
%\end{equation*}
% on the Minkowski space $\bbR^{1+3}$. (For the notations, we refer the reader to \S \ref{subsec:intro:bg}.) The traditional choices of gauge for the Yang-Mills equations include the \emph{(local) Coulomb gauge} $\rd^{\ell} A_{\ell}= 0$ \cite{Klainerman:1995hz} and the \emph{temporal gauge} $A_{0} = 0$ \cite{Segal:1979hg}, \cite{Eardley:1982fb}--\cite{Eardley:1982cz}, \cite{Tao:2000vba}. Each, however, possesses its own set of shortcomings, because of which there had not been many results on low regularity solutions to the Yang-Mills equations with possibly \emph{large} data. 
% 
In this article, along with the companion article \cite{Oh:6stz7nRe}, we propose a novel approach to the problem of gauge choice for the \emph{Yang-Mills equations}
\begin{equation*}
	\covD^{\mu}  F_{\nu \mu} =0
\end{equation*}
 on the Minkowski space $\bbR^{1+3}$ with a non-abelian structural group $\LieGrp$. (For the notations, we refer the reader to \S \ref{subsec:intro:bg}.) An essential ingredient of our approach is the celebrated \emph{Yang-Mills heat flow}
 \begin{equation*}
	\rd_{s} A_{i} = \covD^{\ell} F_{\ell i},
\end{equation*}
which, first proposed by Donaldson \cite{Donaldson:1985vh}, is a well-studied equation in the field of geometric analysis. See \cite{Rade:1992tu}, \cite{Charalambous:2010vt} and etc. 

In the companion article \cite{Oh:6stz7nRe}, the new Yang-Mills heat flow approach has been employed to establish local well-posedness of the Yang-Mills equations for initial data $(\Aini_{i}, \Eini_{i}) \in (\dot{H}^{1}_{x} \cap L^{3}_{x} ) \times L^{2}_{x}$ (see Theorem \ref{thm:lwp4YM} for the precise statement). The key idea in \cite{Oh:6stz7nRe} is that the Yang-Mills heat flow can be used to construct a new global gauge, dubbed the \emph{caloric-temporal gauge}, in which the hyperbolicity and null structure of the Yang-Mills equations become manifest. 

In the present article, we improve this local well-posedness result to \emph{finite energy global well-posedness} of the Yang-Mills equations; see the Main Theorem in \S \ref{subsec:intro:mainThm} for the precise statement. Unlike the simpler model equation $\Box u = u^{3}$, which possesses the same scaling property as the Yang-Mills equations, finite energy global well-posedness of the Yang-Mills equations is \emph{not} an immediate consequence of the local well-posedness result and conservation of energy. The difficulty is that the energy $\bfE[F_{\mu\nu}](t)$, defined in \eqref{eq:YMenergy}, \emph{a priori} only controls certain components of $\rd_{t,x} A_{i}(t,x)$ in $L^{2}_{x}$; this phenomenon is connected to the gauge-invariance of the Yang-Mills equations. We will overcome this difficulty using the Yang-Mills heat flow approach, as developed in \cite{Oh:6stz7nRe}. 

Finite energy global well-posedness of the Yang-Mills equations was first proved by Klainerman-Machedon \cite{Klainerman:1995hz} using the method of local Coulomb gauges, which involves a delicate space-time localization procedure. 
This makes global Fourier-analytic techniques which are necessary in the low regularity setting (e.g., $H^{s, \dlt}$ or $X^{s,b}$ spaces) difficult to apply when the initial data is large. For a more detailed discussion, see \S \ref{subsec:intro:motivation} below. Our approach, on the other hand, does not involve localization in space-time, and therefore is more robust and easily adaptable to other settings. A discussion of future applications of our approach will be given in \S \ref{subsec:intro:apps}. 

%The idea of using a geometric heat flow to deal with the problem of gauge choice had been first put forth by Tao \cite{Tao:2004tm}, \cite{Tao:2008wn} in the context of energy critical wave maps on $\bbR^{1+2}$, and has been also adapted to the related energy critical Schr\"odinger maps by \cite{Bejenaru:2011wy}, \cite{Smith:2011ef}, \cite{Smith:2011ty}. \comment{Add references.}

\subsection{Background: The Yang-Mills equations on $\bbR^{1+3}$} \label{subsec:intro:bg}
We will work on the Minkowski space $\bbR^{1+3}$, equipped with the Minkowski metric of signature $(-+++)$. All tensorial indices will be raised and lowered by using the Minkowski metric. Moreover, we will adopt the Einstein summation convention of summing up repeated upper and lower indices. Greek indices, such as $\mu, \nu, \lmb$, will run over $x^{0}, x^{1}, x^{2}, x^{3}$, whereas latin indices, such as $i, j, k, \ell$, will run \emph{only} over the spatial indices $x^{1}, x^{2}, x^{3}$. We will often use $t$ for $x^{0}$.

Let $\LieGrp$ be a Lie group with the Lie algebra $\LieAlg$, which is equipped with a bi-invariant inner product\footnote{A \emph{bi-invariant inner product} is an inner product on $\LieAlg$ invariant under the adjoint map. A sufficient condition for the existence of such an inner product is that $\LieGrp$ be a product of an abelian and a semi-simple Lie groups.} $(\cdot, \cdot) : \LieAlg \times \LieAlg \to [0, \infty)$. The bi-invariant inner product will be used to define the absolute value of elements in $\LieAlg$, and moreover will be used in turn to define the $L^{p}_{x}$-norm of $\LieAlg$-valued functions.

For simplicity, we will assume that $\LieGrp$ is a matrix group. An explicit example which is useful to keep in mind is the group of special unitary matrices $\LieGrp = \mathrm{SU}(n)$, in which case $\LieAlg = \mathfrak{su}(n)$ is the set of complex traceless anti-hermitian matrices and the bi-invariant metric is given by $(A, B) := \tr(A B^{\star})$. 

Consider a $\LieAlg$-valued 1-form $A_{\mu}$ on $\bbR^{1+3}$, which we will call a \emph{connection 1-form}, or \emph{connection coefficients}\footnote{We take a fairly pragmatic point of view towards the definitions of geometric concepts (such as connection and curvature), for the sake of simplicity. For more information on the geometric background of the concepts introduced here (involving principal bundles, associated vector bundles and etc.), we recommend the reader the standard references \cite{Bleeker:2005uj}, \cite{Kobayashi:1963uh}.}. For any $\LieAlg$-valued tensor field $B$ on $\bbR^{1+3}$, we define the associated \emph{covariant derivative} $\covD = {}^{(A)} \covD$ by 
\begin{equation*}
	\covD_{\mu} B := \rd_{\mu} B + \LieBr{A_{\mu}}{B}, \quad \mu = 0,1,2,3
\end{equation*}
where $\rd_{\mu}$ is the ordinary directional derivative on $\bbR^{1+3}$. 

The commutator of two covariant derivatives gives rise to a $\LieAlg$-valued 2-form $F_{\mu \nu}$, called the \emph{curvature 2-form} associated to $A_{\mu}$, in the following fashion.
\begin{equation*}
	\covD_{\mu} \covD_{\nu} B  - \covD_{\nu} \covD_{\mu} B = \LieBr{F_{\mu \nu}}{B}.
\end{equation*}

Using the definition, it is not difficult to verify that $F_{\mu \nu}$ is expressed directly in terms of $A_{\mu}$ by the formula
\begin{equation*}
	F_{\mu \nu} = \rd_{\mu} A_{\nu} - \rd_{\nu} A_{\mu} + \LieBr{A_{\mu}}{A_{\nu}}.
\end{equation*}

From the way $F_{\mu \nu}$ arises from $A_{\mu}$, it follows that the following \emph{Bianchi identity} holds.
\begin{equation} \label{eq:intro:bianchi} \tag{Bianchi}
	\covD_{\mu} F_{\nu \lmb} + \covD_{\nu} F_{\lmb \mu} + \covD_{\lmb} F_{\mu \nu} = 0.
\end{equation}

A connection 1-form $A_{\mu}$ is said to be a solution to the \emph{Yang-Mills equations} \eqref{eq:hyperbolicYM} on $\bbR^{1+3}$ if the following equation holds for $\nu = 0,1,2,3$.
\begin{equation} \label{eq:hyperbolicYM} \tag{YM}
	\covD^{\mu} F_{\mu \nu} = 0.
\end{equation}

Note the similarity of \eqref{eq:intro:bianchi} and \eqref{eq:hyperbolicYM} with the Maxwell equations $\ud F =0$ and $\rd^{\mu} F_{\nu \mu} = 0$. In fact, the Maxwell equations are a special case of \eqref{eq:hyperbolicYM} in the case $\LieGrp = \mathrm{SU}(1)$.

An essential feature of \eqref{eq:hyperbolicYM} is the \emph{gauge structure}, which we explain now. Let $U$ be a smooth $\LieGrp$-valued function. This $U$ may act on $A, \covD, F$ as a \emph{gauge transform} according to the following rules:
\begin{equation*}
\begin{aligned}
	\widetilde{A}_{\mu} =  U A_{\mu} U^{-1} - \rd_{\mu} U U^{-1}, \qquad
	\widetilde{\covD}_{\mu} =  U \covD_{\mu} U^{-1}, \qquad
	\widetilde{F}_{\mu \nu} =  U F_{\mu \nu} U^{-1}.
\end{aligned}
\end{equation*} 

If a $\LieAlg$-valued tensor transforms in the fashion $\widetilde{B} = U B U^{-1}$, then we say that it is \emph{gauge covariant}, or \emph{covariant under gauge transforms}. Note that the curvature 2-form is gauge covariant. Given a gauge covariant $B$, its covariant derivative $\covD_{\mu} B$ is also gauge covariant, as the following formula shows:
\begin{equation*}
	\widetilde{\covD}_{\mu} \widetilde{B} = U \covD_{\mu} B U^{-1}.
\end{equation*}

Due to bi-invariance, we furthermore have $(\widetilde{B},\widetilde{B}) = (B, B)$.

Note that \eqref{eq:hyperbolicYM} is evidently covariant under smooth gauge transforms. It has the implication that a solution to \eqref{eq:hyperbolicYM} makes sense only as a class of connection 1-forms, which are connected to each other by smooth gauge transforms. Accordingly, we make the following definition.

\begin{definition} 
A \emph{classical solution} to \eqref{eq:hyperbolicYM} is a class of smooth connection 1-forms $A$ satisfying \eqref{eq:hyperbolicYM}, which are related to each other by smooth gauge transforms. A \emph{generalized solution} to \eqref{eq:hyperbolicYM} is defined to be a class of gauge equivalent connection 1-forms $A$ for which there exists a sufficiently smooth representative $A$ (say $\rd_{t,x} A \in C_{t} L^{2}_{x}, A \in C_{t} L^{3}_{x}$) which satisfies \eqref{eq:hyperbolicYM} in the sense of distributions.
\end{definition}

A choice of a particular representative will be referred to as a \emph{gauge choice}. A gauge is usually chosen by imposing a condition, called a \emph{gauge condition}, on the representative. Some classical examples of gauge conditions are the \emph{temporal gauge} $A_{0} = 0$, or the \emph{Coulomb gauge }$\rd^{\ell} A_{\ell} = 0$, where $\ell$, being a latin index, is summed only over the spatial indices $1,2,3$.

In this work, as well as the companion paper \cite{Oh:6stz7nRe}, we study the Cauchy problem associated to \eqref{eq:hyperbolicYM}. As in the case of Maxwell equations, an initial data set consists of $(\Aini_{i}, \Eini_{i})$ for $i=1,2,3$, where $\Aini_{i} = A_{i}(t=0)$ (magnetic potential) and $\Eini_{i} = F_{0i}(t=0)$ (electric field). Note that one component of \eqref{eq:hyperbolicYM}, namely $\nu = 0$, imposes the following \emph{constraint equation} on the possible initial data set $(\Aini_{i}, \Eini_{i})$:
\begin{equation} \label{eq:YMconstraint}
	\rd^{\ell} \Eini_{\ell} + \LieBr{\Aini^{\ell}}{\Eini_{\ell}} = 0.
\end{equation}

The energy of $F_{\mu \nu}$ at time $t$ is defined by
\begin{equation} \label{eq:YMenergy}
	\bfE[F_{\mu \nu}](t) = \frac{1}{2} \int_{\bbR^{3}} \sum_{\ell=1,2,3} (F_{0\ell}(t,x), F_{0\ell}(t,x)) + \sum_{k, \ell=1,2,3, \, k < \ell} (F_{k \ell}(t,x), F_{k \ell}(t,x))  \, \ud x.
\end{equation}

The energy $\bfE[F_{\mu \nu}](t)$ is conserved under \eqref{eq:hyperbolicYM}, i.e., for sufficiently smooth solution $A_{\mu}$ to \eqref{eq:hyperbolicYM} on an interval $I$, $\bfE[F_{\mu \nu}](t_{1}) = \bfE[F_{\mu \nu}](t_{2})$ for every $t_{1}, t_{2} \in I$. Conservation of energy holds, in particular, if $A_{\mu}$ is a regular solution to \eqref{eq:hyperbolicYM}, in the sense of Definition \ref{def:mainThm:reg4YM} below.

Note that \eqref{eq:hyperbolicYM} remain invariant under the scaling
\begin{equation} \label{eq:YMscaling}
	x^{\alp} \to \lmb x^{\alp}, \quad A \to \lmb^{-1} A, \quad F \to \lmb^{-2} F.
\end{equation}

We remark that the conserved energy $\bfE(t)$ and also the norms $\nrm{\rd_{x} \Aini_{i}}_{L^{2}_{x}}$, $\nrm{\Eini_{i}}_{L^{2}_{x}}$ \emph{decrease} as $\lmb$ increases according to the above scaling. This reflects the \emph{sub-criticality} of these quantities compared to the scaling property of \eqref{eq:hyperbolicYM}.

%\comment{Discuss the importance of Yang-Mills; add references.}

\subsection{The problem of gauge choice and previous approaches} \label{subsec:intro:motivation}
We will begin with a discussion on the importance and difficulty of the problem of choosing an appropriate gauge in the study of the Yang-Mills equations. Our discussion will revolve around the following concrete example, which is a classical result of Klainerman-Machedon \cite{Klainerman:1995hz}, stated in a simplified form.

\begin{theorem}[Klainerman-Machedon \cite{Klainerman:1995hz}] \label{thm:intro:KM}
Let $(\Aini_{i}, \Eini_{i})$ be a smooth initial data set satisfying the constraint equation \eqref{eq:YMconstraint}. Consider the Cauchy problem for these data.
\begin{enumerate}
\item \emph{($H^{1}_{x}$ local well-posedness)} There exists a classical solution $A_{\mu}$ to the Cauchy problem for \eqref{eq:hyperbolicYM} on a time interval $(-T^{\star}, T^{\star})$, where $T^{\star} > 0$ depends only on $\nrm{\Aini_{i}}_{\dot{H}^{1}_{x}}, \nrm{\Eini_{i}}_{L^{2}_{x}}$. The solution is unique in an appropriate gauge, e.g. in the temporal gauge $A_{0} = 0$.
\item \emph{(Finite energy global well-posedness)} Furthermore, if the initial data possess finite energy $\bfE(0) < \infty$, then the solution $A_{\mu}$ extends globally.
\end{enumerate}
\end{theorem}

%\begin{remark} 
%The result stated above can be compared with that of Eardley and Moncrief \cite{Eardley:1982fb} -- \cite{Eardley:1982cz}, which gives global well-posedness of the Yang-Mills equations, but for initial data with a higher regularity.
%\end{remark}

After explaining the importance of gauge choice for proving Theorem \ref{thm:intro:KM}, we will briefly summarize the previous approaches to the problem of gauge choice, namely the \emph{(local) Coulomb gauge} \cite{Klainerman:1995hz} and the \emph{temporal gauge} \cite{Segal:1979hg}, \cite{Eardley:1982fb}, \cite{Tao:2000vba}. It will be seen that each has its own set of drawbacks, which in fact makes Theorem \ref{thm:intro:KM} the best result so far in terms of the regularity condition on the initial data, concerning local and global well-posedness of \eqref{eq:hyperbolicYM} for possibly \emph{large}\footnote{We remark that there are better results in the case of \emph{small} initial data, for the reasons to be explained below. See \cite{Tao:2000vba}.} initial data. This will motivate us to propose a new approach to the problem of gauge choice in \S \ref{subsec:intro:mainIdea}.

%This will motivate us to propose a novel approach to the problem of gauge choice, which uses the associated \emph{Yang-Mills heat flow} as an essential ingredient. Our approach does not possess the drawbacks of the previous approaches, and therefore is expected to be more robust and easily adaptable to other problems. In this work and \cite{Oh:6stz7nRe}, we will provide a new proof of Theorem \ref{thm:intro:KM} as the first application of the novel approach.

\subsubsection*{Importance of gauge choice}
There are at least three reasons why a judicious choice of gauge is needed in order to prove Theorem \ref{thm:intro:KM}:
\begin{description}
\item[{\rm (Hyperbolicity)}] To reveal the \emph{hyperbolicity}\footnote{In this work, we will interpret the notion of \emph{hyperbolicity} in a practical fashion and say that a PDE is \emph{hyperbolic} if its principal part is the wave equation. By `revealing the hyperbolicity of \eqref{eq:hyperbolicYM}', we mean reducing the dynamics of the Yang-Mills system to that of a system of wave equations. As we shall see below, this may involve solving elliptic, parabolic and/or transport equations for some variables.} of \eqref{eq:hyperbolicYM};
\item[{\rm (Null structure)}] To exhibit the `special structure' (namely, the \emph{null structure}) of \eqref{eq:hyperbolicYM};
\item[{\rm (Energy estimate)}] To utilize the conserved \emph{energy} $\bfE(t)$ to \emph{estimate} $\nrm{\rd_{x} A_{i}(t)}_{L^{2}_{x}}$.
\end{description}

%In the future, we will refer to these as Issues A, B and C. 

Concerning (Hyperbolicity), observe that the top order terms of \eqref{eq:hyperbolicYM} at the level of $A_{\mu}$ have the form
\begin{equation*}
	\Box A_{\nu} - \rd^{\mu} \rd_{\nu} A_{\mu} = (\hbox{lower order terms}).
\end{equation*}

In an arbitrary gauge, due to the presence of the undesirable second order term $-\rd^{\mu} \rd_{\nu} A_{\mu}$, it is not even clear whether the equation for $A_{\mu}$ is hyperbolic (i.e., a wave equation). Therefore, in order to study \eqref{eq:hyperbolicYM} as a hyperbolic system of equations, the gauge should be chosen, at the very least,  in a way to reveal the hyperbolicity of \eqref{eq:hyperbolicYM}. We remark that this is analogous to the issue that the Yang-Mills heat flow is only \emph{weakly-parabolic}, to be discussed in \S \ref{subsec:intro:YMHF}.

Addressing the issue of (Hyperbolicity) suffices to prove local well-posedness of \eqref{eq:hyperbolicYM} for sufficiently regular initial data (see \cite{Segal:1979hg}, \cite{Eardley:1982fb}). However, it is still insufficient for Theorem \ref{thm:intro:KM}, because of the issue of (Null structure). After an appropriate choice of gauge, which does not have to be precise for the purpose of this heuristic discussion, the wave equation for the connection 1-form $A$ satisfying \eqref{eq:hyperbolicYM} is of the form
\begin{equation} \label{eq:intro:naiveModel} 
	\Box A = \calO(A, \rd A) + (\hbox{cubic and higher})
\end{equation}
where $\calO(A, \rd A)$ refers to a linear combination of bilinear terms in $A$ and $\rd_{t,x} A$.

At this point, we encounter an important difficulty of proving Theorem \ref{thm:intro:KM}: Strichartz estimates (barely, but in an essential way) fall short of proving $H^{1}_{x}$ local well-posedness of \eqref{eq:intro:naiveModel}, due to the well-known failure of the endpoint $L^{2}_{t} L^{\infty}_{x}$ estimate on $\bbR^{1+3}$. In fact, a counterexample, given by Lindblad \cite{Lindblad:1996ws}, demonstrates that even local existence may fail for a general equation of the form \eqref{eq:intro:naiveModel} with $H^{1}_{x}$ initial data\footnote{We mention the recent work \cite{MR2796047} of Gr\"unrock, which further clarifies the nature of Lindblad's counterexample. Namely, it is proved that local well-posedness for $\Box u = u \rd_{t} u$ (which was considered by Lindblad \cite{Lindblad:1996ws}) holds if the initial data belongs to the Fourier-Lebesgue space $\widehat{H}^{r}_{s}$ (defined by $\nrm{f}_{\widehat{H}^{r}_{s}} := \nrm{(1+\abs{\xi})^{s}\widehat{f}}_{L^{r'}_{\xi}}$) with $1 < r \leq 2$ and $s > \frac{2}{r}$. Observe that $\widehat{H}^{2}_{s} = H^{s}$, and as $(r, s) \to (1,2)$ the space $\widehat{H}^{r}_{s}$ almost reaches the scaling-critical space $\widehat{H}^{1}_{2}$.}. Such considerations indicate that a proof of Theorem \ref{thm:intro:KM} necessarily has to exploit the null structure of \eqref{eq:hyperbolicYM}, which distinguishes \eqref{eq:hyperbolicYM} from a general system of semi-linear equations of the similar form. %As we will see in sequel, this `special structure' will go under the name \emph{null form}. 
Since the precise form of the wave equation for the connection 1-form $A$ is highly dependent on the gauge, it is crucial to make a suitable choice of gauge so as to reveal the structure needed to establish Theorem \ref{thm:intro:KM}. 

Once (Hyperbolicity) and (Null structure) are addressed, low regularity local well-posedness of \eqref{eq:hyperbolicYM} (in particular, Part (1) of Theorem \ref{thm:intro:KM}) can, in principle, be established. However, yet another difficulty remains in proving Part (2) of Theorem \ref{thm:intro:KM}, namely the issue of (Energy estimate). Had the conserved energy $\bfE(t)$ directly controlled $\nrm{\rd_{x}A_{i}(t)}_{L^{2}_{x}}$, finite energy global well-posedness would have followed immediately from $H^{1}_{x}$ local well-posedness. However, recalling the expression for the conserved energy
\begin{equation*}
	\bfE(t) = \frac{1}{2} \sum_{\mu < \nu} \nrm{\rd_{\mu} A_{\nu} - \rd_{\nu} A_{\mu} + \LieBr{A_{\mu}}{A_{\nu}}}_{L^{2}_{x}}^{2},
\end{equation*}
we see that in an arbitrary gauge, $\bfE(t)$ can only control a part of the full gradient of $A_{i}$: Namely, the curl of $A_{i}$, or $\nrm{\rd_{i} A_{j} - \rd_{j} A_{i}}_{L^{2}_{x}}$. Therefore, in order to prove Part (2) of Theorem \ref{thm:intro:KM} as well, the chosen gauge must have a structure which allows for utilizing $\bfE(t)$ to control the $L^{2}_{x}$ norm of the full gradient $\rd_{x} A_{i}(t)$.

\subsubsection*{Approach using the (local) Coulomb gauge: Proof of Klainerman-Machedon \cite{Klainerman:1995hz}}
We will now discuss the approach of Klainerman-Machedon \cite{Klainerman:1995hz} using the \emph{local Coulomb gauge}. As we will see, this approach addresses all of (Hyperbolicity), (Null structure) and (Energy estimate), but possesses the drawback of requiring localization in space-time, causing technical difficulties on the boundaries.

A key observation of Klainerman-Machedon \cite{Klainerman:1995hz} (which in fact goes back to the previous work \cite{Klainerman:1994jb} of Klainerman-Machedon on the related Maxwell-Klein-Gordon equations) was that under the (global) \emph{Coulomb gauge} $\rd^{\ell} A_{\ell} = 0$ imposed everywhere on $\bbR^{1+3}$, the issues of (Hyperbolicity) and (Null structure) are simultaneously resolved. That is:
\begin{itemize}
\item After solving elliptic equations for $A_{0}$ and $\rd_{0} A_{0}$, \eqref{eq:hyperbolicYM} reduce to a system of wave equations for $A_{i}$, and
\item The most dangerous quadratic nonlinearities of the wave equations can be shown to be composed of null forms.
\end{itemize}

More precisely, the wave equation for $A_{i}$ takes the form
\begin{equation*}
	\Box A_{i} = \calQ(\abs{\rd_{x}}^{-1} A ,A) + \abs{\rd_{x}}^{-1} \calQ(A, A) + \hbox{(Less dangerous terms)},
\end{equation*}
where each $\calQ$ is a linear combination of bilinear forms
\begin{equation*}
	Q_{jk}(\phi_{1}, \phi_{2}) = \rd_{j} \phi_{1} \rd_{k} \phi_{2} - \rd_{k} \phi_{1} \rd_{j} \phi_{2}, \quad 1 \leq j < k \leq 3,
\end{equation*}
which are particular examples of null forms, introduced by Klainerman \cite{Klainerman:tc} and Christodoulou \cite{MR820070} in the context of the small data global existence problem for nonlinear wave equations, and first used by Klainerman-Machedon \cite{Klainerman:ei} in the context of low regularity well-posedness. Improved estimates are available for such class of bilinear interactions (see \cite{Klainerman:ei}, \cite{Klainerman:1995vs} and etc.), and therefore the desired local well-posedness can be proved. 

The Coulomb gauge has an additional benefit that $\nrm{\rd_{x} A_{i}(t)}_{L^{2}_{x}}$ may be estimated by $\bfE(t)$ (provided that $A_{i}$ is sufficiently regular to start with), as the Coulomb gauge condition $\rd^{\ell} A_{\ell} = 0$ sets the part of $\rd_{x} A_{i}$ which is not controlled by $\bfE(t)$ (namely the \emph{divergence} of $A$, or $\rd^{\ell} A_{\ell}$, according to Hodge decomposition) to be exactly zero. In other words, the Coulomb gauge settles the issue of (Energy estimate) as well. 

Unfortunately, when the structural group $\LieGrp$ is \emph{non-abelian}, there is a fundamental difficulty in imposing the Coulomb gauge globally in space (i.e., on $\bbR^{3}$ for each fixed $t$). Roughly speaking, it is because when $\LieGrp$ is non-abelian, a gauge transform into the Coulomb gauge is given as a solution to a nonlinear elliptic system of PDEs, for which no good regularity theory is available in the large\footnote{In fact, it is possible to show, by a variational argument, that any $A_{i} \in L^{2}_{x}$ may be gauge transformed to a weak solution $\widetilde{A} \in L^{2}_{x}$ to the Coulomb gauge equation $\rd^{\ell} \widetilde{A}_{\ell} = 0$; see \cite{DellAntonio:1991ih}. The problem is that no further regularity of the gauge transform and $\widetilde{A}$ may be inferred, due to the lack of an appropriate regularity theory.}. A closely related phenomenon is \emph {the Gribov ambiguity} \cite{Gribov:1978eh}, which asserts non-uniqueness of a representative satisfying the Coulomb gauge equation $\rd^{\ell} A_{\ell} = 0$ in some equivalence class of connection 1-forms on $\bbR^{3}$ when $\LieGrp$ is non-abelian.

%\comment{State Uhlenbeck's lemma?}
At a more technical level, this difficulty manifests in the fact that \emph{Uhlenbeck's lemma} \cite{Uhlenbeck:1982vna}, which is a standard result asserting the existence of a gauge transform (possessing sufficient regularity) into the Coulomb gauge, requires the curvature $F$ to be small in $L^{3/2}_{x}$. Note that this norm is invariant under the scaling \eqref{eq:YMscaling}, and therefore cannot be assumed to be small by scaling, unlike the energy $\bfE[\overline{\bfF}]$. To get around this problem, the authors of \cite{Klainerman:1995hz} work in what they call \emph{local Coulomb gauges} in small domains of dependence (in which the required norm of $F$ can be assumed small), and glue the local solutions together by exploiting the finite speed of propagation. The execution of this strategy is quite involved due to the presence of the constraint equations \eqref{eq:YMconstraint}. In particular, it requires a delicate boundary condition for $\Box A_{i}$ in order to mesh the analyses of the elliptic and hyperbolic equations arising from \eqref{eq:hyperbolicYM} in the local Coulomb gauge.

\subsubsection*{Approach using the temporal gauge}
A different route to the problem of gauge choice in the context of low regularity well-posedness was suggested by Tao in his paper \cite{Tao:2000vba}, where he proved $H^{s}_{x}$ local well-posedness for $s > 3/4$ (thus going even below the energy regularity) by working in the \emph{temporal gauge} $A_{0} = 0$, under the restriction that the $H^{s}_{x} \times H^{s-1}_{x}$ norm of $(\Aini_{i}, \Eini_{i})$ is \emph{small}. This gauge has the advantage of being easy to impose globally (as gauge transforms into the temporal gauge can be found by solving an ODE), and thus does not have the problem that the Coulomb gauge possesses. Indeed, it had been used by other authors, including Segal \cite{Segal:1979hg} and Eardley-Moncrief \cite{Eardley:1982fb}, to prove local and global well-posedness of \eqref{eq:hyperbolicYM} for (large) initial data with higher regularity (namely, $s \geq 2$). To reiterate this discussion in our framework, the temporal gauge resolves the issues of (Hyperbolicity) and (Null structure)\footnote{However, the issue of (Null structure) is not addressed fully in the sense that smallness of the initial data is needed.} raised above.

However, this gauge possesses its own drawback in that it fails to cope with initial data sets with a large $H^{s}_{x}$ norm for $3/4 < s \leq 1$\footnote{One reason is that it still relies on a Uhlenbeck-type lemma to set $\rd^{\ell} A_{\ell} = 0$ at $t=0$, which requires some sort of smallness of the initial data. There is also a technical difficulty in the Picard iteration argument which does not allow one to use the smallness of the length of the time-interval; ultimately, this originates from the presence of a time derivative on the right-hand side of the equation $\rd_{t} (\rd^{\ell} A_{\ell}) = - \LieBr{A^{\ell}}{\rd_{t} A_{\ell}}$ (which is equivalent to the equation $\covD^{\ell} F_{\ell 0} = 0$). See \cite{Tao:2000vba} for more details.}. Moreover, another drawback is that it is unclear how to deal with the issue of (Energy estimate), namely how $\nrm{\rd_{x} A_{i}(t)}_{L^{2}_{x}}$ may be controlled for every $t$ using the conserved energy $\bfE$.

\subsection{Main idea of our approach} \label{subsec:intro:mainIdea}
The purpose of this paper, along with the companion paper \cite{Oh:6stz7nRe}, is to present a new approach to the problem of gauge choice which does not possess the drawbacks of the previous methods. As such, this approach does not involve localization in space-time and works well for large initial data. Nevertheless, it is (at the very least) as effective as the previous choices of gauge, as it addresses all of (Hyperbolicity), (Null structure) and (Energy estimate) discussed above. As an application of our approach, we provide in this paper an alternative proof of \emph{finite energy global well-posedness} of \eqref{eq:hyperbolicYM}, relying on the techniques established in the companion paper \cite{Oh:6stz7nRe}.

%This, combined\footnote{We remark, however, that this work will rely on results proved in \cite{Oh:6stz7nRe} other than $H^{1}$ local well-posedness as well. On the other hand, \cite{Oh:6stz7nRe} may be read independently of the present paper.} with the new proof of \emph{$H^{1}_{x}$ local well-posedness} of \eqref{eq:hyperbolicYM} given in \cite{Oh:6stz7nRe}, constitutes an alternative proof of Theorem \ref{thm:intro:KM}.

Heuristically, the key idea of our approach is to {`smooth out'} the problem at hand in a {`geometric fashion'}. The expectation is that the problem of gauge choice for the `smoothed out problem' will be much easier thanks to the additional regularity. All the difficulties, then, are shifted to the problem of controlling the error generated by the smoothing procedure. That this is possible for a certain choice of smoothing procedure, based on a geometric (weakly-)parabolic PDE called the \emph{Yang-Mills heat flow}, is the main thesis of this work and \cite{Oh:6stz7nRe}.

%Roughly speaking, the issues A, B are addressed in the companion paper \cite{Oh:6stz7nRe}, which proves $H^{1}_{x}$ local well-posedness of the Yang-Mills equations as a result, whereas the present paper is concerned with Issue C. 
In the following three subsections (\S \ref{subsec:intro:YMHF}--\S \ref{subsec:overview4GWP}), we will discuss how our approach deals with the issues of (Hyperbolicity), (Null structure) and (Energy estimate). More precisely, after a discussion on the Yang-Mills heat flow in \S \ref{subsec:intro:YMHF}, we will give a summary of the companion paper \cite{Oh:6stz7nRe} in \S \ref{subsec:intro:lwp}, in which we explain how the issues of (Hyperbolicity) and (Null structure) are resolved. Then an overview of the main ideas of the present paper in \S \ref{subsec:overview4GWP} will follow, addressing (Energy estimate). 

\begin{remark} 
The idea of using a smoothing procedure tailored to the particular geometry of the problem has proved fruitful in a number of settings. 
For example, in the well-known works \cite{MR0499948}, \cite{MR850408}, \cite{MR1998349} on extension of the Calder\'on-Zygmund theory to non-Euclidean settings\footnote{For an excellent modern survey and further references on this topic, we refer the reader to the monograph \cite{MR3154530}.}, such an idea underlies the proofs of key square function estimates. 
We also mention the work of Klainerman-Rodnianski \cite{MR2221254}, in which the linear heat equation was used (as in Stein \cite{MR0252961}) to develop a Littlewood-Paley theory for tensors on compact two-dimensional manifolds under very weak regularity hypotheses. This theory was applied in \cite{MR2125732} and \cite{MR2221255} to study the causal geometry of rough solutions to the Einstein equations.

A recent development of this idea that directly motivated our approach is the work of Tao \cite{Tao:2004tm}, in which it was proposed to use a \emph{nonlinear} geometric heat flow to deal with the problem of gauge choice in the context of the energy-critical wave map problem. This approach was put into use in the series of preprints \cite{Tao:2008wn} to develop a large energy theory of wave maps into a hyperbolic space $\bbH^{n}$. In this setting, one begins by solving the associated heat flow, in this case the \emph{harmonic map flow}, starting from a wave map restricted to a fixed $t$-slice. Then the key idea is that the harmonic map flow converges (under appropriate conditions) to a single point, same for every $t$, in the target as the heat parameter goes to $\infty$. For this trivial map at infinity, the canonical choice of gauge is clear; namely, one chooses the same orthonormal frame at each point on the domain. This choice is then parallel-transported back along the harmonic map flow. The resulting gauge is dubbed the \emph{caloric gauge}. This gauge proved to be quite useful, and the use of such gauge has also been successfully extended to the related problem of energy-critical Schr\"odinger maps as well, through the works \cite{Bejenaru:2011wy}, \cite{Smith:2011ef}, \cite{Smith:2010ui}, \cite{Smith:2011ty}, \cite{Dodson:2012uj} and \cite{Dodson:2013vh}. 

\end{remark}

\subsection{The Yang-Mills heat flow} \label{subsec:intro:YMHF}
Before delving into a more detailed exposition of our approach, let us first introduce the \emph{Yang-Mills heat flow} (or (YMHF) in short), which will play an important role. Consider a spatial connection 1-form $A_{i}(s)$ ($i=1,2,3$) on $\bbR^{3}$ parametrized by $s \in [0,s_{0}]$ ($s_{0} > 0$). We say that $A_{i}(s)$ is a \emph{Yang-Mills heat flow} if it satisfies the equation
\begin{equation} \label{eq:YMHF} \tag{YMHF}
	\rd_{s} A_{i} = \covD^{\ell} F_{\ell i} , \quad i=1,2,3.
\end{equation}

First introduced by Donaldson \cite{Donaldson:1985vh}, the Yang-Mills heat flow is the gradient flow for the \emph{Yang-Mills energy} on $\bbR^{3}$ (also referred to as the \emph{magnetic energy})
\begin{equation} \label{eq:Menergy}
\bfB[A_{i}] := \frac{1}{2} \sum_{1 \leq i < j \leq 3} \nrm{F_{ij}}_{L^{2}_{x}}^{2}
\end{equation}
and plays an important role in differential geometry. It has been a subject of an extensive research by itself; see, e.g., \cite{Donaldson:1985vh}, \cite{Rade:1992tu}, \cite{Struwe:1994is}, \cite{Charalambous:2010vt}. 

Our intention is to use \eqref{eq:YMHF} as a geometric smoothing device for \eqref{eq:hyperbolicYM}. One must be careful, however, since \eqref{eq:YMHF} is \emph{not} strictly parabolic as it stands at the level of $A_{i}$. Indeed, expanding \eqref{eq:YMHF} in terms of $A_{i}$, the top order terms look like
\begin{equation*}
	\rd_{s} A_{i} = \lap A_{i} - \rd^{\ell} \rd_{i} A_{\ell} + (\hbox{lower order terms}),
\end{equation*}
where $\lap A_{i} - \rd_{i} \rd^{\ell} A_{\ell}$ possesses non-trivial kernel (any $A_{i} = \rd_{i} \phi$, for $\phi$ a $\LieAlg$-valued function). Due to this fact, the Yang-Mills heat flow is said to be only \emph{weakly-parabolic}.

The culprit of the non-parabolicity of \eqref{eq:YMHF} turns out to be the gauge covariance of the term $\covD^{\ell} F_{\ell i}$, which suggests that it can be remedied by studying the gauge structure of the Yang-Mills heat flow in detail. Upon inspection, we see that the gauge structure of the equations \eqref{eq:YMHF} is somewhat restrained, as it is covariant only under gauge transforms that are \emph{independent} of $s$. To deal with the problem of non-parabolicity, we will begin by fixing this issue, i.e., reformulating the Yang-Mills heat flow in a way that is covariant under gauge transforms which may also depend on the $s$-variable.

Along with $A_{i}$, let us also add a component $A_{s}$ and consider $A_{a}$ ($a = x^{1}, x^{2}, x^{3}, s$), which is a connection 1-form on the product manifold $\bbR^{3} \times [0,s_{0}]$. Corresponding to $A_{s}$, we also introduce the \emph{covariant derivative} along the $\rd_{s}$-direction
\begin{equation*}
	\covD_{s} := \rd_{s} + \LieBr{A_{s}}{\cdot}.
\end{equation*}

A \emph{covariant Yang-Mills heat flow} is a solution $A_{a}$ to the following system of equations.
\begin{equation} \label{eq:cYMHF} \tag{cYMHF}
	F_{si} = \covD^{\ell} F_{\ell i}, \quad i = 1,2,3,
\end{equation}
where $F_{si}$ is the commutator between $\covD_{s}$ and $\covD_{i}$, given by the formula
\begin{equation} \label{eq:intro:Fsi}
	F_{si} = \rd_{s} A_{i} - \rd_{i} A_{s} + \LieBr{A_{s}}{A_{i}}.
\end{equation}

The system \eqref{eq:cYMHF} is underdetermined for $A_{a}$, and therefore requires an additional gauge condition (typically on $A_{s}$) in order to be solved. Note that the original Yang-Mills heat flow \eqref{eq:YMHF} is a special case of \eqref{eq:cYMHF}, namely when $A_{s} = 0$. On the other hand, choosing $A_{s} = \rd^{\ell} A_{\ell}$, the top order terms of \eqref{eq:cYMHF} becomes
\begin{equation*}
	\rd_{s} A_{i} - \rd_{i} \rd^{\ell} A_{\ell} = \lap A_{i} - \rd^{\ell} \rd_{i}  A_{\ell} +(\hbox{lower order terms}).
\end{equation*}

The term $\rd^{\ell} \rd_{i}  A_{\ell}$ on each side are cancelled, and we are consequently left with a \emph{strictly} parabolic system of equations for $A_{i}$. In other words, the weakly-parabolic system \eqref{eq:YMHF} is equivalent to a strictly parabolic system of equations, connected via gauge transforms for \eqref{eq:cYMHF}.

Henceforth, the gauge condition $A_{s} = 0$ will be referred to as the \emph{caloric gauge}, in deference to the term introduced by Tao in his work \cite{Tao:2004tm}. The condition $A_{s} = \rd^{\ell} A_{\ell}$ will be dubbed the \emph{DeTurck gauge}, as the procedure outlined above may be viewed as a geometric reformulation of the standard DeTurck trick, introduced first by DeTurck \cite{DeTurck:1983ts} in the context of the Ricci flow and adapted to the Yang-Mills heat flow by Donaldson \cite{Donaldson:1985vh}.

\subsection{Overview of \cite{Oh:6stz7nRe}: Proof of local-wellposedness} \label{subsec:intro:lwp}
Acquainted with the covariant formulation of the Yang-Mills heat flow, we are ready to return to the task of describing our approach in more detail. We will begin by providing a short overview of \cite{Oh:6stz7nRe}, in which local well-posedness is proved for initial data sets with $\dot{H}^{1}_{x}$ regularity; for a more precise statement, see Theorem \ref{thm:lwp4YM}. In particular, we will explain how the issues of (Hyperbolicity) and (Null structure) raised in \S \ref{subsec:intro:motivation} are resolved in our approach.
%Consider the problem of proving Part (1) of Theorem \ref{thm:intro:KM}, which is a local well-posedness statement for initial data in $\dot{H}^{1}_{x} \times L^{2}_{x}$. 

To avoid too many technical details, we will treat here the simpler problem of proving an {\it a priori} bound of a solution to \eqref{eq:hyperbolicYM} in the temporal gauge. That is, for some interval $I := (-T_{0}, T_{0}) \subset \bbR$, we will presuppose the existence of a solution $\Atemp_{\mu}$ to \eqref{eq:hyperbolicYM} in the temporal gauge on $I \times \bbR^{3}$ and aim to establish an estimate of the form
\begin{equation*}
	\nrm{\rd_{t,x} \Atemp_{\mu}}_{C_{t} (I, L^{2}_{x})}  \leq C \sum_{i=1,2,3} \nrm{(\Aini_{i}, \Eini_{i})}_{\dot{H}^{1}_{x} \times L^{2}_{x}}
\end{equation*}
where $\Atemp_{i} (t=0) = \Aini_{i}$, $\rd_{t} \Atemp_{i} (t=0) = \Eini_{i}$. The dagger signifies that $\Atemp_{\mu}$ is a representative in the temporal gauge.

\subsubsection*{Geometric smoothing of $\Atemp_{\mu}$ by the (dynamic) Yang-Mills heat flow}
The first step of the proof is to smooth out the solution $\Atemp_{\mu}$, essentially using the covariant Yang-Mills heat flow. Let us introduce a new variable $s \in [0,s_{0}]$, and extend $\Atemp_{\mu} = \Atemp_{\mu}(t,x)$ to $A_{\bfa} = A_{\bfa} (t,x,s)$ (where $\bfa = t, x^{1}, x^{2}, x^{3}, s$) on $I \times \bbR^{3} \times [0,s_{0}]$ by solving the equations
\begin{equation} \label{eq:dYMHF} \tag{dYMHF}
	F_{s \mu} = \covD^{\ell} F_{\ell \mu}, \quad \mu = 0,1,2,3
\end{equation}
with an appropriate choice of $A_{s}$, starting with $A_{\mu} (s=0) = \Atemp_{\mu}$. Note that this system is \eqref{eq:cYMHF} appended with the equation $F_{s0} = \covD^{\ell} F_{\ell 0}$ for $A_{0}$; it will be referred to as the \emph{dynamic Yang-Mills heat flow} or, in short, (dYMHF). Using Picard iteration, these equations can be solved provided that $s_{0} > 0$ is small enough.

\subsubsection*{The hyperbolic-parabolic-Yang-Mills system}
As a result, we arrive at a connection 1-form $A_{\bfa}$ on $I \times \bbR^{3} \times [0,s_{0}]$ which solves the following system of equations.
\begin{equation} \label{eq:HPYM} \tag{HPYM}
\left\{
\begin{aligned}
	F_{s \mu} &= \covD^{\ell} F_{\ell \mu} \hspace{.25in} \hbox{ on } \hspace{.1in} I \times \bbR^{3} \times [0,s_{0}], \\
	\covD^{\mu} F_{\mu \nu} &= 0 \hspace{.5in} \hbox{ along } I \times \bbR^{3} \times \set{0}.
\end{aligned}
\right.
\end{equation}

We will refer to this as the \emph{Hyperbolic-Parabolic-Yang-Mills system} or, in short, \eqref{eq:HPYM}. This will be the system of equations that we will mainly work with in place of \eqref{eq:hyperbolicYM}. Accordingly, instead of $\Atemp_{\mu}$, we will estimate $\Alow_{\mu} := A_{\mu}(s=s_{0})$, which may be viewed as a smoothed-out version of $\Atemp_{\mu}$, and the error $\rd_{s} A_{\mu}(s)$ (for $s \in (0, s_{0})$) in between. 

\subsubsection*{Gauge choices for \eqref{eq:HPYM}: DeTurck and caloric-temporal gauges}
The next step consists of estimating $\rd_{s} A_{\mu}$ and $\Alow_{\mu}$ by using the equations arising from \eqref{eq:HPYM}. Basically, the strategy is to first use the parabolic (in the $s$-direction) equations to estimate the new variables $\rd_{s} A_{\mu}, \Alow_{\mu}$ at $t=0$, and then to use the hyperbolic (in the $t$-directions) equations to estimate their evolution in $t$. As \eqref{eq:HPYM} is manifestly gauge covariant (under gauge transforms fully dependent on all the variables $x^{0}, x^{1}, x^{2}, x^{3}, s$), we need to fix a gauge in order to carry out such analyses. 

As it turns out, a different gauge choice is needed to achieve each goal. For the purpose of deriving estimates at $t=0$, which should be compatible with the analysis of the $t$-evolution to follow, it is essential to exploit the smoothing property of \eqref{eq:dYMHF}. As such, the gauge of choice here is the \emph{DeTurck gauge} $A_{s} = \rd^{\ell} A_{\ell}$. On the other hand, completely different considerations are required for estimating the $t$-evolution, and here the gauge condition we impose is
\begin{equation} \label{eq:caloricTemporal}
\left\{
\begin{aligned}
	&A_{s} = 0 \quad \hbox{ on } I \times \bbR^{3} \times (0,s_{0}), \\
	&\Alow_{0} = 0 \quad \hbox{ on } I \times \bbR^{3} \times \set{s_{0}}.
\end{aligned}
\right.
\end{equation}
which will be referred to as the \emph{caloric-temporal gauge}. In practice, the DeTurck gauge will be first used to obtain estimates at $t=0$, and then we will perform a gauge transformation\footnote{A technical remark: Performing a gauge transformation $U = U(t,x,s)$ from the DeTurck gauge to the caloric gauge, with the additional condition that $U(t=0, s=0) = \mathrm{Id}$, corresponds exactly to carrying out the standard DeTurck trick \cite{Donaldson:1985vh}. However, this is inappropriate for our purposes, as it turns out that this gauge transform is not bounded on $H^{m}_{x}$ for $m > 1$; as such, it cannot retain the smoothing estimates proved in the DeTurck gauge. Instead, we will use the gauge transform for which $U(t=0, s=1) = \mathrm{Id}$. Under such gauge transform, $\Alow_{i}(t=0)$ remains the same, and thus smooth, at the cost of introducing a non-trivial gauge transform for the initial data at $t=0, s=0$. In some sense, this procedure is an analogue of the Uhlenbeck's lemma in our approach.} into the caloric-temporal gauge to carry out the analysis of the evolution in $t$. We remark that finding such gauge transform is always possible, as it amounts to simply solving a hierarchy of linear ODEs.

A brief discussion on the motivation behind our choice of the \emph{caloric-temporal gauge} is in order. For $\rd_{s} A_{\mu}$ on $I \times \bbR^{3} \times (0, s_{0})$, let us begin by considering the following rearrangement of the formula \eqref{eq:intro:Fsi} for $F_{si}$:
\begin{equation} \label{eq:intro:covHodge}
	\rd_{s} A_{i} = F_{si} + \covD_{i} A_{s}.
\end{equation}

A simple computation\footnote{The identity $\covD^{\ell} F_{s \ell} = 0$ follows from \eqref{eq:cYMHF} and $\covD^{\ell} \covD^{k} F_{\ell k} = 0$, which is proved simply by anti-symmetrizing the indices $\ell, k$.} shows that $F_{si}$ is \emph{covariant-divergence-free}, i.e., $\covD^{\ell} F_{s \ell} = 0$. This suggests that \eqref{eq:intro:covHodge} may be viewed (heuristically) as a \emph{covariant Hodge decomposition} of $\rd_{s} A_{i}$, where $F_{si}$ is the covariant-divergence-free part and $\covD_{i} A_{s}$, being a pure covariant-gradient term, may be regarded as the `covariant-curl-free part' (although, strictly speaking, the covariant-curl does not vanish but is only of lower order for this term). Recall that the Coulomb gauge condition, which had a plenty of good properties as discussed earlier, is equivalent to having zero curl-free part. Therefore, to imitate the Coulomb gauge as closely as possible, we are motivated to set $A_{s} = 0$ on $I \times \bbR^{3} \times (0,s_{0})$; incidentally, this turns out to be the \emph{caloric} gauge condition discussed earlier. 

On the other hand, at $s=s_{0}$, the idea is that $\Alow_{\mu}$ possesses \emph{smooth} initial data $(\Alow_{i}, \Flow_{0i})(t=0)$, thanks to the smoothing property of \eqref{eq:dYMHF}. Therefore, we expect that the problem of gauge choice for $\Alow_{\mu}$ is not as delicate as the original problem; as such, we choose the temporal gauge condition $\Alow_{0} = 0$, which is easy to impose yet sufficient for the analogous problem with smoother initial data, as the previous works \cite{Segal:1979hg}, \cite{Eardley:1982fb} had shown.

\subsubsection*{Resolution of the issues of (Hyperbolicity) and (Null structure)}
With the caloric-temporal gauge, we are finally ready to describe how the issues of (Hyperbolicity) and (Null structure) are resolved in our approach. Let us begin by introducing the \emph{Yang-Mills tension field} $w_{\nu} (s):= \covD^{\mu} F_{\nu \mu}(s)$, which measures the extent of failure of $A_{\mu}(s)$ to satisfy \eqref{eq:hyperbolicYM}. Then we may derive the following system of equations \cite[Appendix A]{Oh:6stz7nRe}:
\begin{align}
	\covD_s w_\nu 
	= & \covD^\ell \covD_\ell w_\nu + 2 \LieBr{\tensor{F}{_\nu^\ell}}{w_\ell} + 2 \LieBr{F^{\mu \ell}}{\covD_{\mu} F_{\nu \ell} + \covD_{\ell} F_{\nu \mu}}, 
	\label{eq:covParabolic4w} \\ 
	 \covD^\mu \covD_\mu F_{s i} 
 	= & 2 \LieBr{\tensor{F}{_s^\mu}}{F_{i \mu}} - 2 \LieBr{F^{\mu \ell}}{\covD_\mu F_{i \ell} + \covD_\ell F_{i \mu}} 
		- \covD^\ell \covD_\ell w_i + \covD_i \covD^\ell w_\ell - 2 \LieBr{\tensor{F}{_i^\ell}}{w_\ell},  
	\label{eq:hyperbolic4F} \\
	\covDlow^{\mu}\Flow_{\nu \mu} = & \wlow_{\nu}. 
	\label{eq:hyperbolic4Alow}
\end{align}

The underlines of \eqref{eq:hyperbolic4Alow} signify that each variable is restricted to $\set{s=s_{0}}$. Furthermore, $w_{\nu} \equiv 0$ at $s=0$, for all $\nu = 0,1,2,3$.

The parabolic equation \eqref{eq:covParabolic4w} can be used to derive estimates for the Yang-Mills tension field $w_{\mu}$. It is important to note that its data at $s=0$ is \emph{zero}, thanks to the fact that $A_{\mu}(s=0)$ satisfies \eqref{eq:hyperbolicYM}. Moreover, note that $w_{0} = -F_{s0}$, which is equal to $- \rd_{s} A_{0}$ thanks to the caloric gauge condition $A_{s} = 0$. In conclusion, after solving the parabolic equation \eqref{eq:covParabolic4w}, the dynamics of \eqref{eq:HPYM} is reduced to that of the variables $F_{si} = \rd_{s} A_{i}$ (again due to $A_{s} = 0$) and $\Alow_{i}$. These are, in turn, estimated by \eqref{eq:hyperbolic4F}, which is a wave equation for $F_{si}$, and \eqref{eq:hyperbolic4Alow}, which is the Yang-Mills equation with a source $\wlow_{\nu}$ for $\Alow_{\mu}$ under the temporal gauge $\Alow_{0} = 0$. This shows the hyperbolicity of \eqref{eq:hyperbolicYM}, which takes care of (Hyperbolicity). 

Next, let us address the issue of exhibiting null forms, i.e., (Null structure). Let us begin by observing that for \eqref{eq:hyperbolic4Alow}, \emph{no null form is needed} to close the estimates; this is because $(\Alow_{i}, \Flow_{0i})(t=0)$ has been smoothed out by \eqref{eq:dYMHF} as mentioned earlier. For \eqref{eq:hyperbolic4F}, on the other hand, there turns out to be a single term which cannot be dealt with simply by Strichartz estimates, namely
\begin{equation*}
	2\LieBr{A_{\ell} - \Alow_{\ell}}{\rd^{\ell} F_{si}}.
\end{equation*}

If $A_{\ell} - \Alow_{\ell}$ were divergence-free, i.e., $\rd^{\ell} (A_{\ell} - \Alow_{\ell}) = 0$, then an argument of Klainerman-Machedon \cite{Klainerman:1994jb}, \cite{Klainerman:1995hz} would show that this nonlinearity may be rewritten as a linear combination of null forms $Q_{jk}(\abs{\rd_{x}}^{-1} (A- \Alow), F_{si})$. Although this is not strictly true, we have 
\begin{equation*}
A_{\ell} - \Alow_{\ell} = - \int_{0}^{s_{0}} F_{s\ell} (s) \, \ud s
\end{equation*}
thanks to the condition $A_{s} = 0$, where $F_{s\ell}$ is \emph{covariant-divergence-free}, i.e., $\covD^{\ell} F_{s \ell} = 0$. This suffices for a variant of the argument of Klainerman-Machedon to work, settling the issue of (Null structure)\footnote{¥}.

Provided that $s_{0}, \abs{I}$ are sufficiently small\footnote{In the Coulomb gauge, the equation for $A_{0}$ is elliptic and therefore smallness of the time interval $I$ cannot be utilized to solve for $A_{0}$ using perturbation; in \cite{Klainerman:1995hz}, the authors exploit the spatial localization to overcome this issue. For \eqref{eq:HPYM} in the caloric-temporal gauge, $A_{0}$ estimated by integrating $F_{s0} = \rd_{s} A_{0}$, where the latter variable satisfies a parabolic equation. For this, smallness of $s_{0}$ can be used, and thus the estimates are still global on $\bbR^{3}$.}, an analysis of \eqref{eq:HPYM} using the gauge conditions indicated above leads to estimates for $\rd_{s} A_{i}, \Alow_{i}$ in the caloric-temporal gauge, such as
\begin{equation} \label{eq:intro:est4HPYM}
\left\{
\begin{aligned}
	\sup_{0 < s < s_{0}} s^{(m+1)/2} \nrm{\rd_{x}^{(m-1)} \rd_{t,x} (\rd_{s} A_{i})(s)}_{C_{t} (I, L^{2}_{x})}
	 \leq  C_{m} \sum_{j=1,2,3} \nrm{(\Aini_{j}, \Eini_{j})}_{\dot{H}^{1}_{x} \times L^{2}_{x}} \\
	\sup_{t \in I} \bb( \int_{0}^{s_{0}} s^{m+1} \nrm{\rd_{x}^{(m-1)} \rd_{t,x} (\rd_{s} A_{i})(t, s)}_{L^{2}_{x}}^{2} \, \frac{\ud s}{s} \bb)^{1/2}
	 \leq  C_{m} \sum_{j=1,2,3} \nrm{(\Aini_{j}, \Eini_{j})}_{\dot{H}^{1}_{x} \times L^{2}_{x}} \\
	s_{0}^{(k-1)/2} \nrm{\rd_{x}^{(k-1)} \rd_{t,x} \Alow_{i}}_{C_{t} (I, L^{2}_{x})} 
	 \leq  C_{k} \sum_{j=1,2,3} \nrm{(\Aini_{j}, \Eini_{j})}_{\dot{H}^{1}_{x} \times L^{2}_{x}}
\end{aligned}
\right.
\end{equation}
up to some integers $m_{0}, k_{0} > 1$, i.e., $1 \leq m \leq m_{0}$, $1 \leq k \leq k_{0}$. We remark that the weights of $s$ are dictated by scaling. 

\subsubsection*{Returning to $\Atemp_{\mu}$} The only remaining step is to translate \eqref{eq:intro:est4HPYM} to the desired estimate for $\nrm{\rd_{t,x} \Atemp_{\mu}}_{C_{t} (I, L^{2}_{x})}$. The first issue arising in this step is that the naive approach of integrating the estimates \eqref{eq:intro:est4HPYM} in $s$ \emph{fails} to bound $\nrm{\rd_{t,x} \Atemp_{\mu}}_{C_{t} (I, L^{2}_{x})}$, albeit only by a logarithm. To resolve this issue, we take the weakly-parabolic equations
\begin{equation*}
	\rd_{s} A_{i} = \lap A_{i} - \rd^{\ell} \rd_{i} A_{\ell} + (\hbox{lower order terms}).
\end{equation*}
differentiate by $\rd_{t,x}$, multiply by $\rd_{t,x} A_{i}$ and then integrate the highest order terms by parts over $\bbR^{3} \times [0, s_{0}]$. This procedure, combined with the $L^{2}_{s}$-type estimates of \eqref{eq:intro:est4HPYM}, overcome the logarithmic divergence.

Another issue is that \eqref{eq:intro:est4HPYM}, being in the caloric-temporal gauge, is in a different gauge from the temporal gauge along $s=0$. Therefore, we are required to control the gauge transform back to the temporal gauge along $s=0$, for which appropriate estimates for $A_{0}(s=0)$ in the caloric-temporal gauge are needed. These are obtained ultimately as a consequence of the analysis of the hyperbolic equations of \eqref{eq:HPYM} (Strichartz estimates, in particular, are used). 

\subsection{Overview of the present paper: Finite energy global well-posedness} \label{subsec:overview4GWP}
In the work of Klainerman-Machedon \cite{Klainerman:1995hz}, as pointed out earlier, finite energy global well-posedness was an easy corollary of the $H^{1}_{x}$ local well-posedness proof, thanks to the fact that in the (local) Coulomb gauge, the conserved energy $\bfE(t)$ controls $\nrm{(A_{i}, F_{0i})(t)}_{H^{1}_{x} \times L^{2}_{x}}$. However, in the temporal gauge, making use of the conserved energy $\bfE(t)$ is not as straightforward since $\bfE(t)$ only controls certain components (namely, the curl) of the full gradient of $A_{i}(t)$. We remind the reader that this was referred to as the issue of (Energy estimate) in \S \ref{subsec:intro:motivation}.

In this paper, we address this issue by using again the caloric-temporal gauge. Instead of showing that the energy $\bfE(t)$ controls $\nrm{(A_{i}, F_{0i})(t)}_{H^{1}_{x} \times L^{2}_{x}}$ for every $t$, which is likely to be false in our gauge, we will prove a slightly weaker statement, namely that $\bfE(t)$ controls the \emph{growth} of $\nrm{(A_{i}, F_{0i})(t)}_{H^{1}_{x} \times L^{2}_{x}}$ as $\abs{t} \to \pm \infty$; in particular, $\nrm{(A_{i}, F_{0i})(t)}_{H^{1}_{x} \times L^{2}_{x}} < \infty$ for every $t \in \bbR$. Our proof of finite energy global well-posedness will proceed roughly in three steps, each of which uses the conserved energy $\bfE(t)$ in a crucial way.

%Nevertheless, it is another remarkable property of our approach that the issue of (Energy estimate) can also be resolved, and therefore finite energy global well-posedness of \eqref{eq:hyperbolicYM} can be proved. 
\vskip.5em
\noindent {\it Step 1.} We start with a solution $\Atemp_{\mu}$ to \eqref{eq:hyperbolicYM} in the temporal gauge on $(-T_{0}, T_{0}) \times \bbR^{3}$. As in the proof of local well-posedness, the first step is to solve \eqref{eq:dYMHF} to extend $\Atemp_{\mu}$ to a solution $A_{\bfa}$ to \eqref{eq:HPYM}. {\it A priori}, however, it is not clear whether it is possible to solve \eqref{eq:dYMHF} on a uniform $s$-interval when $T_{0}$ is large.

To illustrate, suppose that $\Atemp_{\mu}$ does not extend past the time $T_{0}$. Then from the local well-posedness statement, it is necessary that
\begin{equation*}
\nrm{\rd_{t,x} \Atemp_{\mu}(t)}_{L^{2}_{x}} \to \infty \quad \hbox{ as } t \to T_{0}.
\end{equation*}

Because of this, the size of the $s$-interval on which \eqref{eq:dYMHF} can be solved by perturbative methods shrinks as $t \to T_{0}$. As a consequence, there might not exist a non-trivial interval $[0,s_{0}]$ on which \eqref{eq:dYMHF} can be solved for every $t \in (-T_{0}, T_{0})$.

However, such a scenario is ruled out thanks to the conserved energy $\bfE(t)$, and \eqref{eq:dYMHF} can be solved in a uniform manner globally in time. More precisely, it is possible to show that there exists $s_{0} > 0$ depending only on $\bfE(t)$ such that \eqref{eq:dYMHF} on a fixed $t$-slice can be solved on an interval $[0,s_{0}]$. As $\bfE(t)$ is conserved, this shows that $\Atemp_{\mu}$ can be extended to a solution $A_{\bfa}$ to \eqref{eq:HPYM} on $(-T_{0}, T_{0}) \times \bbR^{3} \times [0,s_{0}]$.

\vskip.5em
\noindent {\it Step 2.} With a solution $A_{\bfa}$ of \eqref{eq:HPYM} in hand, we impose the caloric-temporal gauge condition via an appropriate gauge transform. We wish to demonstrate that the conserved energy\footnote{For a solution $A_{\bfa}$ to \eqref{eq:HPYM}, $\bfE(t)$ is defined to be the conserved energy of $A_{\mu}$ at $(t, s=0)$. Note that this quantity is gauge-invariant.} $\bfE(t)$ controls the appropriate fixed-time norms of the dynamic variables, which in this case are $\Alow_{i}$ and $F_{si} = \rd_{s} A_{i}$. 

The key observation is that $\nrm{\covD_{x}^{(k)} F_{\mu \nu}(t, s)}_{L^{2}_{x}}$ is estimated (with an appropriate weight of $s$) by $\bfE(t)$, thanks to covariant parabolic estimates. In particular, $\nrm{\covD_{x}^{(k)} \Flow_{0i}(t)}_{L^{2}_{x}}$ is under control, where $\Flow_{\mu \nu}$ is the connection 2-form restricted to $\set{s=s_{0}}$. As the temporal gauge condition $\Alow_{0} = 0$ holds, we have $\Flow_{0i} = \rd_{t} \Alow_{i}$; therefore, the preceding norm may be integrated in $t$ to control the growth of $\nrm{\rd_{x}^{(k)} \Alow_{i}(t)}_{L^{2}_{x}}$ for $t \in (-T_{0}, T_{0})$. On the other hand, as $F_{si} = \covD^{\ell} F_{\ell i}$ is already of the form $\covD_{x} F_{\mu \nu}$, we can use the conserved energy $\bfE(t)$ to control the appropriate (fixed-time) norms of $F_{si}(t)$ as well, for each $t \in (-T_{0}, T_{0})$. Since $F_{si} = \rd_{s} A_{i}$ thanks to the caloric gauge condition $A_{s} = 0$, we remark that these two estimates can be put together to prove a controlled growth of $\nrm{(A_{i}, F_{0i})(t)}_{H^{1}_{x} \times L^{2}_{x}}$.

\vskip.5em
\noindent{\it Step 3.} Finally, we must unwind all the gauge transformations and return to the solution $\Atemp_{\mu}$ in the temporal gauge. As in the last step of the proof of local well-posedness, this requires estimating $A_{0}$ along $s=0$ in the caloric-temporal gauge, where an important ingredient for the latter is the estimates obtained from the hyperbolic equations of \eqref{eq:HPYM}. Iterating the techniques developed in \cite{Oh:6stz7nRe} for proving local well-posedness on a short time interval, coupled with some new estimates arising from the conserved energy $\bfE$, leads to the desired estimates.
\vskip.5em
For a more rigorous overview of the whole argument of the present paper, we refer the reader to Section \ref{sec:reduction}. There, the Main Theorem is reduced to Theorems \ref{thm:idEst}--\ref{thm:dynEst}, which essentially correspond to Steps 1--3 in the respective order.

\subsection{Other applications of the Yang-Mills heat flow approach} \label{subsec:intro:apps}
In this subsection, we briefly comment on possible future applications of our approach.

\vskip.5em
\noindent{\it Further applications to the Yang-Mills equations.}
As no space-time localizations are involved, our approach has the advantage that Fourier-analytic techniques can be used relatively easily in the large data setting. We remark also that there is no conceptual difficulty in extending our approach to dimensions other than (1+3). 

In a forthcoming work \cite{Oh:uq}, we apply the Yang-Mills heat flow approach to prove local well-posedness of the Yang-Mills equations on $\bbR^{1+4}$ with initial data $(\Aini_{i}, \Eini_{i}) \in H^{1+\eps}_{x} \times H^{\eps}_{x}$ of any size, for any $\eps > 0$. Note that this is almost optimal, in the sense that the scaling critical Sobolev regularity for this problem is $H^{1}_{x} \times L^{2}_{x}$ (energy-critical). We use the Fourier-analytic tools developed in Klainerman-Tataru \cite{Klainerman:1999do}, which were used for establishing the small data case in the Coulomb gauge. We also use the observation that not only $\LieBr{A^{\ell} - \Alow^{\ell}}{\rd_{\ell} F_{si}}$ (as discussed in \S \ref{subsec:intro:lwp}), but all quadratic terms in the wave equation for $F_{si}$ have null structure in the caloric-temporal gauge, modulo cubic and higher terms; see \cite[Section 2.2]{thesis}. We believe that the same method will apply to the Yang-Mills equations on $\bbR^{1+3}$ with initial data $(\Aini_{i}, \Eini_{i}) \in H^{3/4+\eps}_{x} \times H^{-1/4+\eps}_{x}$  for any $\eps > 0$, which will generalize the results of Tao \cite{Tao:2000vba} to large data.

The result of \cite{Oh:uq} may be regarded as the first step towards a long time, large data theory of the energy-critical Yang-Mills equations on $\bbR^{1+4}$, which is a major open problem at the moment. It is interesting to note that the analysis in the companion paper \cite{Oh:6stz7nRe} leading to \emph{local} well-posedness relies only on the \emph{local} theory of the Yang-Mills heat flow, whereas we use the \emph{global} theory of the Yang-Mills heat flow in the present paper (like those in R{\aa}de \cite{Rade:1992tu}) in order to prove \emph{global} well-posedness of the Yang-Mills equations. This suggests that in order to study the long time behavior of the energy-critical Yang-Mills equations using our approach, we would need to draw more from the global theory of the energy-critical Yang-Mills heat flow, which have been developed extensively in geometric analysis; see, e.g., \cite{Donaldson:1985vh}, \cite{Struwe:1994is}. Note that in the case of energy-critical wave map equations on $\bbR^{1+2}$, the use of the caloric gauge \cite{Tao:2008wn} relied on the global theory of the harmonic map heat flow of Eells-Sampson \cite{ES}.

In view of the geometric nature of our approach, it is also possible that the Main Theorem of this paper be extended to a more general class of backgrounds, namely, (1+3)-dimensional globally hyperbolic Lorentzian manifolds. For this purpose, a recent geometric bilinear estimates proved in \cite{Klainerman:2005kj} may be useful. We note that for the Yang-Mills equations on a general (1+3)-dimensional smooth, globally hyperbolic Lorentzian manifold, global existence for sufficiently smooth initial data was proved in Chru\'sciel-Shatah \cite{MR1604914}, by adapting the arguments in Eardley-Moncrief \cite{Eardley:1982fb}. 

\vskip.5em
\noindent{\it Applications to other non-abelian gauge field theories.}
The dynamic Yang-Mills heat flow $F_{s \mu} = \covD^{\ell} F_{\ell \mu}$ provides a natural means to smooth out any space-time connection 1-form $A_{\mu}$ while keeping the gauge structure, whether $A_{\mu}$ solves the Yang-Mills equations $\covD^{\mu} F_{\nu \mu} = 0$ or not. Furthermore, the caloric-temporal gauge condition \eqref{eq:caloricTemporal} makes sense for any solution to the dynamic Yang-Mills heat flow, independent of the Yang-Mills equations on $\set{s = 0}$. These considerations suggest that \eqref{eq:caloricTemporal} is a general gauge condition, like the temporal and Coulomb gauges, which can be imposed for any non-abelian gauge field theory. In view of the desirable properties of the caloric-temporal gauge exhibits for the Yang-Mills equations (as discussed in \S \ref{subsec:intro:lwp} and \S \ref{subsec:overview4GWP}), we expect that the caloric-temporal gauge would be useful in this broader context as well. 

\subsection{Statement of the Main Theorem} \label{subsec:intro:mainThm}
We now give the precise statement of our main result. Let us begin by defining the class of initial data sets of interest.
\begin{definition}[Admissible $H^{1}$ initial data set] \label{def:admID}
We say that a pair $(\Aini_{i}, \Eini_{i})$ of 1-forms on $\bbR^{3}$ is an \emph{admissible $H^{1}$ initial data set} for the Yang-Mills equations if the following conditions hold:
\begin{enumerate}
\item $\Aini_{i} \in \dot{H}_{x}^{1} \cap L^{3}_{x}$ and $\Eini_{i} \in L^{2}$,
\item The \emph{constraint equation}
\begin{equation*}
	\rd^{\ell} \Eini_{\ell} + \LieBr{\Aini^{\ell}}{\Eini_{\ell}} = 0,
\end{equation*}
	holds in the distributional sense.
\end{enumerate}
\end{definition}

Let us also define the notion of \emph{admissible solutions.}
\begin{definition}[Admissible solutions] \label{def:admSol}
Let $I \subset \bbR$. We say that a generalized solution $A_{\mu}$ to the Yang-Mills equations \eqref{eq:hyperbolicYM} in the temporal gauge $A_{0} = 0$ defined on $I \times \bbR^{3}$ is \emph{admissible} if 
\begin{equation*}
	A_{\mu} \in C_{t}(I, \dot{H}^{1}_{x} \cap L^{3}_{x}), \quad \rd_{t} A_{\mu} \in C_{t}(I, L^{2}_{x})
\end{equation*}
and $A_{\mu}$ can be approximated by classical solutions in the temporal gauge in the above topology.
\end{definition}

We begin with a $H^{1}_{x}$ local well-posedness theorem, whose proof using the Yang-Mills heat flow has been given in the companion paper \cite{Oh:6stz7nRe}.
\begin{theorem}[$H^{1}_{x}$ local well-posedness {\cite[Main Theorem]{Oh:6stz7nRe}}] \label{thm:lwp4YM}
Let $(\Aini_{i}, \Eini_{i})$ be an admissible $H^{1}$ initial data set, and define $\calIini := \nrm{\Aini}_{\dot{H}^{1}_{x}} + \nrm{\Eini}_{L^{2}_{x}}$. Consider the initial value problem (IVP) for \eqref{eq:hyperbolicYM} with $(\Aini_{i}, \Eini_{i})$ as the initial data.
\begin{enumerate}
\item There exists $T^{\star} = T^{\star}(\calIini)> 0$, which is non-increasing in $\calIini$, such that a unique admissible solution $A_{\mu} = A_{\mu} (t,x)$ to the IVP in the temporal gauge $A_{0} = 0$ exists on $(-T^{\star}, T^{\star})$. Furthermore, the following estimates hold.
\begin{equation*} 
	\nrm{\rd_{t,x} A}_{C_{t} ((-T^{\star}, T^{\star}),L^{2}_{x})} \leq C_{\calIini} \calIini, \quad 
	\nrm{A}_{C_{t} ((-T^{\star}, T^{\star}), L^{3}_{x})} \leq \nrm{\Aini}_{L^{3}_{x}} + (T^{\star})^{1/2} C_{\calIini} \calIini.
\end{equation*}
\item Let $(\Aini'_{i}, \Eini'_{i})$ be another admissible $H^{1}$ initial data set such that $\nrm{\Aini'}_{\dot{H}^{1}_{x}} + \nrm{\Eini'}_{L^{2}_{x}} \leq \calIini$, and let $A'_{\mu}$ be the corresponding solution given by (1). Then the following estimates for the difference hold.
\begin{align*} 
	\nrm{\rd_{t,x} A - \rd_{t,x} A'}_{C_{t} ((-T^{\star}, T^{\star}),L^{2}_{x})} \leq &
	C_{\calIini} (\nrm{\Aini - \Aini'}_{\dot{H}^{1}_{x}} + \nrm{\Eini - \Eini'}_{L^{2}_{x}} ). \\
	\nrm{A- A'}_{C_{t} ((-T^{\star}, T^{\star}), L^{3}_{x})} \leq &
	\nrm{\Aini - \Aini'}_{L^{3}_{x}} \\ &
	+ (T^{\star})^{1/2} C_{\calIini} (\nrm{\Aini - \Aini'}_{\dot{H}^{1}_{x}} +\nrm{\Eini - \Eini'}_{L^{2}_{x}} ).
\end{align*}

\item Finally, the following version of \emph{persistence of regularity} holds: if $\rd_{x} \Aini_{i}, \Eini_{i} \in H^{m}_{x}$ for an integer $m \geq 0$, then the corresponding solution given by (1) satisfies
\begin{equation*}
	\rd_{t,x} A_{i} \in C^{k_{1}}_{t} ((-T^{\star}, T^{\star}), H^{k_{2}}_{x})
\end{equation*}
for every pair $(k_{1}, k_{2})$ of nonnegative integers such that $k_{1} + k_{2} \leq m$.
\end{enumerate}
\end{theorem}

The Main Theorem of this paper is a global well-posedness statement, which (in essence) says that the solution given by Theorem \ref{thm:lwp4YM} can be extended globally in time. It uses crucially the fact that an admissible initial data set always possesses finite conserved energy, which was defined in \ref{eq:YMenergy}.

We are ready to state our Main Theorem.
\begin{MainTheorem}[Finite energy global well-posedness]
Let $(\Aini_{i}, \Eini_{i})$ be an admissible $H^{1}$ initial data set, and consider the initial value problem (IVP) for \eqref{eq:hyperbolicYM} with $(\Aini_{i}, \Eini_{i})$ as the initial data. Note that by admissibility, $(\Aini_{i}, \Eini_{i})$ always possesses finite conserved energy, i.e., $\bfE[\overline{\bfF}] < \infty$. Then the following statements hold.

\begin{enumerate}
\item The admissible solution given by Theorem \ref{thm:lwp4YM} extends globally in time, uniquely as an admissible solution in the temporal gauge $A_{0} = 0$. 

\item Moreover, if $\Aini_{i}$, $\Eini_{i}$ are smooth and $\rd_{x} \Aini_{i}, \Eini_{i} \in H^{m}_{x}$ for an integer $m \geq 0$, then the corresponding solution given by (1) is also smooth and satisfies
\begin{equation*}
	\rd_{t,x} A_{i} \in C^{k_{1}}_{t} (\bbR, H^{k_{2}}_{x})
\end{equation*}
for every pair $(k_{1}, k_{2})$ of nonnegative integers such that $k_{1} + k_{2} \leq m$.
\end{enumerate}
\end{MainTheorem}

We remark that quantitative estimates as in Parts (1), (2) of Theorem \ref{thm:lwp4YM} can be obtained by applying Theorem \ref{thm:lwp4YM} repeatedly. We have omitted these statements for the sake of brevity.

\begin{remark} 
The temporal gauge in the statements of Theorem \ref{thm:lwp4YM} and the Main Theorem does not play an essential role. We have used this mainly because it is a well-known gauge condition that is easy to impose. In fact, most of the analysis in this paper takes place under the caloric-temporal gauge condition which has been introduced above.
\end{remark}

\subsection{Outline of the paper}
After setting up the notations and conventions in Section \ref{sec:notations}, we begin the proof of the Main Theorem in Section \ref{sec:reduction} by reducing it to establishing Theorems \ref{thm:idEst}, \ref{thm:ctrlByE} and \ref{thm:dynEst}, all of which concern the system \eqref{eq:HPYM}. We remark that Theorems \ref{thm:idEst}, \ref{thm:ctrlByE} and \ref{thm:dynEst} will correspond to Steps 1, 2 and 3 which have been discussed in \S \ref{subsec:overview4GWP}, in the respective order.

The rest of the paper is devoted to proofs of Theorems \ref{thm:idEst}, \ref{thm:ctrlByE} and \ref{thm:dynEst}. In Section \ref{sec:prelim}, we gather some preliminary definitions and results needed in the remainder of the paper. In particular, we present techniques for dealing with covariant parabolic equations in \S \ref{subsec:covTech}. These techniques are put into use in the following section (Section \ref{sec:covParabolic}), where we study the covariant parabolic equations satisfied by the curvature 2-form $F$ of a solution to $F_{si} = \covD^{\ell} F_{\ell i}$  \eqref{eq:cYMHF} or $F_{s \mu} = \covD^{\ell} F_{s \mu}$ \eqref{eq:dYMHF}. As a result, we derive \emph{covariant parabolic estimates}, on which the whole paper is based. Then in Section \ref{sec:YMHF}, we study the systems \eqref{eq:cYMHF} and \eqref{eq:dYMHF} themselves under the caloric gauge condition $A_{s} = 0$\footnote{As a byproduct of our analysis, we obtain an independent proof of global existence of solutions to the original Yang-Mills heat flow $\rd_{s} A_{i} = \covD^{\ell} F_{\ell i}$ with finite Yang-Mills energy, which is a result originally due to R{\aa}de \cite{Rade:1992tu}. See Corollary \ref{cor:gwp4YMHF}.}. Then in the final three sections of this paper, we prove Theorems \ref{thm:idEst}, \ref{thm:ctrlByE} and \ref{thm:dynEst} in order.

For the convenience of the reader, this paper is complemented with two appendices. In Appendix \ref{appendix:dependence}, we provide a table that explains which results from the companion article \cite{Oh:6stz7nRe} are needed where in this paper. In Appendix \ref{appendix:LOS}, a list of symbols used in this paper is given. 

\subsection*{Acknowledgements}
The author is deeply indebted to his Ph.D. advisor Sergiu Klainerman, without whose support and constructive criticisms this work would not have been possible. Also, the author would like to thank Jonathan Luk for reading an earlier version of the paper and providing numerous helpful suggestions. He would also like to express gratitude to l'ENS d'Ulm for hospitality, where a major part of this work was done. He also thanks the anonymous referees for their useful suggestions for improving the readability of this paper. The author was supported by the Samsung Scholarship. 

\section{Notations and Conventions} \label{sec:notations}
In this paper, we will use boldfaced letters to refer to all space-time components; more precisely, $\bfF$ denotes any of the 6 components of $F_{\mu \nu}$, and $\bfA$, $\bfF_{s}$ denote any of the 4 components of $A_{\nu}$, $F_{s\nu}$, respectively. On the other hand, plain letters will refer to \emph{only} spatial components, i.e., $F = F_{ij}$, $A = A_{i}$, and $F_{s} = F_{si}$ for $i, j= 1,2,3$. A norm of such an expression, such as $\nrm{\bfA}$ or $\nrm{A}$, is to be understood as the maximum over the respective range of indices, i.e., $\nrm{\bfA} = \sup_{\mu=0,1,2,3} \nrm{A_{\mu}}$, $\nrm{A} = \sup_{i=1,2,3} \nrm{A_{i}}$ and etc.

We will use the notation $\calO(\phi_{1}, \ldots, \phi_{k})$ to denote a $k$-linear expression in the \emph{values} of $\phi_{1}, \ldots, \phi_{k}$. For example, when $\phi_{i}$ and the expression itself are scalar-valued, then $\calO(\phi_{1}, \ldots, \phi_{k}) = C \phi_{1} \phi_{2} \cdots \phi_{k}$ for some constant $C$. In many cases, however, each $\phi_{i}$ and the expression $\calO(\phi_{1}, \ldots, \phi_{k})$ will actually be matrix-valued. In such case, $\calO(\phi_{1}, \ldots, \phi_{k})$ will be a matrix, whose each entry is a $k$-linear functional of the matrices $\phi_{i}$. 

In stating various estimates, we will adopt the standard convention of denoting finite positive constants which are different, possibly line to line, by the same letter $C$. Dependence of $C$ on other parameters will be made explicit by subscripts. Furthermore, we will adopt the convention that $C$ \emph{always} depends in a non-decreasing manner with respect to each of its parameters, in its respective range, unless otherwise specified. For example, $C_{\bfE, \calIAlow}$, where $\bfE, \calIAlow$ range over positive real numbers, is a positive, non-decreasing function of both $\bfE$ and  $\calIAlow$.

Finally, in addition to plain greek and latin indices, we will utilize bold latin indices, such as $\bfa, \bfb$, which will refer to all possible indices $x^{0}, x^{1}, x^{2}, x^{3}, s$.

\section{Reduction of the Main Theorem} \label{sec:reduction}
\subsection{Preliminaries}
Before we begin, let us borrow a few definitions and lemmas from \cite{Oh:6stz7nRe}, which will be useful in the proof of the Main Theorem. We start with the notion of \emph{regular} functions and initial data sets.
\begin{definition}[Regular functions] \label{def:regFtns}
Let $I \subset \bbR$, $J \subset [0, \infty)$ be intervals.
\begin{enumerate}
\item A function $\phi = \phi(x)$ defined on $\bbR^{3}$ is \emph{regular} if $\phi \in H^{\infty}_{x} := \cap_{m=0}^{\infty} H^{m}_{x}$.
\item A function $\phi = \phi(t, x)$ defined on $I \times \bbR^{3}$ is \emph{regular} if $\phi \in C^{\infty}_{t}(I, H^{\infty}_{x}) := \cap_{k,m=0}^{\infty} C^{k}_{t}(I, H^{m}_{x})$.
\item A function $\psi = \psi(t, x,s)$ defined on $I \times \bbR^{3} \times J$ is \emph{regular} if $\phi \in C^{\infty}_{t,s}(I \times J, H^{\infty}_{x}):= \cap_{k,m=0}^{\infty} C^{k}_{t,s}(I \times J, H^{m}_{x})$.
\end{enumerate}
\end{definition}

\begin{definition}[Regular initial data sets] \label{def:reg4id}
We say that an initial data set $(\Aini_{i}, \Eini_{i})$ to \eqref{eq:hyperbolicYM} is \emph{regular} if, in addition to satisfying the constraint equation \eqref{eq:YMconstraint}, $\Aini_{i}, \Eini_{i} \in H^{\infty}_{x}$.
\end{definition}

The first lemma tells us that an admissible initial data set may be approximated by a sequence of regular initial data sets.
\begin{lemma}[Approximation lemma {\cite[Lemma 4.5]{Oh:6stz7nRe}}] \label{lem:regApprox}
Any admissible $H^{1}$ initial data set $(\Aini_{i}, \Eini_{i}) \in (\dot{H}^{1}_{x} \cap L^{3}_{x}) \times L^{2}_{x}$ can be approximated by a sequence of regular initial data sets $(\Aini_{(n) i}, \Eini_{(n) i})$ satisfying the constraint equation \eqref{eq:YMconstraint}. More precisely, the initial data sets $(\Aini_{(n) i}, \Eini_{(n) i})$ may be taken to satisfy the following properties.
\begin{enumerate}
\item $\Aini_{(n)}$ is smooth, compactly supported, and
\item $\Eini_{(n)} \in H^{\infty}_{x}$.
\end{enumerate}
\end{lemma}

The idea is that a regular initial data set leads to, by persistence of regularity, a \emph{regular solution} to \eqref{eq:hyperbolicYM} and \eqref{eq:HPYM}, which is defined as follows.
\begin{definition}[Regular solutions]  \label{def:mainThm:reg4YM}
We say that a representative $A_{\mu} : I \times \bbR^{3} \to \LieAlg$ of a classical solution to \eqref{eq:hyperbolicYM} is \emph{regular} if $A_{\mu} \in C^{\infty}_{t}(I, H^{\infty}_{x})$. Furthermore, we say that a smooth solution $A_{\bfa} : I \times \bbR^{3} \times J \to \LieAlg$ to \eqref{eq:HPYM} is \emph{regular} if $A_{\bfa} \in C^{\infty}_{t,s}(I \times J, H^{\infty}_{x})$.
\end{definition}

Related to the notion of a regular solution, we also introduce the definition of a \emph{regular gauge transform}.
\begin{definition}[Regular gauge transform] \label{def:reg4gt}
We say that a gauge transform $U$ on $I \times \bbR^{3} \times J$ is a \emph{regular gauge transform} if $U - \mathrm{Id}, \, U^{-1} - \mathrm{Id} \in C^{\infty}_{t,s}(I \times J, H^{\infty}_{x})$.
A gauge transform $U$ defined on $I \times \bbR^{3}$ is a \emph{regular gauge transform} if it is a regular gauge transform viewed as an $s$-independent gauge transform on $I \times \bbR^{3} \times J$ for some $J \subset [0, \infty)$.
\end{definition}

Note that a regular solution (either to \eqref{eq:hyperbolicYM} or to \eqref{eq:HPYM}) remains regular under a regular gauge transform.

Let us also define a norm $\calA_{0}[A_{0}](I)$ (for $I \subset \bbR$ an interval)  for a $\LieAlg$-valued function $A_{0}$ on $I \times \bbR^{3}$ as follows:
\begin{equation} \label{eq:calA0}
\begin{aligned}
	\calA_{0}[A_{0}](I) := & \nrm{A_{0}}_{L^{\infty}_{t} L^{3}_{x}(I)} + \nrm{\rd_{x} A_{0}}_{L^{\infty}_{t} L^{2}_{x}(I)} \\
					& + \nrm{A_{0}}_{L^{1}_{t} L^{\infty}_{x}(I)} + \nrm{\rd_{x} A_{0}}_{L^{1}_{t} L^{3}_{x}(I)} + \nrm{\rd_{x}^{(2)} A_{0}}_{L^{1}_{t} L^{2}_{x}(I)}.
\end{aligned}
\end{equation}

The next lemma shows that this norm is exactly what one needs in order to estimate gauge transforms into the temporal gauge.

\begin{lemma}[Estimates for gauge transform to temporal gauge {\cite[Lemma 4.6]{Oh:6stz7nRe}}] \label{lem:est4gt2temporal}
Consider the following ODE on $(-T, T) \times \bbR^{3}$:
\begin{equation} \label{eq:est4gt2temporal:0}
\left \{
\begin{aligned}
	& \rd_{t} V = V A_{0} \\
	& V(t=0) = \Vini,
\end{aligned}
\right.
\end{equation}
where we assume that $A_{0}$ is smooth and $\calA_{0}(-T, T) < \infty$.

\begin{enumerate}
\item Suppose that $\Vini = \Vini(x)$ is a smooth $\LieGrp$-valued function on $\set{t=0} \times \bbR^{3}$ such that 
\begin{equation*}
	\Vini, \Vini^{-1} \in L^{\infty}_{x}, \quad \rd_{x} \Vini, \rd_{x} \Vini^{-1} \in L^{3}_{x}, \quad \rd_{x}^{(2)} \Vini, \rd_{x}^{(2)} \Vini^{-1} \in L^{2}_{x} .
\end{equation*} 

Then there exists a unique solution $V$ to the ODE, which obeys the following estimates.\begin{equation} \label{eq:est4gt2temporal:V}
\left\{
\begin{aligned}
	\nrm{V}_{L^{\infty}_{t} L^{\infty}_{x} (-T, T)} 
	\leq & C_{\calA_{0}(-T, T)} \cdot \nrm{\Vini}_{L^{\infty}_{x}}, \\
	\nrm{\rd_{t,x} V}_{L^{\infty}_{t} L^{3}_{x} (-T, T)} 
	\leq & C_{\calA_{0}(-T, T)} \cdot (\nrm{\rd_{x} \Vini}_{L^{3}_{x}} +\calA_{0}(-T, T) \nrm{\Vini}_{L^{\infty}_{x}}), \\
	\nrm{\rd_{x} \rd_{t,x} V}_{L^{\infty}_{t} L^{2}_{x} (-T, T)} 
	\leq & C_{\calA_{0}(-T, T)} \cdot (\nrm{\rd_{x}^{(2)} \Vini}_{L^{2}_{x}} +\calA_{0}(-T, T) (\nrm{\Vini}_{L^{\infty}_{x}} + \nrm{\rd_{x} \Vini}_{L^{3}_{x}}) ). 
\end{aligned}
\right.
\end{equation}

These estimates remain true with $V$, $\Vini$ replaced by $V^{-1}$, $\Vini^{-1}$, respectively.
\end{enumerate}

\end{lemma}

\subsection{Reduction of the Main Theorem to Theorems \ref{thm:idEst}, \ref{thm:ctrlByE} and \ref{thm:dynEst}} \label{subsec:reduction}
For the convenience of the reader, we begin by recalling the definition of the system \eqref{eq:HPYM}:
\begin{equation*} \tag{HPYM}
\left\{
\begin{aligned}
	F_{s \mu} &= \covD^{\ell} F_{\ell \mu} \hspace{.25in} \hbox{ on } \hspace{.1in} I \times \bbR^{3} \times [0,1], \\
	\covD^{\mu} F_{\mu \nu} &= 0 \hspace{.5in} \hbox{ along } I \times \bbR^{3} \times \set{0}.
\end{aligned}
\right.
\end{equation*}
Recall also that the first line, i.e., $F_{s \mu} = \covD^{\ell} F_{\ell \mu}$, is the dynamic Yang-Mills heat flow \eqref{eq:dYMHF}.

The first theorem we present below concerns two points: A) After scaling, it ensures that \eqref{eq:dYMHF} can always be solved (from which we obtain a solution to \eqref{eq:HPYM}) on the unit $s$-interval provided that the conserved energy is small, and B) It gives a quantitative estimate for the data for \eqref{eq:HPYM} at $t=0$ (namely $\calI(0)$) in terms of the size $\calIini$ of the initial data set $(\Aini_{i}, \Eini_{i})$.

Given a solution $A_{\bfa}$ to \eqref{eq:HPYM} on $I \times \bbR^{3} \times [0,1]$, we assert the existence of a norm $\calI(t)$ $(t \in I)$ of $\rd_{t,x} F_{si}(t,s)$ $(0 < s < 1)$ and $\rd_{t,x}\Alow_{i}(t)$, for which the following theorem holds. The precise definition of $\calI(t)$ will be given in Section \ref{sec:pfOfIdEst}.

\begin{bigTheorem}[Transformation into the caloric-temporal gauge and estimates at $t=0$]  \label{thm:idEst}
Consider a regular initial data set $(\Aini_{i}, \Eini_{i})$ to \eqref{eq:hyperbolicYM} which satisfies\footnote{We remind the reader the notation $\overline{\bfF} = F_{\mu \nu}(t=0, s=0)$.}
\begin{equation} \label{eq:idEst:hyp}
	\nrm{\Aini}_{\dot{H}^{1}} < \dlt_{P}, \quad \bfE[\overline{\bfF}] < \dlt.
\end{equation}
where $\dlt_{P}, \dlt > 0$ are small absolute constants. Let $\Atemp_{i}$ be the corresponding regular solution to \eqref{eq:hyperbolicYM} in the temporal gauge given by Theorem \ref{thm:lwp4YM}, which we assume to exist on $(-T_{0}, T_{0}) \times \bbR^{3}$ for some $T_{0} > 0$. Then:

\begin{enumerate}
\item There exists a regular gauge transform $V = V(t,x)$ on $(-T_{0}, T_{0}) \times \bbR^{3}$ and a regular solution $A_{\bfa} = A_{\bfa}(t,x,s)$ to \eqref{eq:HPYM} in the caloric-temporal gauge on $(-T_{0}, T_{0}) \times \bbR^{3} \times [0,1]$ such that $F_{\mu \nu}$ is regular and
\begin{equation} \label{eq:idEst:0}
A_{\mu}(s=0) = V \Atemp_{\mu} V^{-1} - \rd_{\mu} V V^{-1}.
\end{equation}

\item With the notation $\calIini = \nrm{\Aini}_{\dot{H}^{1}_{x}} + \nrm{\Eini}_{L^{2}_{x}}$, the following estimates hold.
\begin{equation} \label{eq:idEst:1}
	\calI(0) \leq C_{\calIini} \cdot \calIini, \quad 
	\nrm{\Vini}_{L^{\infty}_{x}} \leq C_{\calIini}, \quad 
	\nrm{\rd_{x} \Vini}_{L^{3}_{x}} + \nrm{\rd_{x}^{(2)} \Vini}_{L^{2}_{x}} \leq C_{\calIini} \cdot \calIini.
\end{equation}
These estimates remain true with $\Vini$ replaced by $\Vini^{-1}$ as well.

\end{enumerate}
\end{bigTheorem}

The non-trivial initial gauge transform $\Vini$ has been introduced to ensure that $\Alow_{i}$ is smoother than $\Atemp_{i}$. It can be thought of as a substitute for the Uhlenbeck's lemma in our setting; for more details, we refer the reader to \cite[Proof of Theorem 4.8]{Oh:6stz7nRe}. The quantitative estimates \eqref{eq:idEst:1}, as well as the existence of $\Vini$, follow from \cite[Theorem 4.8]{Oh:6stz7nRe}.

The next theorem says that the conserved energy $\bfE(t)$ can be used to control $\calI(t)$ for every $t \in (-T_{0}, T_{0})$; we refer the reader to Step 2 in \S \ref{subsec:overview4GWP} for the idea behind the theorem.

\begin{bigTheorem} [Fixed time estimates by $\bfE$] \label{thm:ctrlByE}
Let $T_{0} > 0$ and consider a regular solution $A_{\bfa}$ to \eqref{eq:HPYM} in the caloric-temporal gauge on $(-T_{0}, T_{0}) \times \bbR^{3} \times [0,1]$ satisfying $\calI(0) < \infty$ and $\bfE[\overline{\bfF}] < \infty$.

 Then for $t \in (-T_{0}, T_{0})$, $\calI(t)$ can be bounded in terms of the initial data and $T_{0}$, i.e.,
 \begin{equation} \label{eq:ctryByE:0}
\sup_{t \in (-T_{0}, T_{0})} \calI(t) \leq C_{\calI(0), \bfE[\overline{\bfF}], T_{0}} < \infty.
\end{equation}
\end{bigTheorem}

From Theorems \ref{thm:idEst} and \ref{thm:ctrlByE}, we obtain {\it a priori} estimates on each fixed-time slice $\set{t} \times \bbR^{3} \times [0,1]$. In order to estimate the gauge transform back to the temporal gauge, however, one needs to control $\calA_{0}$ (recall Lemma \ref{lem:est4gt2temporal}), and for this purpose it turns out that these fixed-time estimates are insufficient. In order to estimate $\calA_{0}$ we need to take advantage of the fact that the dynamic variables $F_{s}, \Alow$ satisfy wave equations, which is exactly what the next theorem achieves. 

\begin{bigTheorem}[Short time estimates for \eqref{eq:HPYM} in the caloric-temporal gauge] \label{thm:dynEst}
Let $T_{0} > 0$ and consider a regular solution $A_{\bfa}$ to \eqref{eq:HPYM} in the caloric-temporal gauge on $(-T_{0}, T_{0}) \times \bbR^{3} \times [0,1]$ such that $\bfF = F_{\mu \nu}$ is regular on $(-T_{0}, T_{0}) \times \bbR^{3} \times [0,1]$. Suppose furthermore that 
\begin{equation} \label{eq:dynEst:hyp}
	\sup_{t \in (-T_{0}, T_{0})} \bfE[\bfF(t, s=0)] < \dlt, \quad \sup_{t \in (-T_{0}, T_{0})} \calI(t) \leq D,
\end{equation}
where $D > 0$ is an arbitrarily large finite number and $\dlt > 0$ is an absolute small constant independent of $D$.

Then there exists a number $d = d(D, \dlt)$, which depends on $D, \dlt$ in a non-increasing fashion, such that on every subinterval $I_{0} \subset I$ of length $d$, the following estimate holds.
\begin{equation} \label{eq:dynEst:0}
	\sup_{s \in [0,1]} \nrm{\rd_{t,x} A(s)}_{C_{t}(I_{0}, L^{2}_{x})} + \calA_{0}[A_{0}(s=0)](I_{0}) \leq C_{D, \dlt}.
\end{equation}
\end{bigTheorem}

In essence, Theorem \ref{thm:dynEst} is a result of a fairly standard local-in-time analysis of the wave equations of \eqref{eq:HPYM}. However, there is a little twist, which necessitates the extra hypotheses \eqref{eq:dynEst:hyp} and demands an explanation. Among the equations of \eqref{eq:HPYM} is an equation for $F_{s0}$ which, unlike the other components $F_{si}$, is parabolic. As such, smallness of the time interval \emph{cannot} be utilized to solve this equation in a perturbative manner\footnote{In \cite{Oh:6stz7nRe}, this issue is bypassed by keeping the lengths of the $s$- and the time intervals fixed and requiring the size of the data to be small by scaling. If one unwinds the scaling, this amounts to taking the length of both the $s$- and the time intervals small.}. What saves us is the fact that the parabolic equation for $F_{s0}$ is \emph{covariant}, and therefore can be analyzed using the covariant techniques presented in \S \ref{subsec:covTech}. The first inequality of \eqref{eq:dynEst:hyp} provides the necessary smallness for this analysis, whereas the second one is needed to estimate the errors arising from switching covariant derivatives to usual derivatives. A rigorous proof will be given in Section \ref{sec:pfOfDynEst}.

We are now prepared to give a proof of the Main Theorem, under the assumption that Theorems \ref{thm:idEst}, \ref{thm:ctrlByE} and \ref{thm:dynEst} are true. 

\begin{proof} [Proof of the Main Theorem, assuming Theorems \ref{thm:idEst}, \ref{thm:ctrlByE} and \ref{thm:dynEst}]
To begin with, let us consider a \emph{regular} initial data set.  Applying Theorem \ref{thm:lwp4YM} to $(\Aini_{i}, \Eini_{i})$, there exists a unique regular solution to \eqref{eq:hyperbolicYM} in the temporal gauge on some time interval centered at $0$, which we will denote by $\Atemp_{\mu}$. We will first show this solution exists globally in time. 

For the purpose of contradiction, suppose that the solution $\Atemp_{\mu}$ cannot be extended globally as a unique regular solution to \eqref{eq:hyperbolicYM} in the temporal gauge. Then there exists a positive finite number $0 < T_{0} < \infty$, which is the largest positive number for which the solution $\Atemp_{\mu}$ can be extend as a regular solution on $(-T_{0}, T_{0})$. We claim that there exists a finite positive constant $C = C_{\calIini, \bfE[\overline{\bfF}], T_{0}}$, which depends only on the initial data and $T_{0}$, such that the following inequality holds.
\begin{equation}\label{eq:mainThm:pf:0}
\sup_{t \in (-T_{0}, T_{0})} \nrm{\Atemp_{\mu}(t)}_{\dot{H}^{1}_{x}} \leq C_{\calIini, \bfE[\overline{\bfF}], T_{0}} < \infty.
\end{equation}

Let us complete the proof of the Main Theorem first, under the assumption that the claim is true. If the claim is true, then the solution can be extended as a unique regular solution to $(-T_{0}-\eps, T_{0}+\eps)$ for some $\eps >0$ by Theorem \ref{thm:lwp4YM}, which is a contradiction. It follows that $T_{0} = \infty$, and thus $\Atemp_{\mu}$ can be extended globally in time as a unique regular solution to \eqref{eq:hyperbolicYM} in the temporal gauge. Observe that the estimate \eqref{eq:mainThm:pf:0} still holds for the global solution $\Atemp_{\mu}$ for every $T_{0} > 0$.

Next, Lemma \ref{lem:regApprox} implies that an admissible initial data can be approximated by a sequence of regular initial data sets $(\Aini_{(n) i}, \Eini_{(n) i})$. Let us denote the corresponding unique global regular solutions by $A_{(n)\mu}$. Using Theorem \ref{thm:lwp4YM} repeatedly (with the help of \eqref{eq:mainThm:pf:0}), the following statement may be proved: For every $T_{1} > 0$, the sequence of regular solutions $A_{(n)\mu}$ restricted to the time interval $(-T_{1}, T_{1})$ is a Cauchy sequence in the topology $C_{t}((-T_{1}, T_{1}), \dot{H}^{1}_{x} \cap L^{3}_{x})$. Hence a limit $A_{\mu}$ exists on $(-T_{1}, T_{1})$. Moreover, it is also possible to show that $\rd_{t} A_{(n)\mu} \to \rd_{t} A_{\mu}$ in $C_{t}((-T_{1}, T_{1}), L^{2}_{x})$. Thus it follows that $A_{\mu}$ is an admissible solution to \eqref{eq:hyperbolicYM} in the temporal gauge on $(-T_{1}, T_{1})$. Uniqueness among the class of admissible solutions follows from the corresponding statement for regular solutions. As $T_{1} > 0$ is arbitrary, $A_{\mu}$ is global, and the Main Theorem follows.

We only have left to establish the claim \eqref{eq:mainThm:pf:0}, which is a rather straightforward application of Theorems \ref{thm:idEst}, \ref{thm:ctrlByE} and \ref{thm:dynEst}. First, by scaling, we may assume that $\nrm{\Aini}_{\dot{H}^{1}} < \dlt_{P}$ and $\bfE[\overline{\bfF}] < \dlt$, i.e., \eqref{eq:idEst:hyp} holds. This allows us to apply Theorem \ref{thm:idEst}, from which we obtain a regular gauge transform $V$ and a regular solution $A_{\bfa}$ to \eqref{eq:HPYM} in the caloric-temporal gauge on $(-T_{0}, T_{0}) \times \bbR^{3} \times [0,1]$ such that \eqref{eq:idEst:0} holds. By Theorem \ref{thm:ctrlByE}, along with the estimate for $\calI(0)$ in \eqref{eq:idEst:1}, we see that
\begin{equation*}
\sup_{t \in (-T_{0}, T_{0})} \calI(t) \leq C_{\calIini, \bfE[\overline{\bfF}], T_{0}} < \infty.
\end{equation*}

To use Theorem \ref{thm:dynEst}, let us cover $(-T_{0}, T_{0})$ by subintervals of length $d$; the number of subintervals required can be bounded from above by, say, $10 (T_{0}/d)$. Applying Theorem \ref{thm:dynEst} on each subinterval, we are led to the estimate
\begin{equation*}
\sup_{s \in [0,1]} \nrm{\rd_{t,x} A(s)}_{C_{t}(I_{0}, L^{2}_{x})} + \calA_{0}[A_{0}(s=0)](I_{0}) \leq C_{\calIini, \bfE[\overline{\bfF}], T_{0}} < \infty.
\end{equation*}

The only remaining step is to transfer the above estimate to $\Atemp_{\mu}$; for this purpose, observe from \eqref{eq:idEst:0} that $V$ satisfies $\rd_{t} V = V A_{0}(s=0)$. Using Lemma \ref{lem:est4gt2temporal}, along with the previous estimates for $\calA_{0}$ and $\Vini$, we are led to the following estimate for the gauge transform $V$:
\begin{equation*}
\nrm{V}_{L^{\infty}_{t,x}} + \nrm{\rd_{t,x} V}_{L^{\infty}_{t} L^{3}_{x}} + \nrm{\rd_{x}^{(2)} V}_{L^{\infty}_{t} L^{2}_{x}} \leq C_{\calIini, \bfE[\overline{\bfF}], T_{0}} < \infty.
\end{equation*}

Here, all norms have been taken over $(-T_{0}, T_{0}) \times \bbR^{3}$. The preceding estimate, applied to the formula \eqref{eq:idEst:0}, implies \eqref{eq:mainThm:pf:0} as desired. \qedhere
\end{proof}
 
\section{Preliminaries} \label{sec:prelim}
\subsection{p-normalized norms and the Correspondence Principle} \label{subsec:prelim:pnorm}
For the purpose of studying parabolic equations, it is quite convenient to utilize norms that are normalized according to the scaling properties of these equations. An estimate concerning homogeneous norms can be easily translated to the corresponding estimate in terms of the normalized norms, via a simple principle we dub the \emph{Correspondence Principle}. For a more detailed discussion, we refer the reader to \cite[\S 3.3 -- 3.5]{Oh:6stz7nRe}; here, we will only give a brief summary which will suffice for the use in the present paper.

The basic idea is that $s^{1/2}$ scales like $x$, where $s$ is the time parameter for the parabolic equation. Therefore, whatever `dimension of $x$' a norm has, we will normalize by compensating it with the appropriate factor of $s^{-1/2}$. 

To be more precise, consider a norm $\nrm{\cdot}_{X}$ defined for functions on $\bbR^{3}$, which is homogeneous of degree $2 \ell$ in the sense that $\nrm{\phi(\cdot)}_{X} = \lmb^{2 \ell} \nrm{\phi(\cdot / \lmb)}_{X}$. This indicates that the norm $\nrm{\cdot}_{X}$ has the `dimension of $x^{2 \ell}$', and therefore we shall normalize it by multiplying by $s^{-\ell}$. Accordingly, we define the \emph{p-normalization}\footnote{The p in `p-normalization' stands for `parabolic'.} of $X$ at $s$ (denoted by the calligraphic typeface $\calX(s)$) as 
\begin{equation*}
	\nrm{\cdot}_{\calX(s)} := s^{-\ell} \nrm{\cdot}_{X}.
\end{equation*}

The only p-normalizations considered in this paper are those of $L^{q}_{x}$ and $\dot{H}^{m}_{x}$, which are denoted by $\calL^{q}_{x}(s)$ and $\dot{\calH}^{m}_{x}(s)$, respectively.

A derivative, such as $\rd_{i}$ or $\covD_{i}$, has the `dimension of $x^{-1}$'. Therefore, the \emph{p-normalizations} of $\rd_{i}$ and $\covD_{i}$ at $s$ (denoted by $\nb_{i}(s)$ and $\calD_{i}(s)$, respectively) are defined as
\begin{equation*}
	\nb_{i}(s) := s^{1/2} \rd_{i}, \quad \calD_{i}(s) := s^{1/2} \covD_{i}.
\end{equation*}

An estimate concerning homogeneous norms (e.g. H\"older, Sobolev Gagliardo-Nirenberg and etc.) naturally leads to a corresponding estimate in terms of the respective p-normalized norms; we will refer to this process as the \emph{Correspondence Principle}. 
%We will not make any effort to formulate a rigorous version of the principle, as it would be unpractical and overly complicated; instead, we will be satisfied with the following `cookbook-recipe' type formulation, whose validity should be obvious every time the principle is invoked.

\begin{corrPrinciple}
Suppose that we are given an estimate in terms of norms $X_{i}$ of scalar- or $\LieAlg$-valued functions $\sgm_{i} = \sgm_{i}(x)$, all of which are homogeneous. Suppose furthermore that the estimate is scale-invariant, in the sense that both sides transform the same under scaling. 

Starting from such an estimate, make the following substitutions on both sides: 
\begin{equation*}
\rd_{x} \to \nb_{x}(s), \quad  \covD_{x} \to \calD_{x}(s), \quad  X_{i} \to \calX_{i}(s). 
\end{equation*}

Furthermore, substitute each $\sgm_{i}(x)$ by a function $G_{i}(x, s)$ on $\bbR^{3} \times J$. Then the resulting estimate still holds with the same constant for every $s \in J$.
\end{corrPrinciple}

The `proof' of this principle is quite simple; it amounts to the observation that the weight of $s$ required to p-normalize each side is the same. 

Next, let us define some norms with respect to the $s$-variable. Let $J \subset [0,\infty)$ be an interval, and consider a measurable function $f = f(s)$ defined on $J$. For $\ell \in \bbR$ and $1 \leq p \leq \infty$, we define the  norm $\calL^{\ell,p}_{s}(J)$ by
\begin{equation*}
	\nrm{f}_{\calL^{\ell,p}_{s}(J)} := \bb( \int_{J} (s^{\ell} f(s))^{p} \frac{\ud s}{s} \bb)^{1/p}.
\end{equation*}

The following lemma, which is nothing but the H\"older inequality in the $s$-variable, is quite useful.
\begin{lemma} [H\"older for $\calL^{\ell,p}_{s}$] \label{lem:Holder4Ls}
Let $\ell, \ell_{1}, \ell_{2} \geq 0$, $1 \leq p, p_{1}, p_{2} \leq \infty$ and $f, g$ functions on $J = (0,s_{0})$ (or $J = (0, s_{0}]$) such that $\nrm{f}_{\calL^{\ell_{1}, p_{1}}_{s}}, \nrm{g}_{\calL^{\ell_{2}, p_{2}}_{s}} < \infty$. Then we have
\begin{equation*} 
	\nrm{fg}_{\calL^{\ell,p}_{s}(J)} \leq C s_{0}^{\ell-\ell_{1}-\ell_{2}} \nrm{f}_{\calL^{\ell_{1},p_{1}}_{s}(J)} \nrm{g}_{\calL^{\ell_{2},p_{2}}_{s}(J)}
\end{equation*}
provided that either $\ell = \ell_{1} + \ell_{2}$ and $\frac{1}{p} = \frac{1}{p_{1}} + \frac{1}{p_{2}}$, or $\ell > \ell_{1} + \ell_{2}$ and $\frac{1}{p} \geq \frac{1}{p_{1}} + \frac{1}{p_{2}}$. In the former case, $C = 1$, while in the latter case, $C$ depends on $\ell-\ell_{1}-\ell_{2}$ and $\frac{1}{p} - \frac{1}{p_{1}} - \frac{1}{p_{2}}$.
\end{lemma}

Finally, let us consider a mix of the preceding two types of norms. For a function $\psi = \psi(x, s)$ defined on $\bbR^{3}\times J$ such that $s \to \nrm{\psi(s)}_{\calX(s)}$ is measurable on $J$, we define the norm $\calL^{\ell, p}_{s} \calX(J)$ to be
\begin{equation*}
	\nrm{\psi}_{\calL^{\ell,p}_{s} \calX(J)} := \nrm{\nrm{\psi(s)}_{\calX(s)}}_{\calL^{\ell, p}_{s} (J)}.
\end{equation*}

In order to derive estimates in terms of such norms, we will often use the Correspondence Principle and Lemma \ref{lem:Holder4Ls} in tandem. Let us demonstrate this with an example. Starting with a homogeneous estimate (which follows from H\"older and Corollary \ref{cor:covSob})
\begin{equation*}
	\nrm{\calO(\sgm_{1}, \sgm_{2})}_{L^{2}_{x}} 
	\leq C \nrm{\sgm_{1}}_{L^{3}_{x}} \nrm{\sgm_{2}}_{L^{6}_{x}} 
	\leq C \nrm{\sgm_{1}}^{1/2}_{L^{2}_{x}} \nrm{\covD_{x} \sgm_{1}}^{1/2}_{L^{2}_{x}} \nrm{\covD_{x} \sgm_{2}}_{L^{2}_{x}},
\end{equation*}
applying the Correspondence Principle, taking the $\calL^{\ell, p}_{s}(0, s_{0})$ norm and using Lemma \ref{lem:Holder4Ls}, we arrive at
\begin{equation*}
	\nrm{\calO(G_{1}, G_{2})}_{\calL^{\ell, p}_{s} \calL^{2}_{x}(0, s_{0})} 
	\leq C s_{0}^{\ell-\ell_{1}-\ell_{2}} \nrm{G_{1}}_{\calL^{\ell_{1}, p_{1}}_{s} \calL^{2}_{x}(0, s_{0})}^{1/2} 
		\nrm{\calD_{x} G_{1}}_{\calL^{\ell_{1}, p_{1}}_{s} \calL^{2}_{x}(0, s_{0})}^{1/2}
		 \nrm{\calD_{x} G_{2}}_{\calL^{\ell_{2}, p_{2}}_{s} \calL^{2}_{x}(0, s_{0})},
\end{equation*}
where $C, \ell, \ell_{1}, \ell_{2}, p, p_{1}, p_{2}$ are as in Lemma \ref{lem:Holder4Ls}.

\subsection{Covariant techniques} \label{subsec:covTech}
Here, we collect some techniques which are applicable to the study of \emph{covariant} parabolic equations. The use of such techniques, instead of those for handling the usual scalar heat equation, is the key analytic difference between this paper and \cite{Oh:6stz7nRe}. We remark that estimates like those in this subsection were also established in Tao \cite{Tao:2008wn} and Smith \cite{Smith:2011ef} in the case $\LieGrp = \mathrm{SO}(n)$.

\begin{lemma}[Kato's inequality] \label{lem:kato}
Let $\sgm$ be a $\LieAlg$-valued function. Then
\begin{equation} \label{eq:kato:1}
	\abs{\rd_{x} \abs{\sgm}} \leq \abs{\covD_{x} \sgm}
\end{equation}
in the distributional sense.
\end{lemma}
\begin{remark} 
Kato's inequality is often also referred to as the \emph{diamagnetic inequality}. (See, e.g., \cite{Tao:2008wn} and \cite{Smith:2011ef}).
\end{remark}
\begin{proof} 
Let $\eps > 0$. We compute
\begin{equation*}
	\rd_{x} \sqrt{(\sgm, \sgm) + \eps} = \frac{(\sgm, \covD_{x} \sgm)}{\sqrt{(\sgm, \sgm) + \eps}} \leq \abs{\frac{\sqrt{(\sgm, \sgm)}}{\sqrt{(\sgm, \sgm) + \eps}}}\cdot \abs{\covD_{x} \sgm} \leq \abs{\covD_{x} \sgm}.
\end{equation*} 

Testing against a non-negative test function and taking $\eps \to 0$, we see that $\rd_{x} \abs{\sgm} \leq  \abs{\covD_{x} \sgm}$ in the distributional sense. Repeating the same argument to $- \rd_{x} \sqrt{(\sgm, \sgm) + \eps}$, we obtain \eqref{eq:kato:1}.
\end{proof}

The following Sobolev inequalities for covariant derivatives are easy consequences of Kato's inequality.
\begin{corollary}[Sobolev and Gagliardo-Nirenberg inequalities for covariant derivatives] \label{cor:covSob}
For a regular $\LieAlg$-valued function $\sgm$, the following estimates hold.
\begin{align}
	&\nrm{\sgm}_{L^{3}_{x}} \leq C \nrm{\sgm}_{L^{2}_{x}}^{1/2} \nrm{\covD_{x} \sgm}_{L^{2}_{x}}^{1/2},  \label{eq:covSob:1} \\
	&\nrm{\sgm}_{L^{6}_{x}} \leq C \nrm{\covD_{x} \sgm}_{L^{2}_{x}}, \label{eq:covSob:2} \\
	&\nrm{\sgm}_{L^{\infty}_{x}} \leq C \nrm{\covD_{x} \sgm}_{L^{2}_{x}}^{1/2} \nrm{\covD_{x}^{(2)} \sgm}_{L^{2}_{x}}^{1/2}. \label{eq:covSob:3}
\end{align}
\end{corollary}

%By integrating by parts, it is not difficult to prove the following basic elliptic estimate.
%\begin{lemma} \label{}
%For sufficiently nice $\sgm$, the following estimate holds.
%\begin{equation*}
%	\nrm{\covD_{x}^{(2)} \sgm}_{L^{2}_{x}}^{2} \leq \nrm{\covD^{\ell} \covD_{\ell} \sgm}_{L^{2}_{x}}^{2} + \frac{1}{2} \nrm{\LieBr{F_{k \ell}}{\sgm}}_{L^{2}_{x}}^{2} + \int \abs{(\LieBr{F^{k \ell}}{\covD_{k} \sgm}, \covD_{\ell} \sgm)} \, \ud x.
%\end{equation*}
%\end{lemma}

Next, consider an inhomogeneous covariant heat equation
\begin{equation} \label{eq:covHeat}
	(\covD_{s} - \covD^{\ell} \covD_{\ell})\sgm = \calN.
\end{equation}

Adapting the usual proof of the energy integral inequality (integration by parts) for the ordinary heat equation to \eqref{eq:covHeat}, we obtain the following gauge-invariant version of the energy integral inequality.

\begin{lemma}[Energy integral inequality] \label{lem:pEst4covHeat}
Let $\ell \in \bbR$, $(s_{1}, s_{2}] \subset (0, \infty)$ and suppose that $\sgm$ and $A_{i}$ are `sufficiently nice\footnote{A sufficient condition for \eqref{eq:pEst4covHeat:1} to hold, which will be verifiable in applications below, is that $\sgm$ is smooth and the left-hand side of \eqref{eq:pEst4covHeat:1} is finite.}'. Then the following estimate holds.
\begin{equation} \label{eq:pEst4covHeat:1}
\begin{aligned}
	&\nrm{\sgm}_{\calL^{\ell, \infty}_{s} \calL^{2}_{x}(s_{1}, s_{2}]} + \nrm{\calD_{x} \sgm}_{\calL^{\ell, 2}_{s} \calL^{2}_{x}(s_{1}, s_{2}]} \\
	& \qquad \leq  C s_{1}^{\ell} \nrm{\sgm(s_{1})}_{\calL^{2}_{x}(s_{1})} + C (\ell - 3/4) \nrm{\sgm}_{\calL^{\ell,2}_{s} \calL^{2}_{x}(s_{1}, s_{2}]} 
	+ C \nrm{\calN}_{\calL^{\ell + 1, 1}_{s} \calL^{2}_{x} (s_{1}, s_{2}]}.
\end{aligned}
\end{equation}
\end{lemma}
\begin{proof}
We will carry out a formal computation, discarding all boundary terms at the spatial infinity which arise; it is easy to verify that for `sufficiently nice' $\sgm$ and $A_{i}$, this can be made into a rigorous proof. 

Let $\sbr \in (s_{1}, s_{2}]$. Taking the bi-invariant inner product of the equation $(\covD_{s} - \covD^{\ell} \covD_{\ell})\sgm = \calN$ with $s^{2\ell - 3/2} \sgm$ and integrating by parts over $(s_{1}, \sbr]$, we arrive at
\begin{align*}
	& \frac{1}{2} \sbr^{2\ell - 3/2} \int (\sgm, \sgm)(\sbr) \, \ud x + \int_{s_{1}}^{\sbr} \int s^{2 \ell - 1/2} (\covD^{\ell} \sgm, \covD_{\ell} \sgm)(s) \, \ud x \, \frac{\ud s}{s} \\
	& \quad = \frac{1}{2} s_{1}^{2\ell - 3/2} \int (\sgm, \sgm)(s_{1}) \, \ud x + (\ell - 3/4) \int_{s_{1}}^{\sbr} \int s^{2\ell - 3/2} (\sgm, \sgm)(s) \, \ud x \, \frac{\ud s}{s} \\
	& \qquad + \int_{s_{1}}^{\sbr} s^{2\ell - 1/2} (\calN(s), \sgm(s)) \, \ud x \, \frac{\ud s}{s}.
\end{align*}

Taking the supremum over $s_{1} < \sbr \leq s_{2}$ and rewriting in terms of p-normalized norms, we obtain
\begin{align*}
	\frac{1}{2} \nrm{\sgm}_{\calL^{\ell, \infty}_{s} \calL^{2}_{x}(s_{1}, s_{2}]}^{2} + \nrm{\sgm}_{\calL^{\ell,2}_{s} \dot{\calH}^{1}_{x}(s_{1}, s_{2}]}^{2}
	\leq & \frac{1}{2} s_{1}^{2\ell} \nrm{\sgm(s_{1})}_{\calL^{2}_{x}(s_{1})}^{2} + (\ell - 3/4) \nrm{\sgm}_{\calL^{\ell,2}_{s} \calL^{2}_{x} (s_{1}, s_{2}]}^{2} \\
	& + \nrm{(\calN, \sgm)}_{\calL^{2\ell+1,1}_{s} \calL^{1}_{x}(s_{1}, s_{2}]}.
\end{align*}

By H\"older, Lemma \ref{lem:Holder4Ls} and Cauchy-Schwarz, the last term can be estimated by $\nrm{\calN}_{\calL^{\ell+1,1}_{s} \calL^{2}_{x} (s_{1}, s_{2}]}^{2} + \frac{1}{4} \nrm{\sgm}_{\calL^{\ell, \infty}_{s} \calL^{2}_{x}(s_{1}, s_{2}]}^{2}$, where the latter can be absorbed into the left-hand side. Then taking the square root of both sides, we obtain \eqref{eq:pEst4covHeat:1}. \qedhere
\end{proof}

Proceeding as in the proof of Kato's inequality, we can derive the following parabolic inequality for $\abs{\sgm}$.
\begin{lemma} [Bochner-Weitzenb\"ock-type inequality] \label{lem:heatIneq4abs}
The following inequality holds in the distributional sense.
\begin{equation} \label{eq:heatIneq4abs}
	(\rd_{s}  - \lap) \abs{\sgm} \leq \abs{\calN}. 
\end{equation}
\end{lemma}
\begin{proof} 
This lemma was essentially proved in \cite{Tao:2008wn}; we shall give a proof nevertheless for completeness. Let $\eps > 0$. We compute
\begin{align*}
	&\rd_{s} \sqrt{(\sgm, \sgm) + \eps} = \frac{1}{\sqrt{(\sgm, \sgm) + \eps}} (\sgm, \covD_{s} \sgm), \\
	&\lap \sqrt{(\sgm, \sgm) + \eps} = \frac{1}{\sqrt{(\sgm, \sgm) + \eps}} \bb( (\sgm, \covD^{\ell} \covD_{\ell} \sgm) + (\covD^{\ell} \sgm , \covD_{\ell} \sgm) - \frac{(\sgm, \covD^{\ell} \sgm) (\sgm, \covD_{\ell} \sgm)}{(\sgm, \sgm)+\eps} \bb).
\end{align*}

Therefore,
\begin{align*}
	(\rd_{s} - \lap) \sqrt{(\sgm, \sgm) + \eps} 
	= &\frac{1}{\sqrt{(\sgm, \sgm)+\eps}} \bb( (\sgm, \calN) - (\covD^{\ell} \sgm , \covD_{\ell} \sgm) + \frac{(\sgm, \covD^{\ell} \sgm) (\sgm, \covD_{\ell} \sgm)}{(\sgm, \sgm)+\eps} \bb) \\
	\leq & \frac{1}{\sqrt{(\sgm, \sgm) + \eps}} (\sgm, \calN) \leq \abs{\calN}.
\end{align*}
Testing against a non-negative test function and taking $\eps \to 0$, we obtain \eqref{eq:heatIneq4abs}. \qedhere
\end{proof}

The virtue of \eqref{eq:heatIneq4abs} is that it allows us to use estimates arising from the (standard) heat kernel. Before we continue, let us briefly recap the definition and basic properties of the heat kernel. 

Let $e^{s\lap}$ denote the solution operator for the free heat equation. It is an integral operator, defined by
\begin{equation*}
	e^{s \lap} \psi_0(x) = \frac{1}{\sqrt{4 \pi s}^3} \int e^{- \abs{x - y}^2/{4 s}} \psi_0 (y) \, \ud y.
\end{equation*}

The kernel on the right hand side is called the \emph{heat kernel} on $\bbR^3$. Using Young's inequality, it is easy to derive the following basic inequality for the heat kernel:
\begin{equation} \label{eq:est4heatkernel}
	\nrm{e^{s \lap} \psi_0 }_{L^r_x} \leq C_{p,r} \, s^{-3/(2p) + 3/(2r)} \nrm{\psi_0 }_{L^p_x},
\end{equation}
where $1 \leq p \leq r$.

Now consider the initial value problem for the inhomogeneous heat equation $(\rd_{s} - \lap) \psi = N$. Duhamel's principle tells us that this problem can be equivalently formulated in an integral form as follows:
\begin{equation*}
	\psi(s) = e^{s \lap} \psi(s=0) + \int_{0}^{s} e^{(s - \sbr)\lap} N(\sbr) \, \ud \sbr.
\end{equation*}

With these prerequisites, we are ready to derive a simple comparison principle for $\abs{\sgm}$, along with a simple weak maximum principle; both statements are easily proved using basic properties of the heat kernel.

\begin{corollary} \label{cor:heatIneq4Duhamel}
Let $\overline{\sgm} := \sgm(s=0)$. Then the following point-wise inequality holds.
\begin{equation} \label{eq:heatIneq4Duhamel:1}
	\abs{\sgm}(x, s) \leq e^{s \lap} \abs{\overline{\sgm}}(x) + \int_{0}^{s} e^{(s-\sbr) \lap} \abs{\calN(\sbr)}(x) \, \ud \sbr,
\end{equation}

As a consequence, the following \emph{weak maximum principle} holds.
\begin{equation} \label{eq:heatIneq4Duhamel:2}
	\sup_{0 \leq \sbr \leq s}\nrm{\sgm(\sbr)}_{L^{\infty}_{x}} 
	\leq \nrm{\sgm(s=0)}_{L^{\infty}_{x}} + \int_{0}^{s} \nrm{\calN(\sbr)}_{L^{\infty}_{x}} \, \ud \sbr.
\end{equation}
\end{corollary}
\begin{proof} 
The first inequality is an immediate consequence of \eqref{eq:heatIneq4abs}, Duhamel's principle, and the fact that the heat kernel $K(x,y) = \frac{1}{(4 \pi s)^{3/2} }e^{-\abs{x-y}^{2}/ 4 s}$ is everywhere positive. The second one follows by taking the $L^{\infty}_{x}$ norm of \eqref{eq:heatIneq4Duhamel:1} and using \eqref{eq:est4heatkernel}. \qedhere
\end{proof}

For later use, we need the following lemma for the Duhamel integral, whose proof utilizes the basic inequality \eqref{eq:est4heatkernel} for the heat kernel. 
\begin{lemma} \label{lem:est4Duhamel}
The following estimate holds.
\begin{equation} \label{eq:est4Duhamel:1}
	\nrm{\int_{0}^{s} e^{(s-\sbr)\lap} \calN(\sbr) \, \ud \sbr}_{\calL^{1,2}_{s} \calL^{2}_{x}(0,s_{0}]} \leq C \nrm{\calN}_{\calL^{1+1,2}_{s} \calL^{1}_{x}(0,s_{0}]}.
\end{equation}
\end{lemma}
\begin{proof} 
Unwinding the definitions of p-normalized norms, \eqref{eq:est4Duhamel:1} is equivalent to
\begin{equation} \label{eq:est4Duhamel:pf:1}
	\bb( \int_{0}^{s_{0}} s^{1/2} \nrm{\int_{0}^{s} e^{(s-\sbr)\lap} \calN(\sbr) \, \ud \sbr}_{L^{2}_{x}}^{2} \, \frac{\ud s}{s} \bb)^{1/2} \leq C \bb( \int_{0}^{s_{0}} s \nrm{\calN(s)}_{L^{1}_{x}}^{2} \, \frac{\ud s}{s} \bb)^{1/2}.
\end{equation} 

Let us put $f(s) = s^{1/2} \nrm{\calN(s)}_{L^{1}_{x}}$; then it suffices to estimate the left-hand side of \eqref{eq:est4Duhamel:pf:1} by $C \nrm{f}_{\calL^{2}_{s}(0, s_{0}]}$. By Minkowski and \eqref{eq:est4heatkernel}, we have
\begin{equation*}
	\nrm{\int_{0}^{s} e^{(s-\sbr)\lap} \calN(\sbr) \, \ud \sbr}_{L^{2}_{x}} \leq C \int_{0}^{s} (s-\sbr)^{-3/4} (\sbr)^{1/2}  f(\sbr) \, \frac{\ud \sbr}{\sbr} 
\end{equation*}

Therefore the left-hand side of \eqref{eq:est4Duhamel:pf:1} is bounded from above by
\begin{equation} \label{eq:est4Duhamel:pf:2}
C \bb( \int_{0}^{s_{0}}  \bb(\int_{0}^{s} s^{1/4} (s-\sbr)^{-3/4} (\sbr)^{1/2} f(\sbr) \, \frac{\ud \sbr}{\sbr} \bb)^{2} \, \frac{\ud s}{s}\bb)^{1/2}.
\end{equation}

Observe that
\begin{equation*}
\sup_{s \in (0, s_{0}]} \int_{0}^{s} s^{1/4} (s-\sbr)^{-3/4} (\sbr)^{1/2} \, \frac{\ud \sbr}{\sbr} \leq C, \quad \sup_{\sbr \in (0, s_{0}]} \int_{\sbr}^{s_{0}} s^{1/4} (s-\sbr)^{-3/4} (\sbr)^{1/2} \, \frac{\ud s}{s} \leq C.
\end{equation*}

Therefore, by Schur's test, \eqref{eq:est4Duhamel:pf:2} is estimated by $\nrm{f(s)}_{\calL^{2}_{s}(0, s_{0}]}$ as desired. \qedhere
\end{proof}

Finally, we end this section with a simple lemma which is useful for substituting covariant derivatives by usual derivatives and vice versa.

\begin{lemma} \label{lem:dSub}
For $k \geq 1$, and $\alp$ be a multi-index of order $k$. Then the following schematic algebraic identities hold.
\begin{align}
	\covD^{(\alp)}_{x} \sgm 
	=& \rd_{x}^{(\alp)} \sgm + \sum_{\star} \calO_{\alp} (\rd_{x}^{(\ell_{1})} A, \rd_{x}^{(\ell_{2})} A, \cdots, \rd_{x}^{(\ell_{j})} A, \rd_{x}^{(\ell)} \sgm), \label{eq:dSub:cov2u} \\
	\rd_{x}^{(\alp)} \sgm 
	=& \covD^{(\alp)}_{x} \sgm + \sum_{\star} \calO_{\alp} (\rd_{x}^{(\ell_{1})} A, \rd_{x}^{(\ell_{2})} A, \cdots, \rd_{x}^{(\ell_{j})} A, \covD_{x}^{(\ell)} \sgm).\label{eq:dSub:u2cov}
\end{align}

In both cases, the summation is over all $1 \leq j \leq k$ and $0 \leq \ell_{1}, \ldots, \ell_{j} ,\ell \leq k-1$ such that
\begin{equation*}
	j + \ell_{1} + \cdots \ell_{j} + \ell = k.
\end{equation*}
\end{lemma}

\begin{proof} 
In the case $k=1$, both \eqref{eq:dSub:cov2u} and \eqref{eq:dSub:u2cov} follow from the simple identity 
\begin{equation*}
	\covD_{i} \sgm = \rd_{i} \sgm + \LieBr{A_{i}}{\sgm}.
\end{equation*}

The cases of higher $k$ follow from a simple induction argument, using Leibniz's rule. We omit the proof. \qedhere
\end{proof}

%==========

\section{Analysis of covariant parabolic equations} \label{sec:covParabolic}
\subsection{Covariant parabolic equations of \eqref{eq:dYMHF}}
Let $I \subset \bbR$ be an interval, and consider a smooth solution $A_{\bfa}$ to the \emph{dynamic Yang-Mills heat flow}
\begin{equation*}  \tag{dYMHF}
	F_{s \mu} = \covD^{\ell} F_{\ell \mu}, \quad \mu = 0, 1, 2, 3,
\end{equation*}
on $I \times \bbR^{3} \times [0,1]$. Note that these equations are a part of \eqref{eq:HPYM}; one could say that these are parabolic equations of the Hyperbolic-\emph{Parabolic}-Yang-Mills system.

Let us first derive the parabolic equation satisfied by $F_{\mu \nu}$, namely,
\begin{equation} \label{eq:covParabolic4F}
	\covD_{s} F_{\mu \nu} - \covD^{\ell} \covD_{\ell} F_{\mu \nu} = - 2 \LieBr{\tensor{F}{_{\mu}^{\ell}}}{F_{\nu \ell}}.
\end{equation}

We start with the \emph{Bianchi identity} 
\begin{equation} \label{eq:fullBianchi}
	\covD_{\bfa} F_{\bfb \bfc} + \covD_{\bfb} F_{\bfc \bfa} + \covD_{\bfc} F_{\bfa \bfb} = 0, 
\end{equation}
which easily follows from the formula $F_{\bfa \bfb} = \rd_{\bfa} A_{\bfb} - \rd_{\bfb} A_{\bfa} + \LieBr{A_{\bfa}}{A_{\bfb}}$. Taking the case $\bfa = s$, $\bfb = \mu$ and $\bfc = \nu$, we arrive at the identity
\begin{equation*}
	\covD_{s} F_{\mu \nu} = \covD_{\mu} F_{s \nu} - \covD_{\nu} F_{s \mu}.
\end{equation*}

Since we are considering a solution to \eqref{eq:dYMHF}, the right-hand side is equal to $\covD_{\mu} \covD^{\ell} F_{\ell \nu} - \covD_{\nu} \covD^{\ell} F_{\ell \mu}$. Commuting the covariant derivatives and applying 
\begin{equation*}
	\covD_{\mu} F_{\ell \nu} - \covD_{\nu} F_{\ell \mu} = \covD_{\ell} F_{\mu \nu},
\end{equation*}
which is again a consequence of the Bianchi identity, we arrive at \eqref{eq:covParabolic4F}.

Next, let us derive covariant parabolic equations satisfied by higher \emph{covariant} derivatives of $\bfF$. Given a $\LieAlg$-valued tensor $B$, we compute
\begin{align*}
	\covD_{i} \covD_{s} B - \covD_{i} \covD^{\ell} \covD_{\ell} B
	= & \covD_{s} \covD_{i} B - \covD^{\ell} \covD_{\ell} \covD_{i} B - 2 \LieBr{\tensor{F}{_{i}^{\ell}}}{\covD_{\ell} B}.
\end{align*}

Concisely, $\LieBr{\covD_{i}}{\covD_{s} - \covD^{\ell} \covD_{\ell}} B = \calO(F, \covD_{x} B)$. Using this, it is not difficult to prove the following proposition.

\begin{proposition} [Covariant parabolic equations of \eqref{eq:dYMHF}] \label{prop:covParabolic}
Let $A_{\bfa}$ be a solution to \eqref{eq:dYMHF}. Then the curvature $2$-form $F_{\mu \nu}$ satisfies the following parabolic equation.
\begin{equation} \label{eq:BWeq:1}
	(\covD_{s} - \covD^{\ell} \covD_{\ell}) F_{\mu \nu} = - 2 \LieBr{\tensor{F}{_{\mu}^{\ell}}}{F_{\nu \ell}}.
\end{equation}

The covariant derivatives of $F_{\mu \nu}$ satisfy the following schematic equation.
\begin{equation} \label{eq:BWeq:2}
	(\covD_{s} - \covD^{\ell} \covD_{\ell}) (\covD_{x}^{(k)} \bfF) = \sum_{j =0}^{k} \calO(\covD^{(j)}_{x} \bfF, \covD^{(k-j)}_{x} \bfF).
\end{equation}
\end{proposition}

Proceeding in the same manner for the non-temporal components $A_{a}$ ($a=x^{1}, x^{2}, x^{3}, s$), which solve the system $F_{si} = \covD^{\ell} F_{\ell i}$ on their own, we may derive the following equation for $\covD_{x}^{(k)} F_{ij}$:
\begin{equation} \label{eq:BWeq:3}
	(\covD_{s} - \covD^{\ell} \covD_{\ell}) (\covD_{x}^{(k)} F) = \sum_{j =0}^{k} \calO(\covD^{(j)}_{x} F, \covD^{(k-j)}_{x} F).
\end{equation}

\subsection{Estimates for the covariant parabolic equations}
Let us fix a time $t \in I$. Let us recall the definition \eqref{eq:YMenergy} of the Yang-Mills energy of $\bfF(t)$ at $s=0$:
\begin{equation*}
	\bfE(t) = \bfE[\bfF(t, s=0)] = \sum_{\mu <\nu} \frac{1}{2} \nrm{F_{\mu \nu}(t, s=0)}_{L^{2}_{x}}^{2}.
\end{equation*}

Recall that $\calD_{i} := s^{1/2} \covD_{i}$. The following proposition, which is proved by applying covariant techniques to \eqref{eq:BWeq:2}, is the analytic heart of this paper.

\begin{proposition}[Covariant parabolic estimates for $\bfF$] \label{prop:pEst4covF}
Let $I \subset \bbR$ be an interval, and $t \in I$. Suppose that $A_{\bfa}$ is a solution to \eqref{eq:dYMHF} on $I \times \bbR^{3} \times [0,1]$ such that $A_{\bfa} \in C^{\infty}_{t,s}(I \times [0,1], H^{\infty}_{x})$. There exists $\dlt > 0 $ such that if $\bfE(t) < \dlt$, then the following estimate holds for each integer $k \geq 1$.
\begin{equation} \label{eq:pEst4covF:1}
	\nrm{\calD_{x}^{(k-1)} \bfF(t)}_{\calL^{3/4, \infty}_{s} \calL^{2}_{x}(0,1]} + \nrm{\calD_{x}^{(k)} \bfF(t)}_{\calL^{3/4, 2}_{s} \calL^{2}_{x}(0,1]} \leq C_{k, \bfE(t)} \cdot \sqrt{\bfE(t)}.
\end{equation}
\end{proposition}

\begin{proof} 
Let us start with the cases $k = 1, 2$. Let $\subr \in (0, 1]$. Applying the energy integral estimate \eqref{eq:pEst4covHeat:1} with $\ell = 3/4$ to \eqref{eq:BWeq:1} and and $\ell = 3/4 + 1/2$ to \eqref{eq:BWeq:2} for $\covD_{x} \bfF$, we have
\begin{align*}
	&\nrm{\bfF}_{\calL^{3/4, \infty}_{s} \calL^{2}_{x}(0,\subr]} + \nrm{\calD_{x} \bfF}_{\calL^{3/4, 2}_{s} \calL^{2}_{x}(0,\subr]}
	\leq C \sqrt{\bfE} + C\nrm{\calO(\bfF, \bfF)}_{\calL^{3/4+1, 1}_{s} \calL^{2}_{x}(0, \subr]}, \\
	&\nrm{\calD_{x} \bfF}_{\calL^{3/4, \infty}_{s} \calL^{2}_{x}(0,\subr]} + \nrm{\calD_{x}^{(2)} \bfF}_{\calL^{3/4, 2}_{s} \calL^{2}_{x}(0,\subr]}
	\leq C \nrm{\calD_{x} \bfF}_{\calL^{3/4, 2}_{s} \calL^{2}_{x}(0,\subr]} + C \nrm{\calO(\calD_{x} \bfF,  \bfF )}_{\calL^{3/4+1, 1}_{s} \calL^{2}_{x}(0, \subr]}.
\end{align*}

No term at $s=0$ arises for the second estimate, as we have $\liminf_{s \to 0} s^{3/4} \nrm{\calD_{x} \bfF (s)}_{\calL^{2}_{x}(s)} = 0$ for a regular $\bfF$.

Combining the two inequalities, we obtain
\begin{equation*}
	\calB_{2}(\subr) \leq C \sqrt{\bfE} + C (\nrm{\calO(\bfF, \bfF)}_{\calL^{3/4+1, 1}_{s} \calL^{2}_{x}(0, \subr]} + \nrm{\calO(\calD_{x} \bfF, \bfF)}_{\calL^{3/4+1, 1}_{s} \calL^{2}_{x}(0, \subr]}),
\end{equation*}
where 
\begin{equation*}
\calB_{2}(\subr) := \sum_{k=1,2} (\nrm{\calD_{x}^{(k-1)} \bfF}_{\calL^{3/4, \infty}_{s} \calL^{2}_{x}(0,\subr]} + \nrm{\calD_{x}^{(k)} \bfF}_{\calL^{3/4, 2}_{s} \calL^{2}_{x}(0,\subr]} ).
\end{equation*}

Using H\"older and Corollary \ref{cor:covSob}, we see that
\begin{equation*}
	\nrm{\calO(\sgm_{1}, \sgm_{2})}_{L^{2}_{x}} \leq C \nrm{\sgm_{1}}_{L^{2}_{x}}^{1/2} \nrm{\covD_{x} \sgm_{1}}_{L^{2}_{x}}^{1/2} \nrm{\covD_{x} \sgm_{2}}_{L^{2}_{x}}.
\end{equation*}

By the Correspondence Principle, Lemma \ref{lem:Holder4Ls} and the fact that $\subr \leq 1$, we have
\begin{equation*}
\begin{aligned}
	\nrm{\calO(\bfF, \bfF)}_{\calL^{3/4+1,1}_{s}  \calL^{2}_{x}(0, \subr]} 
	\leq & C \subr^{1/4} \nrm{\bfF}_{\calL^{3/4,\infty}_{s} \calL^{2}_{x}(0, \subr]}^{1/2} \nrm{\calD_{x} \bfF}_{\calL^{3/4,2}_{s} \calL^{2}_{x}(0, \subr]}^{3/2} 
	\leq C \calB_{2}(\subr)^{2}.
\end{aligned}
\end{equation*}

Similarly, we also have
\begin{equation*}
\begin{aligned}
	\nrm{\calO(\calD_{x} \bfF, \bfF)}_{\calL^{3/4+1,1}_{s}  \calL^{2}_{x}(0, \subr]} 
	\leq & C \calB_{2}(\subr)^{2}.
\end{aligned}
\end{equation*}

Therefore, we obtain a bound of the form $\calB_{2}(\subr) \leq C \sqrt{\bfE} + C \calB_{2}(\subr)^{2}$, for every $\subr \in (0,1]$. Then by a simple bootstrap argument, the bound $\calB_{2}(1) \leq C \sqrt{\bfE}$ follows, which implies the desired estimate.

Let us turn to the case $k \geq 3$, which is proved by induction. Fix $k \geq 3$, and suppose, for the purpose of induction, that \eqref{eq:pEst4covF:1} holds for up to $k-1$. That is, defining
\begin{equation*}
	\calB_{k-1} := \sum_{j=1}^{k-1} \bb[ \nrm{\calD_{x}^{(j-1)} \bfF}_{\calL^{3/4,\infty}_{s} \calL^{2}_{x}(0,1]} + \nrm{\calD_{x}^{(j)} \bfF}_{\calL^{3/4,2}_{s} \calL^{2}_{x}(0,1]} \bb],
\end{equation*}
we will assume that $\calB_{k-1} \leq C_{k, \bfE} \cdot \sqrt{\bfE}$.

Applying the energy integral estimate \eqref{eq:pEst4covHeat:1} with $\ell = \frac{3}{4}  + \frac{k-1}{2}$ to \eqref{eq:BWeq:2} for $\covD_{x}^{(k-1)} \bfF$, we obtain
\begin{equation*}
	\nrm{\calD_{x}^{(k-1)} \bfF}_{\calL^{3/4,\infty}_{s} \calL^{2}_{x}} + \nrm{\calD_{x}^{(k)} \bfF}_{\calL^{3/4,2}_{s} \calL^{2}_{x}}
	\leq C\nrm{\calD_{x}^{(k-1)} \bfF}_{\calL^{3/4,2}_{s} \calL^{2}_{x}} + C \sum_{j=0}^{k-1} \nrm{\calO(\calD_{x}^{(j)} \bfF, \calD_{x}^{(k-1-j)} \bfF)}_{\calL^{3/4+1,1}_{s} \calL^{2}_{x}}
\end{equation*}
where we used the fact that $\liminf_{s \to 0} s^{3/4} \nrm{\calD_{x}^{(k-1)} \bfF(s)}_{\calL^{2}(s)} = 0$. 

The first term is bounded by $\calB_{k-1}$; therefore, \eqref{eq:pEst4covF:1} for $k$ will follow once we establish
\begin{equation} \label{eq:pEst4covF:pf:1}
	\sum_{j=0}^{k-1} \nrm{\calO(\calD_{x}^{(j)} \bfF, \calD_{x}^{(k-1-j)} \bfF)}_{\calL^{3/4+1,1}_{s} \calL^{2}_{x}} \leq C \calB_{k-1}^{2}.
\end{equation}

By Leibniz's rule, we see that \eqref{eq:pEst4covF:pf:1} follows once we establish the estimates
\begin{equation} \label{eq:pEst4covF:pf:2}
\left\{
\begin{aligned}
	& \nrm{\calO(\calD_{x} G_{1}, \calD_{x} G_{2})}_{\calL^{3/4+1,1}_{s} \calL^{2}_{x}} \leq C \calB_{2}^{2} \\
	& \nrm{\calO(G_{1}, \calD_{x}^{(2)} G_{2})}_{\calL^{3/4+1,1}_{s} \calL^{2}_{x}} + \nrm{\calO(\calD_{x}^{(2)}G_{1},  G_{2})}_{\calL^{3/4+1,1}_{s} \calL^{2}_{x}}
	\leq C \calB_{2}^{2},
\end{aligned}
\right.
\end{equation}
for any $\LieAlg$-valued 2-forms $G_{i} = G_{i}(x, s)$. Note that these roughly correspond to the case $k=3$ of \eqref{eq:pEst4covF:pf:1}.

Using the Correspondence Principle, Lemma \ref{lem:Holder4Ls}, and recalling the definition of $\calB_{k-1}$, it suffices to prove the estimates
\begin{align*}
	\nrm{\calO(\covD_{x} \sgm_{1}, \covD_{x} \sgm_{2})}_{L^{2}_{x}} 
	\leq & C \nrm{\covD_{x} \sgm_{1}}_{L^{2}_{x}}^{1/2} \nrm{\covD_{x}^{(2)} \sgm_{1}}_{L^{2}_{x}}^{1/2} \nrm{\covD_{x}^{(2)} \sgm_{2}}_{L^{2}_{x}}, \\
	\nrm{\calO(\sgm_{1}, \covD_{x}^{(2)} \sgm_{2})}_{L^{2}_{x}} 
	\leq & C \nrm{\covD_{x} \sgm_{1}}_{L^{2}_{x}}^{1/2} \nrm{\covD_{x}^{(2)} \sgm_{1}}_{L^{2}_{x}}^{1/2} \nrm{\covD_{x}^{(2)} \sgm_{2}}_{L^{2}_{x}}.
\end{align*}

The former is an easy consequence of H\"older, \eqref{eq:covSob:1} and \eqref{eq:covSob:2}, whereas the latter is proved similarly by applying H\"older, \eqref{eq:covSob:2} and \eqref{eq:covSob:3}. \qedhere
\end{proof}

Recalling $F_{s\nu} = \covD^{\ell} F_{\ell \nu}$, we obtain the following estimates for $F_{s\nu}$.

\begin{corollary} \label{cor:pEst4covFs}
Under the same hypotheses as Proposition \ref{prop:pEst4covF}, the following estimates hold for every integer $k \geq 0$.
\begin{align} 
	\nrm{\calD_{x}^{(k)} \bfF_{s}}_{\calL^{5/4, \infty}_{s} \calL^{2}_{x}(0,1]} + \nrm{\calD_{x}^{(k)} \bfF_{s}}_{\calL^{5/4, 2}_{s} \calL^{2}_{x}(0,1]} 
	\leq & C_{k, \bfE} \cdot \sqrt{\bfE}, \label{eq:pEst4covFs:1} \\
	\nrm{\calD_{x}^{(k)} \bfF_{s}}_{\calL^{5/4, \infty}_{s} \calL^{\infty}_{x}(0,1]} + \nrm{\calD_{x}^{(k)} \bfF_{s}}_{\calL^{5/4, 2}_{s} \calL^{\infty}_{x}(0,1]} 
	\leq & C_{k, \bfE} \cdot \sqrt{\bfE}. \label{eq:pEst4covFs:2}
\end{align}
\end{corollary}

\begin{proof} 
The $L^{2}$-type estimate \eqref{eq:pEst4covFs:1} follows immediately from Proposition \ref{prop:pEst4covF} by the relation $F_{s\nu} = \covD^{\ell} F_{\ell \nu}$. The  $L^{\infty}$-type estimate \eqref{eq:pEst4covFs:2} then follows from \eqref{eq:pEst4covFs:1} by \eqref{eq:covSob:3} of Corollary \ref{cor:covSob} (covariant Gagliardo-Nirenberg) and the Correspondence Principle. \qedhere
\end{proof}

The above discussion may be easily restricted to a connection 1-form $(A_{i}, A_{s})$ satisfying \eqref{eq:cYMHF}, i.e., $F_{si} = \covD^{\ell} F_{\ell i}$. Repeating the proof of Proposition \ref{prop:pEst4covF}, the following proposition easily follows.
\begin{proposition}[Covariant parabolic estimates for $F$] \label{prop:pEst4covFij}
Let $\dlt > 0$ be as in Proposition \ref{prop:pEst4covF}, and consider a solution $(A_{i}, A_{s})$ to the covariant Yang-Mills heat flow $F_{si} = \covD^{\ell} F_{\ell i}$ on $\bbR^{3} \times [0,1]$ such that $A_{i}, A_{s} \in C^{\infty}_{s}([0,1], H^{\infty}_{x})$. If $\bfB = \bfB[F(s=0)] < \dlt$, where $\bfB[F(s=0)]$ is the magnetic energy of $F(s=0)$ defined in \eqref{eq:Menergy}, then the following estimate holds for every integer $k \geq 1$.
\begin{equation} \label{eq:pEst4covFij:1}
	\nrm{\calD_{x}^{(k-1)} F}_{\calL^{3/4, \infty}_{s} \calL^{2}_{x}(0,1]} + \nrm{\calD_{x}^{(k)} F}_{\calL^{3/4, 2}_{s} \calL^{2}_{x}(0,1]} \leq C_{k, \bfB} \cdot \sqrt{\bfB}.
\end{equation}
\end{proposition}

\section{Analysis of Yang-Mills heat flows in the caloric gauge} \label{sec:YMHF}

\subsection{Analysis of the Yang-Mills heat flow} \label{subsec:YMHF}
In this subsection, we will consider the following IVP for the covariant Yang-Mills heat flow $F_{si} = \covD^{\ell} F_{\ell i}$ (which are the spatial components of \eqref{eq:dYMHF}) in the caloric gauge $A_{s} = 0$:
\begin{equation*} 
\left\{
\begin{aligned}
	& \rd_{s} A_{i} = \covD^{\ell} F_{\ell i}, \quad \hbox{ on } \bbR^{3} \times [0,s_{0}] \\
	& A_{i}(s=0) = \Aini_{i},
\end{aligned}
\right.
\end{equation*}
where $s_{0} > 0$. As this system is the original Yang-Mills heat flow \eqref{eq:YMHF}, we will refer to it as such. We will mainly be concerned with the class of \emph{regular initial data sets} and \emph{regular solutions} to \eqref{eq:YMHF}, which are defined as follows.

\begin{definition} \label{def:reg4YMHF}
We say that a connection 1-form $\Aini_{i}$ on $\bbR^{3}$ is a \emph{regular initial data set} for \eqref{eq:YMHF} if $\Aini_{i} \in H^{\infty}_{x}$. Furthermore, we say that a smooth solution $A_{i}$ to \eqref{eq:YMHF} defined on $\bbR^{3} \times [0,s_{0}]$ is a \emph{regular solution} to \eqref{eq:YMHF} if $A_{i} \in C^{\infty}_{s} ([0, s_{0}], H^{\infty}_{x})$.
\end{definition}

Our immediate goal is to establish a local well-posedness theorem (Theorem \ref{thm:implwp4YMHF}), where the interval of existence depends \emph{only} on the magnetic energy $\bfB[\Fini]$ of the initial data. The starting point of our analysis is Theorem 5.1 from \cite{Oh:6stz7nRe}, which is an $\dot{H}^{1}_{x}$ local existence statement. We restate the theorem below for the convenience of the reader.

\begin{theorem}[{\cite[Theorem 5.1]{Oh:6stz7nRe}}]  \label{thm:lwp4YMHF}
Consider the above initial value problem (IVP) for \eqref{eq:YMHF} with initial data $\Aini_{i} \in \dot{H}^{1}_{x}$ at $s=0$. Then the following statements hold.
\begin{enumerate}
\item There exists a number $s^{\star} = s^{\star} (\nrm{\Aini}_{\dot{H}^{1}_{x}}) > 0$, which is non-increasing in $\nrm{\Aini}_{\dot{H}^{1}_{x}}$, such that there exists a solution $A_{i} \in C_{s} ([0,s^{\star}],  \dot{H}^{1}_{x})$ to the IVP satisfying
\begin{equation} \label{eq:lwp4YMHF:1}
	\sup_{s \in [0,s^{\star}]} \nrm{A(s)}_{\dot{H}^{1}_{x}} \leq C  \nrm{\Aini}_{\dot{H}^{1}_{x}}.
\end{equation}

\item Let $\Aini'_{i} \in \dot{H}^{1}_{x}$ be another initial data set such that $\nrm{\Aini'}_{\dot{H}^{1}_{x}} \leq \nrm{\Aini}_{\dot{H}^{1}_{x}}$, and $A'$ the corresponding solution to the IVP on $[0, s^{\star}]$ given in (1). Then the following estimate for the difference $\dlt A := A - A'$ holds. 
\begin{equation} \label{eq:lwp4YMHF:2}
	\sup_{s \in [0,s^{\star}]} \nrm{\rd_{x}( \dlt A) (s)}_{\dot{H}^{1}_{x}} \leq C  \nrm{\dlt \Aini}_{\dot{H}^{1}_{x}}.
\end{equation}

\item If $\Aini_{i}(t)$ ($t \in I$) is a one parameter family of initial data such that $\Aini_{i} \in C^{\infty}_{t}(I, H^{\infty}_{x})$, then $A_{i} = A_{i}(t,x,s)$ given by (1) for each $t \in I$ satisfies $A_{i} \in C^{\infty}_{t,s}(I \times [0,s^{\star}], H^{\infty}_{x})$.
\end{enumerate}
\end{theorem}

According to our definition, the first statement of Part (3) of Theorem \ref{thm:lwp4YMHF} states that if $\Aini_{i}$ is a regular initial data set for \eqref{eq:YMHF}, then there exists a solution $A_{i}$ to \eqref{eq:YMHF} which is \emph{regular} in the sense of Definition \ref{def:reg4YMHF}.

\begin{remark}[Remark on the proof of Theorem \ref{thm:lwp4YMHF}] \label{}
Let us return to the perspective of viewing \eqref{eq:YMHF} as the covariant Yang-Mills heat flow $F_{si} = \covD^{\ell} F_{\ell i}$  with the caloric gauge condition $A_{s} = 0$ imposed. The problem with the caloric gauge is that the equations for $A_{i}$ are not strictly parabolic, but only weakly-parabolic. As discussed in the introduction, we can make the covariant Yang-Mills heat flow strictly parabolic by choosing a different gauge, namely the DeTurck gauge $A_{s} = \rd^{\ell} A_{\ell}$. Local well-posedness for $\dot{H}^{1}_{x}$ data (on some $s$-interval $[0, s^{\star}]$) then follows by standard parabolic theory. To come back to the caloric gauge $A_{s} = 0$, however, we must perform a gauge transform. The desired gauge transform can be obtained by solving the following ODE.
\begin{equation*}
\left\{
\begin{aligned}
	\rd_{s} U =& U A_{s} \quad \hbox{ on } \bbR^{3} \times [0,s^{\star}], \\
	U(s=0) =& \mathrm{Id}.
\end{aligned}
\right.
\end{equation*}

This ODE can be easily derived by the gauge transform formula $\widetilde{A}_{s} = U A_{s} U^{-1} - \rd_{s} U U^{-1}$. The initial condition $U(s=0) = \mathrm{Id}$ has been chosen to leave the initial data the same. It is then possible to show that the gauge transform is bounded on $C_{s}([0, s^{\star}], \dot{H}^{1}_{x})$, which proves Theorem \ref{thm:lwp4YMHF}. We remark also that the gauge transform is regular (in the sense of Definition \ref{def:reg4gt}) if the initial data set is (in the sense of Definition \ref{def:reg4YMHF}).

We add that this procedure ends up being the standard DeTurck trick, as in \cite{Donaldson:1985vh}, applied to \eqref{eq:YMHF}. 
\end{remark}

As a first step, we complement Theorem \ref{thm:lwp4YMHF} with the following uniqueness statement in the class of regular solutions.

\begin{lemma}[Uniqueness of regular solution to \eqref{eq:YMHF}] \label{lem:uni4YMHF}
Let $A_{i}, A'_{i}$ be two regular solutions to \eqref{eq:YMHF} on a common $s$-interval $J = [0, s_{0}]$. If their initial data coincide, i.e., $\Aini_{i}= \Aini'_{i}$, then so do the solutions, i.e., $A_{i} = A'_{i}$.
\end{lemma}

\begin{proof} 
%The main difficulty of the proof arises from the fact that the Yang-Mills heat flow $\rd_{s} A_{i} = \covD^{\ell} F_{\ell i}$ is not parabolic at the level of $A_{i}$. To resolve this issue, we will look at a system of coupled equations for $A_{i}, F_{ij}$ and $F_{si} := \rd_{s} A_{i}$.
%
By taking the difference between the parabolic equations satisfied by $F_{ij}$ and $F'_{ij}$, we obtain the following equation for $\dlt F_{ij} = F_{ij} - F'_{ij}$.
\begin{equation} \label{eq:uni4YMHF:pf:1}
	(\rd_{s} - \covD^{\ell} \covD_{\ell}) (\dlt F_{ij}) = - 2\LieBr{\tensor{F}{_{i}^{\ell}}}{\dlt F_{j \ell}} - 2  \LieBr{\tensor{\dlt F}{_{i}^{\ell}}}{F'_{j \ell}} + (\dlt \covD^{\ell} \covD_{\ell}) F'_{ij},
\end{equation}
where $\dlt \covD^{\ell} \covD_{\ell}$ is defined by
\begin{equation*}
	(\dlt \covD^{\ell} \covD_{\ell}) B := 2 \LieBr{\dlt A^{\ell}}{\rd_{\ell} B} + \LieBr{\rd^{\ell} (\dlt A_{\ell})}{B} + \LieBr{A^{\ell}}{\LieBr{\dlt A_{\ell}}{B}} + \LieBr{\dlt A^{\ell}}{\LieBr{A'_{\ell}}{B}}.
\end{equation*}

Next, recall that $F_{si} = \covD^{\ell} F_{\ell i}$ satisfies the equation
\begin{equation*}
	(\rd_{s} - \covD^{\ell} \covD_{\ell}) F_{si} = - 2\LieBr{\tensor{F}{_{s}^{\ell}}}{F_{i\ell}}.
\end{equation*}

Taking the difference between the equations for $F_{si}$ and $F'_{si}$, we obtain
\begin{equation} \label{eq:uni4YMHF:pf:2}
	(\rd_{s} - \covD^{\ell} \covD_{\ell}) (\dlt F_{si}) = - 2\LieBr{\tensor{F}{_{s}^{\ell}}}{\dlt F_{i\ell}} - 2\LieBr{\tensor{\dlt F}{_{s}^{\ell}}}{F'_{i\ell}} + (\dlt \covD^{\ell} \covD_{\ell}) F'_{si}.
\end{equation}

As $F_{si} = \rd_{s} A_{i}$, we have
\begin{equation} \label{eq:uni4YMHF:pf:3}
	\rd_{s} (\dlt A_{i}) = \dlt F_{si}.
\end{equation}

Finally, thanks to the special identity $\covD^{\ell} F_{s\ell} = 0$, we have the following equation for $\rd^{\ell} (\dlt A_{\ell})$. 
\begin{equation} \label{eq:uni4YMHF:pf:4}
	\rd_{s} \rd^{\ell} (\dlt A_{\ell}) = - \LieBr{A^{\ell}}{\dlt F_{s\ell}} - \LieBr{\dlt A^{\ell}}{F'_{s\ell}}.
\end{equation}

Now for each $s \in J$, let us define $\dlt \calB(s)$ to be
\begin{equation*}
	\dlt \calB(s) := \max_{i,j} \bb( \nrm{\dlt A_{i}(s)}_{L^{\infty}_{x}} + \nrm{\rd^{\ell}(\dlt A)_{\ell}(s)}_{L^{\infty}_{x}}+ \nrm{\dlt F_{ij}(s)}_{L^{\infty}_{x}} + \nrm{\dlt F_{si}(s)}_{L^{\infty}_{x}} \bb).
\end{equation*}

Let $\eps > 0$. Applying the weak maximum principle \eqref{eq:heatIneq4Duhamel:2} of Corollary \ref{cor:heatIneq4Duhamel} to \eqref{eq:uni4YMHF:pf:1}, \eqref{eq:uni4YMHF:pf:2} and integrating \eqref{eq:uni4YMHF:pf:3}, \eqref{eq:uni4YMHF:pf:4}, for each $s \in [0, s_{0} - \eps]$ we arrive at
\begin{equation*}
	\dlt \calB(s) < \dlt \calB(0) + (C+C^{2}) \int_{0}^{s} \dlt \calB(s) \, \ud s.
\end{equation*}
where 
\begin{equation*}
	C = \max_{i, j}(\nrm{A_{i}}_{L^{\infty}_{t,x,s}} + \nrm{F_{ij}}_{L^{\infty}_{t,x,s}} + \nrm{F_{si}}_{L^{\infty}_{t,x,s}}
	+ \nrm{A'_{i}}_{L^{\infty}_{t,x,s}} + \nrm{F'_{ij}}_{L^{\infty}_{t,x,s}} + \nrm{F'_{si}}_{L^{\infty}_{t,x,s}}),
\end{equation*}
with all norms taken over $I \times \bbR^{3} \times [0, s_{0}-\eps]$. Note that $C < \infty$, since $A_{\mu}, A'_{\mu}$ are regular. Then as $\dlt \calB(0) = 0$, by Gronwall's inequality, it follows that $\dlt \calB(s) = 0$ for all $[0, s_{0} -\eps]$. Taking $\eps \to 0$, we see that $A_{i} = A'_{i}$ on $I \times \bbR^{3} \times [0, s_{0})$. \qedhere
\end{proof}

\begin{theorem} [Improved local well-posedness for \eqref{eq:YMHF}] \label{thm:implwp4YMHF}
Consider the above IVP for \eqref{eq:YMHF} with a regular initial data set $\Aini_{i}$ (in the sense of Definition \ref{def:reg4YMHF}). Suppose furthermore that the initial magnetic energy is sufficiently small, i.e.,
\begin{equation*}
	\bfB[\Fini] = \frac{1}{2}\sum_{i < j} \nrm{\Fini_{ij}}^{2}_{L^{2}_{x}} < \dlt,
\end{equation*}
where $\Fini_{ij} := \rd_{i} \Aini_{j} -\rd_{j} \Aini_{i} + \LieBr{\Aini_{i}}{\Aini_{j}}$ and $\dlt > 0$ is the positive number as in Proposition \ref{prop:pEst4covF}. 

Then there exists a unique regular solution $A_{i}$ to the IVP on $\bbR^{3} \times [0, 1]$.
\end{theorem}

\begin{remark} 
Other constituents of a local well-posedness statement, such as continuous dependence on the data, can be proved by a minor modification of the proof below. Also, the statement can be extended to a rougher class of initial data and solutions by an approximation argument. We shall not provide proofs for these as they are not needed in the sequel.
\end{remark}

\begin{proof} 
By Theorem \ref{thm:lwp4YMHF} and Lemma \ref{lem:uni4YMHF}, there exists $s^{\star} > 0$ such that a unique regular solution $A_{i}$ to the IVP for \eqref{eq:YMHF} exists on $[0,s^{\star}]$ and obeys
\begin{equation} \label{eq:implwp4YMHF:pf:0}
	\sup_{0 \leq s \leq s^{\star}}\nrm{A(s)}_{L^{6}_{x}} \leq C \sup_{0 \leq s \leq s^{\star}} \nrm{A(s)}_{\dot{H}^{1}_{x}} \leq C \nrm{\Aini}_{\dot{H}^{1}_{x}}.
\end{equation}

We remark that the first inequality holds by Sobolev embedding.

Let us denote by $s_{\max}$ the largest $s$-parameter for which $A_{i}$ extends as a unique regular solution on $[0, s_{\max})$. We claim that under the hypothesis that $\bfB[\Fini] < \dlt$, the following statement holds:
\begin{equation} \label{eq:implwp4YMHF:pf:1}
\hbox{If } s_{\max} \leq 1\hbox{ then } \sup_{s \in [0, s_{\max})} \nrm{A_{i}}_{\dot{H}^{1}_{x}} < \infty. 
\end{equation}

If this claim is true, then we can apply Theorem \ref{thm:lwp4YMHF} and Lemma \ref{lem:uni4YMHF} to extend $A_{i}$ past $s_{\max}$ if $s_{\max} \leq 1$. Therefore, it follows that $s_{\max} > 1$.

Let us establish \eqref{eq:implwp4YMHF:pf:1}. The first step is to show that $\nrm{A(s)}_{L^{6}_{x}}$ does not blow up on $[0, s_{\max})$. By \eqref{eq:implwp4YMHF:pf:0}, it suffices to restrict our attention to $s > s^{\star}$; therefore, $s \in (s^{\star}, s_{\max})$. Since $s_{\max } \leq 1$, by Proposition \ref{prop:pEst4covFij} and Corollary \ref{cor:covSob}, we see that
\begin{equation*}
	\nrm{\rd_{s} A_{i}(s)}_{L^{6}_{x}} = \nrm{\covD^{\ell} F_{\ell i}(s)}_{L^{6}_{x}} \leq s^{-1} C_{\bfB} \cdot \sqrt{\bfB}.
\end{equation*}

Integrating from $s=s^{\star}$ and using \eqref{eq:implwp4YMHF:pf:0}, we arrive at
\begin{equation} \label{eq:implwp4YMHF:pf:2}
\sup_{s^{\star} < s < s_{\max}} \nrm{A(s)}_{L^{6}_{x}} \leq \nrm{\Aini}_{\dot{H}^{1}_{x}} + C_{\bfB} \cdot \abs{\log s^{\star}}  \sqrt{\bfB} < \infty.
\end{equation} 

Note that integrating from $s^{\star}$ allows us to bypass the issue of logarithmic divergence at $s=0$.

Next, let us show that $\nrm{A_{i}(s)}_{\dot{H}^{1}_{x}}$ does not blow up on $[0, s_{\max})$. Again, it suffices to consider $s \in (s^{\star}, s_{\max})$. Recall that $\covD_{x} \covD^{\ell} F_{\ell i} = \rd_{x} \covD^{\ell} F_{\ell i} + \LieBr{A}{\covD^{\ell} F_{\ell i}}$; thus by the triangle inequality and H\"older,
\begin{equation*}
	\nrm{\rd_{s} A_{i}(s)}_{\dot{H}^{1}_{x}} \leq \nrm{\covD_{x} \covD^{\ell} F_{\ell i}(s)}_{L^{2}_{x}} + \nrm{A(s)}_{L^{6}_{x}} \nrm{\covD^{\ell} F_{\ell i}(s)}_{L^{3}_{x}}.
\end{equation*}

Using Proposition \ref{prop:pEst4covFij} and Corollary \ref{cor:covSob}, we obtain
\begin{equation*}
	\nrm{\rd_{s} A_{i}(s)}_{\dot{H}^{1}_{x}} 
	\leq s^{-1} C_{\bfB} \cdot \sqrt{\bfB} + s^{-3/4} C_{\bfB} \cdot \sqrt{\bfB} \, \bb( \sup_{s^{\star} < s < s_{\max}} \nrm{A(s)}_{L^{6}_{x}} \bb).
\end{equation*}

Recalling \eqref{eq:implwp4YMHF:pf:1} and integrating from $s^{\star}$, we see that $\sup_{s^{\star} < s < s_{\max}}\nrm{A_{i}}_{\dot{H}^{1}_{x}} < \infty$, as desired. \qedhere
\end{proof}

For any regular initial data with finite magnetic energy, we can use scaling to make $\bfB(s=0) < \dlt$; thus, Theorem \ref{thm:implwp4YMHF} applies also to initial data with large magnetic energy. Furthermore, using the fact that the magnetic energy $\bfB(s)$ is non-increasing in $s$ under the Yang-Mills heat flow (which is formally obvious, as the Yang-Mills heat flow is the gradient flow of $\bfB$; see \cite{Rade:1992tu}), we can in fact iterate Theorem \ref{thm:implwp4YMHF} to obtain a unique global solution to the IVP, leading to an independent proof of the following classical result of \cite{Rade:1992tu}.

\begin{corollary}[R{\aa}de \cite{Rade:1992tu}] \label{cor:gwp4YMHF}
Consider the IVP for \eqref{eq:YMHF} with a regular initial data set $\Aini_{i}$ which possesses finite magnetic energy, i.e., $\bfB[\Fini] = (1/2) \sum_{i < j} \nrm{\Fini_{ij}}_{L^{2}_{x}} < \infty$. Then there exists a unique global regular solution $A_{i}$ to the IVP on $\bbR^{3} \times [0, \infty)$.

\end{corollary}

%-----

\subsection{Analysis of the dynamic Yang-Mills heat flow in the caloric gauge} \label{subsec:dYMHF}
Hereafter, we shall study the dynamic Yang-Mills heat flow \eqref{eq:dYMHF} in the caloric gauge $A_{s} = 0$. Writing out the left-hand side of \eqref{eq:dYMHF}, we obtain
\begin{equation} \label{eq:dYMHFcal}  
	\rd_{s} A_{\mu} = \covD^{\ell} F_{\ell \mu}, \quad \mu = 0,1,2,3.
\end{equation}

Let $I \subset \bbR$ be an interval. We will study the IVP associated to \eqref{eq:dYMHFcal} on $I \times \bbR^{3} \times [0,1]$, with the initial data given by
\begin{equation} 
	A_{\mu} (t,x, 0) = \Aini_{\mu} (t,x) \quad \hbox{ on } I \times \bbR^{3}.
\end{equation}

As in \S \ref{subsec:YMHF}, we will be focusing on \emph{regular} initial data sets and solutions to \eqref{eq:dYMHFcal}, whose definitions we give below.
\begin{definition} \label{def:reg4dYMHF}
Let $I \subset \bbR$ be an interval. We say that a connection 1-form $\Aini_{\mu} = \Aini_{\mu}(t,x)$ defined on $I \times \bbR^{3}$ is a \emph{regular initial data set} for \eqref{eq:dYMHF} in the caloric gauge if $\Aini_{i} \in C^{\infty}_{t}(I, H^{\infty}_{x})$. Furthermore, we say that a smooth solution $A_{\mu}$ to \eqref{eq:dYMHFcal} defined on $I \times \bbR^{3} \times [0,s_{0}]$ is a \emph{regular solution} to \eqref{eq:dYMHF} in the caloric gauge if $A_{\mu} \in C^{\infty}_{t,s}(I \times [0, s_{0}], H^{\infty}_{x})$.
\end{definition}

We begin with a uniqueness lemma for a regular solution to \eqref{eq:dYMHF} in the caloric gauge.
\begin{lemma}[Uniqueness of regular solution to \eqref{eq:dYMHF} in the caloric gauge] \label{lem:uni4dYMHF}
Let $A_{\mu}, A'_{\mu}$ be two regular solutions to \eqref{eq:dYMHF} in the caloric gauge on $I \times \bbR^{3} \times [0, s_{0}]$ for some $I \subset \bbR$ and $s_{0} > 0$. If their initial data coincide, i.e., $A_{\mu}(s=0)= A'_{\mu}(s=0)$ on $I \times \bbR^{3}$, then so do the solutions, i.e., $A_{\mu} = A'_{\mu}$ on $I \times \bbR^{3} \times [0, s_{0}]$.
\end{lemma}
\begin{proof} 
Note that the spatial components $A_{i}(t)$ satisfy \eqref{eq:YMHF} for each fixed $t$; therefore, by Lemma \ref{lem:uni4YMHF}, it follows that $A_{i} = A'_{i}$. As a consequence, we are only left to show $A_{0} = A'_{0}$.

As we already know that $A_{i} = A'_{i}$, observe that $\dlt F_{0i} = F_{0i} - F'_{0i}$ now obeys the simple equation
\begin{equation*}
	(\rd_{s} - \covD^{\ell} \covD_{\ell}) (\dlt F_{0i}) = -2 \LieBr{\tensor{\dlt F}{_{0}^{\ell}}}{F_{i\ell}}. 
\end{equation*}

Note furthermore that $\dlt F_{0i}(s=0) = 0$. Applying the weak maximum principle \eqref{eq:heatIneq4Duhamel:2} and using Gronwall's inequality as before, we see that $F_{0i} = F'_{0i}$ on $I \times \bbR^{3} \times [0, s_{0})$. Then $\rd_{s} A_{0} = \covD^{\ell} F_{\ell 0} = \rd_{s} A'_{0}$ everywhere. Since $A_{0} (s=0) = A'_{0}(s=0)$, it follows that $A_{0} = A'_{0}$ on $I \times \bbR^{3} \times [0, s_{0})$, which concludes the proof.  \qedhere
\end{proof}

We end this section with a well-posedness statement (by which we mean existence and uniqueness) for the IVP for \eqref{eq:dYMHF} in the caloric gauge with regular initial data set possessing finite energy $\bfE[0] < \infty$. In fact, the theorem itself concerns the case of small energy, but the theorem extends to the large energy case by scaling. 

\begin{theorem} [Local well-posedness for \eqref{eq:dYMHF} in the caloric gauge]  \label{thm:impLWP4dYMHF}
Let $I \subset \bbR$ be an interval and consider the above IVP for \eqref{eq:dYMHF} in the caloric gauge with a regular initial data set $\Aini_{\mu}$ (in the sense of Definition \ref{def:reg4dYMHF}). Suppose furthermore that the energy is uniformly small on $I$, i.e.,
\begin{equation*}
	\sup_{t \in I} \bfE[\overline{\bfF}(t)] = \sup_{t \in I} \frac{1}{2}\sum_{\mu < \nu} \nrm{\Fini_{\mu \nu}(t)}^{2}_{L^{2}_{x}} < \dlt,
\end{equation*}
where $\dlt > 0$ is the positive number as in Proposition \ref{prop:pEst4covF}. 
Then there exists a unique regular solution $A_{\mu}$ to the IVP on $I \times \bbR^{3} \times [0, 1]$. 
\end{theorem}
\begin{proof} 
As in the proof of Lemma \ref{lem:uni4dYMHF}, we begin with the observation that the spatial components $A_{i}(t)$ satisfy \eqref{eq:YMHF} for each fixed $t$. Observe furthermore that $\bfB[\Fini(t)] \leq \bfE[\overline{\bfF}(t)] < \dlt$, where $\dlt > 0$ is the small constant common to Propositions \ref{prop:pEst4covF} and \ref{prop:pEst4covFij}. Repeating the proof of Theorem \ref{thm:implwp4YMHF} (but this time making use of the second statement of Part (3) of Theorem \ref{thm:lwp4YMHF}), we obtain a smooth spatial 1-form $A_{i}$ on $I \times \bbR^{3} \times [0,1]$ solving \eqref{eq:YMHF} for each fixed $t$, such that $A_{i} \in C^{\infty}_{t,s}(I \times [0,1], H^{\infty}_{x})$. At this point, we are left to solve for $A_{0}$, making use of the equation
\begin{equation} \label{eq:impLWP4dYMHF:pf:0}
	\rd_{s} A_{0} = \covD^{\ell} F_{\ell 0}.
\end{equation}

Here, we follow Step 1 of \cite[Proof of Theorem 4.8]{Oh:6stz7nRe}. Note that $F_{0i}$ satisfies the covariant parabolic equation
\begin{equation} \label{eq:impLWP4dYMHF:pf:1}
	\rd_{s} F_{0i} - \covD^{\ell} \covD_{\ell} F_{0i} = 2 \LieBr{\tensor{F}{_{i}^{\ell}}}{F_{0\ell}}
\end{equation}

Since $A_{i}$ and $F_{ij}$ have already been solved for, we may view this as a linear system of parabolic equation for $F_{0 i}$ with smooth coefficients, which is amenable to standard techniques. Indeed, applying \cite[Proposition 5.7]{Oh:6stz7nRe}, there exists a unique regular solution $F'_{0i} = - F'_{i0}$ to \eqref{eq:impLWP4dYMHF:pf:1} on $I \times \bbR^{3} \times [0,1]$ with $F'_{0i}(s=0) = \Fini_{0i}$.  Then motivated by \eqref{eq:impLWP4dYMHF:pf:0}, we extend $\Aini_{0}$ to $A_{0}$ on $I \times \bbR^{3} \times [0,1]$ by solving
\begin{equation*}
\left\{
\begin{aligned}
	& \rd_{s} A_{0} = \covD^{\ell} F'_{\ell 0}\quad \hbox{ on } I \times \bbR^{3} \times [0,1], \\
	&A_{0}(s=0) = \Aini_{0} \quad \hbox{ on } I \times \bbR^{3}. 
\end{aligned}
\right.
\end{equation*}

By construction, it is easy to verify that $A_{0}  \in C^{\infty}_{t,s} (I \times [0,1], H^{\infty}_{x})$. With $A_{0}$ defined on $I \times \bbR^{3} \times [0,1]$, we may define the curvature components $F_{0i} = \rd_{0} A_{i} - \rd_{i} A_{0} + \LieBr{A_{0}}{A_{i}}$ and ask whether $F_{0i} = F'_{0i}$. This is indeed the case, by \cite[Lemma 6.1]{Oh:6stz7nRe}, from which it follows that $A_{0}$ satisfies \eqref{eq:impLWP4dYMHF:pf:0}. Combined with Lemma \ref{lem:uni4dYMHF}, we conclude that $A_{\mu}$ is the desired solution to \eqref{eq:dYMHF} in the caloric gauge. \qedhere
\end{proof}

\subsection{Substitution of covariant derivatives by usual derivatives} \label{subsec:dSub}
At several points below, we will need to transfer estimates for covariant derivatives to the corresponding estimates for usual derivatives. The purpose of this part is to develop a general technique for carrying out such procedures. Our starting point is the following proposition, which concerns estimates for the $L^{\infty}_{x}$ norm of $\bfA$.

To state the following proposition, we need the following definition.
\begin{equation} \label{eq:calIAlow}
	{}^{(\Alow)}\calI(t) := \sum_{k=1}^{31} \nrm{\rd_{t,x} \Alow(t)}_{\dot{H}^{k-1}_{x}}.
\end{equation}

In fact, this is a part of a larger norm $\calI(t)$, whose definition will be given in Section \ref{sec:pfOfIdEst}.

\begin{proposition} \label{prop:intEst4A:s}
Let $I \subset \bbR$ be an interval, $t \in I$, and consider a regular solution $A_{\mu}$ to \eqref{eq:dYMHF} in the caloric gauge $A_{s} = 0$ on $I \times \bbR^{3} \times [0,1]$. Suppose that $\bfE(t) = \bfE[\bfF(t, s=0)] < \dlt$, where $\dlt > 0$ is the small constant in Proposition \ref{prop:pEst4covF}. Then the following estimate holds for all $0 \leq k \leq 29$.
\begin{equation} \label{eq:intEst4A:s:0}
\nrm{\nb_{x}^{(k)} \bfA(t)}_{\calL^{1/4, \infty}_{s} \calL^{\infty}_{x}(0,1]} \leq \calIAlow(t) + C_{k, \bfE(t), \calIAlow(t)} \cdot \sqrt{\bfE(t)}.
\end{equation}
\end{proposition}

\begin{proof} 
%\comment{Use Lemma \ref{lem:dSub} and Corollary \ref{cor:pEst4covFs}.} 
Henceforth, we shall fix $t \in I$ and omit writing $t$. By the caloric gauge condition $A_{s} = 0$, we have the relation $\rd_{s} A_{\nu} = F_{s \nu}$, where the latter can be controlled by Corollary \ref{cor:pEst4covFs}. Observe furthermore that
\begin{equation*}
	\nrm{\rd_{x}^{(k)} \bfAlow}_{L^{\infty}_{x}} \leq C \, \calIAlow
\end{equation*}
for $0 \leq k \leq 29$, by Sobolev (or Gagliardo-Nirenberg). Now, the idea is to use the fundamental theorem of calculus of control $\rd_{x}^{(k)} \bfA(s)$ for $0 < s \leq 1$.

We will proceed by induction on $k$. Let us start with the case $k=0$. By the fundamental theorem of calculus and Minkowski's inequality, we have
\begin{equation*}
	s^{1/4} \nrm{\bfA(s)}_{L^{\infty}_{x}} \leq s^{1/4} \nrm{\bfAlow}_{L^{\infty}_{x}} + \int_{s}^{1} (s/s')^{1/4} (s')^{5/4} \nrm{\bfF_{s}(s')}_{L^{\infty}_{x}} \, \frac{\ud s'}{s'}.
\end{equation*}

As remarked earlier, the first term on the right-hand side may be estimated by $\calIAlow$ uniformly in $s \in (0,1]$. For the second term, we apply \eqref{eq:pEst4covFs:2} of Corollary \ref{cor:pEst4covFs} and estimate $(s')^{5/4} \nrm{\bfF_{s}(s')}_{L^{\infty}_{x}}$ by $C_{k, \bfE} \cdot \sqrt{\bfE}$. The case $k=0$ of \eqref{eq:intEst4A:s:0} follows, since
\begin{equation*}
\sup_{0 < s \leq 1} \int_{s}^{1} (s/s')^{1/4} \, \frac{\ud s'}{s'} \leq C < \infty.
\end{equation*} 

Next, for the purpose of induction, assume that \eqref{eq:intEst4A:s:0} holds for $0, 1, \ldots, k-1$, where $1 \leq k \leq 29$. Taking $\rd^{(k)}_{x}$ of $\rd_{s} A_{\nu} = F_{s \nu}$, using the fundamental theorem of calculus, Minkowski's inequality and multiplying both sides by $s^{1/4+k/2}$, we arrive at
\begin{equation*}
	s^{1/4} \nrm{\nb^{(k)}_{x} \bfA(s)}_{L^{\infty}_{x}} 
	\leq s^{1/4+k/2} \nrm{\rd_{x}^{(k)} \bfAlow}_{L^{\infty}_{x}}
	+ \int_{s}^{1} (s/s')^{1/4+k/2} (s')^{5/4} \nrm{\nb_{x}^{(k)} \bfF_{s}(s')}_{L^{\infty}_{x}} \, \frac{\ud s'}{s'}.
\end{equation*}

Once we establish
\begin{equation} \label{eq:intEst4A:s:pf:1}
	\sup_{0 < s \leq 1} \nrm{s^{5/4} \nb_{x}^{(k)} \bfF_{s}}_{\calL^{\infty}_{x}} \leq C_{k, \bfE, \calIAlow} \cdot \sqrt{\bfE},
\end{equation}
then proceeding as in the previous case, \eqref{eq:intEst4A:s:0} for $k$ follows, which completes the induction.

Fix $0 < s \leq 1$. Applying \eqref{eq:dSub:u2cov} of Lemma \ref{lem:dSub} and multiplying both sides by $s^{5/4+ k/2}$, we see that
\begin{equation*}
	s^{5/4} \nb_{x}^{(k)} \bfF_{s}(s) = s^{5/4} \calD_{x}^{(k)} \bfF_{s}(s) + \sum_{\star} s^{j/4} \calO(s^{1/4} \nb_{x}^{(\ell_{1})} A, \cdots, s^{1/4} \nb_{x}^{(\ell_{j})} A, s^{5/4} \calD_{x}^{(\ell)} \bfF_{s}),
\end{equation*}
where the range of the summation is as specified in Lemma \ref{lem:dSub}. Let us take the $L^{\infty}_{x}$-norm of both sides; by the triangle inequality and \eqref{eq:pEst4covFs:2} of Corollary \ref{cor:pEst4covFs}, it suffices to control 
\begin{equation*}
	s^{j/4} \nrm{\calO(s^{1/4} \nb_{x}^{(\ell_{1})} A, \cdots, s^{1/4} \nb_{x}^{(\ell_{j})} A, s^{5/4} \calD_{x}^{(\ell)} \bfF_{s})}_{L^{\infty}_{x}}
\end{equation*}
for each summand of $\sum_{\star}$. Observe that we gained an extra power of $s^{1/4}$ for each factor of $A_{i}$ replacing $\rd_{i}$, thanks to the sub-criticality of the problem at hand. We will throw away these extra powers, which is okay as $0 < s \leq 1$. Observe that $0 \leq \ell_{1}, \ldots, \ell_{j} \leq k-1$; therefore, by the induction hypothesis, we have
\begin{equation*}
\nrm{s^{1/4} \nb_{x}^{(\ell_{1})} A(s)}_{L^{\infty}_{x}}, \cdots, \nrm{s^{1/4} \nb_{x}^{(\ell_{j})} A(s)}_{L^{\infty}_{x}} \leq C_{k, \bfE, \calIAlow}.
\end{equation*}

Note furthermore that $\nrm{s^{5/4} \calD_{x}^{(\ell)} \bfF_{s}(s)}_{L^{\infty}_{x}} \leq C_{\ell, \bfE} \cdot \sqrt{\bfE}$ by Corollary \ref{cor:pEst4covFs}. Hence, by H\"older, each summand may be estimated by $C_{k, \bfE, \calIAlow}$, and thus \eqref{eq:intEst4A:s:pf:1} follows.
\end{proof}

As a consequence, we obtain the following corollary which allows us to easily switch estimates for covariant derivatives to those for usual derivatives.

\begin{corollary} [Substitution of covariant derivatives by usual derivatives] \label{cor:cov2u}
Assume that the hypotheses of Proposition \ref{prop:intEst4A:s} holds. Let $\sgm$ be a $\LieAlg$-valued function on $\set{t} \times \bbR^{3} \times (0,1]$, $m \geq 0$ an integer, $b \geq 0, 1 \leq p, r \leq \infty$. Suppose that there exists $D > 0$ such that the estimate
\begin{equation} \label{eq:cov2u:0}
	\nrm{\calD_{x}^{(k)} \sgm}_{\calL^{b, p}_{s} \calL^{r}_{x}} \leq D
\end{equation}
holds for $0 \leq k \leq m$. Then we have
\begin{equation} \label{eq:cov2u:1}
	\nrm{\nb_{x}^{(k)} \sgm}_{\calL^{b, p}_{s} \calL^{r}_{x}} \leq C_{\calIAlow(t), \bfE(t)} \cdot D
\end{equation}
for $0 \leq k \leq \min(m, 30)$.
\end{corollary}

\begin{proof} 
We will again omit $t$ in this proof. The case $k = 0$ is obvious; we thus fix $1 \leq k \leq \min(m, 30)$. Using \eqref{eq:dSub:u2cov} of Lemma \ref{lem:dSub} to $\sgm$ and multiplying by $s^{b+k/2}$, we get
\begin{equation*}
	s^{b} \nb_{x}^{(k)} \sgm(s) = s^{b} \calD_{x}^{(k)} \sgm(s) + \sum_{\star} s^{j/4} \calO(s^{1/4} \nb_{x}^{(\ell_{1})} A(s), \cdots, s^{1/4} \nb_{x}^{(\ell_{j})} A(s), s^{b} \calD_{x}^{(\ell)} \sgm(s)),
\end{equation*}
where the range of summation $\sum_{\star}$ is as specified in Lemma \ref{lem:dSub}. Taking the $\calL^{p}_{s} \calL^{r}_{x}$ norm of both sides, applying the triangle inequality and using \eqref{eq:cov2u:0} to estimate $\nrm{s^{b} \calD_{x}^{(k)} \sgm(s)}_{\calL^{p}_{s} \calL^{r}_{x}} = \nrm{\calD_{x}^{(k)} \sgm}_{\calL^{b, p}_{s} \calL^{r}_{x}} \leq D$, we are left to establish
\begin{equation} \label{eq:cov2u:pf:1}
s^{j/4} \nrm{\calO(s^{1/4} \nb_{x}^{(\ell_{1})} A(s), \cdots, s^{1/4} \nb_{x}^{(\ell_{j})} A(s), s^{b} \calD_{x}^{(\ell)} \sgm(s))}_{\calL^{p}_{s} \calL^{r}_{x}} \leq C_{\calIAlow} \cdot D
\end{equation}
for each summand in $\sum_{\star}$. Note that we have an extra power of $s^{j/4}$, which we can just throw away (as $0 < s \leq 1$). Let us use H\"older to put each $s^{1/4} \nb_{x}^{(\ell_{i})} A(s)$ in $\calL^{\infty}_{s} \calL^{\infty}_{x}$ and $s^{b} \calD_{x}^{(\ell)} \sgm(s)$ in $\calL^{p}_{s} \calL^{r}_{x}$. Then using Proposition \ref{prop:intEst4A:s} (this is possible since $k \leq 30$) and \eqref{eq:cov2u:0} to control the respective norms, we obtain \eqref{eq:cov2u:pf:1}. \qedhere
\end{proof}
%\begin{lemma} \label{lem:intEst4A0:s}
%Under the same assumptions as Proposition \ref{prop:pEst4covF}, the following estimate holds for all $k \geq 0$.
%\begin{equation*}
%\nrm{\calD_{x}^{(k)} A_{0}}_{\calL^{0, \infty}_{s} \calL^{\infty}_{x}(0,1]}
%\leq C_{k, \bfE} \cdot \bfE + \underline{\calI}
%\end{equation*}
%\end{lemma}

%In this section, we begin by deriving estimates for $A_{\mu}$ which will be used throughout the rest of the paper. As an immediate consequence, we give a proof of the global existence of the Yang-Mills heat flow.

\section{Transformation to the caloric-temporal gauge and estimates at $t=0$: \\ Proof of Theorem \ref{thm:idEst}} \label{sec:pfOfIdEst}
The purpose of this section is to prove Theorem \ref{thm:idEst}. Let us begin by giving the precise definition of the norm $\calI(t)$, which was mentioned in Section \ref{sec:reduction}.

For a solution $A_{\bfa}$ to \eqref{eq:HPYM} on $I \times \bbR^{3} \times [0,1]$, define $\calI(t) := {}^{(F_{s})} \calI(t) + {}^{(\Alow)} \calI(t)$ $(t \in I)$, where
\begin{align*}
	& {}^{(F_{s})} \calI(t) := \sum_{k=1}^{10} \bb[ \nrm{\nb_{t,x} F_{s}(t)}_{\calL^{5/4,\infty}_{s} \dot{\calH}^{k-1}_{x}} + \nrm{\nb_{t,x} F_{s}(t)}_{\calL^{5/4,2}_{s} \dot{\calH}^{k-1}_{x}}\bb], \\
	& {}^{(\Alow)}\calI(t) := \sum_{k=1}^{31} \nrm{\rd_{t,x} \Alow(t)}_{\dot{H}^{k-1}_{x}}.
\end{align*}

A key ingredient for proving Theorem \ref{thm:idEst} is Theorem \ref{thm:impLWP4dYMHF}, which was proved in the previous section. The remaining analytic ingredients, on the other hand, are mostly contained in \cite[Theorem 4.8]{Oh:6stz7nRe}. Let us give a simplified version of \cite[Theorem 4.8]{Oh:6stz7nRe}, which suffices for our purpose and can be easily read off from the original version.
\begin{theorem}[{\cite[Theorem 4.8]{Oh:6stz7nRe}}, simplified] \label{thm:oldIdEst}
Let $0 < T \leq 1$, and $\Atemp_{\mu}$ be a regular solution to \eqref{eq:hyperbolicYM} in the temporal gauge $\Atemp_{0} = 0$ on $(-T, T) \times \bbR^{3}$ with the initial data $(\Aini_{i}, \Eini_{i})$ at $t=0$. Define $\calIini := \nrm{\Aini}_{\dot{H}^{1}_{x}} + \nrm{\Eini}_{L^{2}_{x}}$. If
\begin{equation} \label{eq:ThmA:hypothesis}
	\sup_{t \in (-T, T)} \nrm{\Atemp(t)}_{\dot{H}^{1}_{x}} < \dlt_{P}
\end{equation}
where $\dlt_{P}$ is a small absolute constant, then the following statements hold.
\begin{enumerate}
\item There exists a regular gauge transform $V = V(t,x)$ on $(-T, T) \times \bbR^{3}$ (in the sense of Definition \ref{def:reg4gt}) and a regular solution $A_{\bfa}$ to \eqref{eq:HPYM} such that 
\begin{equation}  \label{eq:ThmA:0}
	A_{\mu}(s=0) = V (\Atemp_{\mu}) V^{-1} - \rd_{\mu} V V^{-1}.
\end{equation}

\item Furthermore, the solution $A_{\bfa}$ satisfies the caloric-temporal gauge condition, i.e., $A_{s} = 0$ everywhere and $\Alow_{0} = 0$.

%\item Let $\Atemp'_{\mu}$ be another regular solution to the Yang-Mills equation in the temporal gauge with the initial data $(\Aini'_{i}, \Eini'_{i})$ satisfying $\nrm{\Aini}_{\dot{H}^{1}_{x}} + \nrm{\Eini}_{L^{2}_{x}} \leq \calIini$ and \eqref{eq:idEst:hypothesis}. Let $A'_{\bfa}$ be the solution to \eqref{eq:HPYM} in the caloric-temporal gauge obtained from $\Atemp'_{\mu}$ by (1) and (2).  Then 
%where $\dlt \calIini := \nrm{\dlt \Aini}_{\dot{H}^{1}_{x}} + \nrm{\dlt \Eini}_{L^{2}_{x}}$. 
%We furthermore have
%\begin{equation} \label{eq:idEst:2}
%	\nrm{\Alow(t=0)}_{L^{3}_{x}} \leq D^{1}(\calIini) + \nrm{\Aini}_{L^{3}_{x}}, \quad \nrm{\dlt \Alow(t=0)}_{L^{3}_{x}} \leq D^{0}(\calIini) \, \dlt \calIini + \nrm{\dlt \Aini}_{L^{3}_{x}}.
%\end{equation}

%\item Let $V'$ be the gauge transform obtained from $\Atemp'_{i}$ by (1), and let us write $\Vini := V(t=0), \Vini' := V'(t=0)$. 

\item The following initial data estimate holds.
\begin{equation*} 
	\calI(0) \leq C_{\calIini} \cdot \calIini.
\end{equation*}

\item For $\Vini := V(t=0)$, the following estimates hold.
\begin{equation*} 
	\nrm{\Vini}_{L^{\infty}_{x}} \leq C_{\calIini}, \quad  \nrm{\rd_{x} \Vini}_{L^{3}_{x}} + \nrm{\rd_{x}^{(2)} \Vini}_{L^{2}_{x}} \leq C_{\calIini} \cdot \calIini.
\end{equation*}
	The same estimates with $\Vini$ replaced by $\Vini^{-1}$, respectively, also hold.
%\begin{equation} \label{eq:idEst:4}
%	\nrm{\dlt \Vini}_{L^{\infty}_{x}} + \nrm{\rd_{x} \dlt \Vini}_{L^{3}_{x}} + \nrm{\rd_{x}^{(2)} \dlt \Vini}_{L^{2}_{x}} \leq C_{\calIini} \cdot \dlt \calIini.
%\end{equation}
\end{enumerate}
%The same estimates with $\Vini$ and $\dlt \Vini$ replaced by $\Vini^{-1}$ and $\dlt \Vini^{-1}$, respectively, also hold.

\end{theorem}

With Theorem \ref{thm:oldIdEst} in hand, we are ready to give a proof of Theorem \ref{thm:idEst}.
 
\begin{proof} [Proof of Theorem \ref{thm:idEst}]
We begin with a regular solution $\Atemp_{\mu}$ to the Yang-Mills equations $\covD^{\mu} F_{\nu \mu} = 0$ in the temporal gauge, defined on $(-T_{0}, T_{0}) \times \bbR^{3}$. Thanks to the regularity assumption, note that $ \Atemp_{\mu} \in C^{\infty}_{t}((-T_{0}, T_{0}), H^{\infty}_{x})$.

\pfstep{Step 1. Construction of regular solution to \eqref{eq:HPYM} in caloric-temporal gauge}
Recall the hypothesis \eqref{eq:idEst:hyp}. By smoothness in $t$ and conservation of energy, respectively, it follows that
\begin{equation} \label{eq:idEst:pf:0}
	\sup_{t \in (- \eps_{0}, \eps_{0})} \nrm{\Atemp}_{\dot{H}^{1}} < \dlt_{P}, \quad \sup_{t \in (-T_{0}, T_{0})}\bfE[\bfF^{\dagger}(t)] < \dlt,
\end{equation}
for some small $\eps_{0} > 0$. The second smallness condition allows us to apply Theorem \ref{thm:impLWP4dYMHF}, from which we obtain a unique regular solution $\widetilde{A}_{\mu}$ to \eqref{eq:dYMHFcal} on $(-T_{0}, T_{0}) \times \bbR^{3} \times [0,1]$. 
Supplied with $\widetilde{A}_{s} = 0$, $\widetilde{A}_{\bfa} = (\widetilde{A}_{\mu}, \widetilde{A}_{s})$ is a solution to \eqref{eq:HPYM}. We will apply a gauge transform $V = V(t,x,s)$ to $\widetilde{A}_{\bfa}$ to enforce the caloric-temporal gauge condition. Let us denote the resulting connection coeffients $A_{\bfa}$, i.e.,
\begin{equation*}
	A_{\bfa} := V \widetilde{A}_{\bfa} V^{-1} - \rd_{\bfa} V V^{-1}.
\end{equation*}

In order for $A_{\bfa}$ to be in the caloric-temporal gauge, we need a gauge transform $V$ which is (i) independent of $s$ (to keep $A_{s} =0$) and (ii) makes $\Alow_{0} = 0$. These two requirements are in fact equivalent (once one assumes enough regularity of $V$) to $V$ solving the ODE
\begin{equation*}
\left\{
\begin{aligned}
& \rd_{t} V = V \, \widetilde{\Alow}_{0}, \\
& V(t=0) = \Vini,
\end{aligned}
\right.
\end{equation*}
where $\widetilde{\Alow}_{0} := \widetilde{A}_{0}(s=1)$ and $\Vini$ is a gauge transform on $\bbR^{3}$, to be specified in Step 2 in accordance with Theorem \ref{thm:oldIdEst}. 

\pfstep{Step 2. Application of Theorem \ref{thm:oldIdEst}}

The next step of the proof is to apply Theorem \ref{thm:oldIdEst} to choose $\Vini$ and furthermore obtain a quantitative estimate for $\calIAlow(0)$. Thanks to the first inequality of \eqref{eq:idEst:pf:0}, we may apply Theorem \ref{thm:oldIdEst} on the time interval $(-\eps_{0}, \eps_{0})$. Let us mark the objects obtained from Theorem \ref{thm:oldIdEst} with a prime, i.e., $A'_{\mu}$, $V'$ and $\Vini'$. Consider $\widetilde{A}'_{\mu} = \widetilde{A}'_{\mu}(t,x,s)$ defined by
\begin{equation*}
	\widetilde{A}'_{\mu} := (V')^{-1} A'_{\mu} V' - \rd_{\mu} (V')^{-1} V',
\end{equation*}
where we remind the reader that $V' = V'(t,x)$, $(V')^{-1} = (V')^{-1}(t,x)$ are independent of $s$.

Note that $\widetilde{A}'_{\mu}$ is a regular solution to \eqref{eq:dYMHF} in the caloric gauge, as is $\widetilde{A}_{\mu}$. Moreover, their initial data sets coincide (both being $\Atemp_{\mu}$). By the uniqueness lemma (Lemma \ref{lem:uni4dYMHF}), we conclude that $\widetilde{A}_{\mu} = \widetilde{A}'_{\mu}$ on the set on which $\widetilde{A}'_{\mu}$ is defined, i.e.,
\begin{equation*}
\widetilde{A}_{\mu} = (V')^{-1} A'_{\mu} V' - \rd_{\mu} (V')^{-1} V'
\end{equation*}
on $(-\eps_{0}, \eps_{0}) \times \bbR^{3} \times [0,1]$. As $\Alow'_{0} =0$, we also see that
\begin{equation*}
\left\{
\begin{aligned}
& \rd_{t} V' = V' \widetilde{\underline{A}}_{0}, \\
& V'(t=0) = \Vini',
\end{aligned}
\right.
\end{equation*}
on $(-\eps_{0}, \eps_{0}) \times \bbR^{3}$. 

At this point, let us make the choice $\Vini = \Vini'$. Then the previous ODE is exactly that satisfied by $V$. Therefore, by uniqueness for ODE with smooth coefficients, $V = V'$ on $(-\eps_{0}, \eps_{0}) \times \bbR^{3}$, and hence we conclude that
\begin{equation*}
	A_{\mu} = A'_{\mu}
\end{equation*}
on $(-\eps_{0}, \eps_{0}) \times \bbR^{3} \times [0,1]$. From Parts (3), (4) of Theorem \ref{thm:oldIdEst}, the quantitative estimates in \eqref{eq:idEst:1} follow. 
Moreover, it is not difficult to show that $V$ is a regular gauge transform on $(-T_{0}, T_{0}) \times \bbR^{3}$. It also follows that $A_{\mu}$ is a regular solution on $(-T_{0}, T_{0}) \times \bbR^{3} \times [0,1]$, since $\widetilde{A}_{\mu}$ was regular. This completes the proof of Theorem \ref{thm:idEst}. \qedhere
\end{proof}

%\comment{This relies crucially on the smoothing properties of (YMHF) for $A$, which is false. Explain how to get around this issue by a variant of deTurck's trick.}

%==========

\section{Fixed-time estimates by $\bfE$ : Proof of Theorem \ref{thm:ctrlByE}}
Our aim in this section is to prove Theorem \ref{thm:ctrlByE}. Splitting $\calI(t) = \calIAlow(t) + \calIFs(t)$, we will reduce the theorem to establishing two inequalities, namely \eqref{eq:ctrlByE:Alw:0} and \eqref{eq:ctrlByE:Fs:0} of Propositions \ref{prop:ctrlByE:Alw} and \ref{prop:ctrlByE:Fs}, respectively. 

Throughout this section, we will be concerned with a regular solution $A_{\bfa}$ to \eqref{eq:HPYM} in the caloric-temporal gauge on $(-T_{0}, T_{0}) \times \bbR^{3} \times [0,1]$ ($T_{0} > 0$), which satisfies $\calI(0) = \calIAlow(0) +  \calIFs(0) < \infty$ and $\bfE[\overline{\bfF}] < \infty$. Then by conservation of energy for the hyperbolic Yang-Mills equations at $s=0$, we see that
\begin{equation*}
	\bfE[\bfF(t, s=0)] = \bfE[\overline{\bfF}] \quad \forall t \in (-T_{0}, T_{0}).
\end{equation*}

We will denote the common value of $\bfE[\bfF(t, s=0)]$ by $\bfE$.

\begin{proposition} \label{prop:ctrlByE:Alw}
There exists $N > 0$ such that for any $t \in (-T_{0}, T_{0})$, we have
\begin{equation} \label{eq:ctrlByE:Alw:0}
	\calIAlow(t) \leq C_{\calIAlow(0), \bfE} \, (1+\abs{t})^{N}.
\end{equation}
\end{proposition}

\begin{remark}
In the course of the proof, it will be clear that $N$ may be chosen to depend only on the number of derivatives of $\Alow_{i}$ controlled. In our case, in which we control up to 31 derivatives of $\Alow_{i}$, we may choose $N = 32$.
\end{remark}

\begin{proof} 
By symmetry, it suffices to consider $t > 0$. The main idea is to use the relation
\begin{equation} \label{eq:ctrlByE:Alw:pf:0}
	\Flow_{0i} = \rd_{t} \Alow_{i},
\end{equation}
which holds thanks to the fact that we are in the temporal gauge $\Alow_{0} = 0$ along $s=1$, and proceed as in the proof of Proposition \ref{prop:intEst4A:s}.

We first estimate the $L^{\infty}_{x}$ norms. We claim that
\begin{equation} \label{eq:ctrlByE:Alw:pf:1}
	\nrm{\rd_{x}^{(k)} \Alow(t)}_{L^{\infty}_{x}} \leq C_{k, \calIAlow(0), \bfE} \, (1+t)^{k+1}
\end{equation}
for $0 \leq k \leq 29$.  

Let us begin with the case $k=0$ and proceed by induction. Note the inequality
\begin{equation*}
\nrm{\Alow_{i}(t)}_{L^{\infty}_{x}} 
\leq \nrm{\Alow_{i}(t=0)}_{L^{\infty}_{x}} + \int_{0}^{t} \nrm{\Flow_{0i}(t')}_{L^{\infty}_{x}} \, \ud t'.
\end{equation*}

Using Proposition \ref{prop:pEst4covF}, we may estimate the last term by $C_{\bfE} \, t$; from this, the $k=0$ case of \eqref{eq:ctrlByE:Alw:pf:1} follows.

Next, to carry out the induction, let us assume that \eqref{eq:ctrlByE:Alw:pf:1}  holds for $0, 1, \ldots, k-1$, where $1 \leq k \leq 29$. Taking $\rd_{x}^{(k)}$ of both sides of \eqref{eq:ctrlByE:Alw:pf:0} and using the fundamental theorem of calculus, we obtain
\begin{equation*}
\nrm{\rd_{x}^{(k)} \Alow_{i}(t)}_{L^{\infty}_{x}} 
\leq \nrm{\rd_{x}^{(k)} \Alow_{i}(t=0)}_{L^{\infty}_{x}} + \int_{0}^{t} \nrm{\rd_{x}^{(k)} \Flow_{0i}(t')}_{L^{\infty}_{x}} \, \ud t'.
\end{equation*}

The first term is estimated by $\calIAlow(t=0)$, as $1 \leq k \leq 29$. For the second term, we apply \eqref{eq:dSub:u2cov} of Lemma \ref{lem:dSub}. Then it suffices to estimate
\begin{equation*}
	\int_{0}^{t} \bb( \nrm{\covD_{x}^{(k)} \Flow_{0i}(t')}_{L^{\infty}_{x}} + \sum_{\star} \nrm{\calO(\rd_{x}^{(\ell_{1})} A, \ldots, \rd_{x}^{(\ell_{j})} A, \covD_{x}^{(\ell)} \Flow_{0i})(t')}_{L^{\infty}_{x}} \bb) \, \ud t'
\end{equation*}
where the range of the summation $\sum_{\star}$ is as in Lemma \ref{lem:dSub}. In particular, $\ell_{1}, \ldots, \ell_{j}, \ell \leq k-1$. Let us use H\"older to estimate each factor in $L^{\infty}_{x}$, and estimate the derivatives of $A$ and $\Flow_{0i}$ by the induction hypothesis and Propostion \ref{prop:pEst4covF}, respectively. Then it is not difficult to see that the worst term (in terms of growth in $t$) is of the size
\begin{equation*}
	 C_{k, \calIAlow(0), \bfE} \int_{0}^{t} (1+t')^{k} \, \ud t' =  C_{k, \calIAlow(0), \bfE} \, (1+t)^{k+1}.
\end{equation*}

Therefore, \eqref{eq:ctrlByE:Alw:pf:1} for $k$ follows. By induction, this establishes the claim.

With \eqref{eq:ctrlByE:Alw:pf:1} in hand, we now proceed to prove
\begin{equation} \label{eq:ctrlByE:Alw:pf:2}
	\nrm{\rd_{x}^{(k)}  \Alow(t)}_{L^{2}_{x}} \leq C_{k, \calIAlow(0), \bfE} \, (1+t)^{k+1}
\end{equation}
for $1 \leq k \leq 31$.  

Arguing as in the proof of \eqref{eq:ctrlByE:Alw:pf:1}, we arrive at the inequality
\begin{equation*}
\begin{aligned}
\nrm{\rd_{x}^{(k)} \Alow_{i}(t)}_{L^{2}_{x}} 
\leq & \nrm{\rd_{x}^{(k)} \Alow_{i}(t=0)}_{L^{2}_{x}} \\
& + \int_{0}^{t} \bb( \nrm{\covD_{x}^{(k)} \Flow_{0i}(t')}_{L^{2}_{x}} + \sum_{\star} \nrm{\calO(\rd_{x}^{(\ell_{1})} A, \ldots, \rd_{x}^{(\ell_{j})} A, \covD_{x}^{(\ell)} \Flow_{0i})(t')}_{L^{2}_{x}} \bb) \, \ud t'.
\end{aligned}
\end{equation*}

It suffices to estimate the $t'$-integral. For $1 \leq k \leq 30$, let us use H\"older to estimate each $\rd_{x}^{(\ell_{i})} A$ in $L^{\infty}_{x}$ and $\covD^{(\ell)}_{x} \Flow_{0i}$ in $L^{2}_{x}$. Then we estimate these by \eqref{eq:ctrlByE:Alw:pf:1} and Proposition \ref{prop:pEst4covF}, respectively, from which \eqref{eq:ctrlByE:Alw:pf:2} follows immediately for $1 \leq k \leq 30$. 

Next, proceeding similarly in the case $k=31$, all terms are easily seen to be okay except
\begin{equation*}
	\nrm{\calO(\rd_{x}^{(30)} A, \Flow_{0i})(t')}_{L^{2}_{x}},
\end{equation*}
for which we cannot use \eqref{eq:ctrlByE:Alw:pf:1}. In this case, however, we may put $\rd_{x}^{(30)} A$ in $L^{2}_{x}$ and $\Flow_{0i}$ in $L^{\infty}_{x}$. Then the former can be estimated by using the case $k=30$ of \eqref{eq:ctrlByE:Alw:pf:2} that we have just established, whereas the estimate for the latter follows from Proposition \ref{prop:pEst4covF}. It follows that this term is of size $C_{k, \calIAlow(0), \bfE} \, (1+t')^{31}$. Integrating over $[0,t]$ gives the growth $C_{k, \calIAlow(0), \bfE} \, (1+t)^{32}$.

Finally, we are left to prove the estimates for $\rd_{t} \Alow_{i}$. For this purpose, we claim
\begin{equation} \label{eq:ctrlByE:Alw:pf:3}
	\nrm{\rd_{x}^{(k-1)} \rd_{t} \Alow(t)}_{L^{2}_{x}} \leq C_{k, \calIAlow(0), \bfE} \, (1+t)^{k-1}
\end{equation}
for $1 \leq k \leq 30$. 

To prove \eqref{eq:ctrlByE:Alw:pf:3}, recall that $\rd_{t} \Alow_{i} = \Flow_{0i}$; therefore, the case $k=1$ follows immediately from Proposition \ref{prop:pEst4covF}. For $k > 1$, we take $\rd_{x}^{(k-1)}$ and use Lemma \ref{lem:dSub} to substitute the usual derivatives by covariant derivatives. Then by \eqref{eq:ctrlByE:Alw:pf:1}, \eqref{eq:ctrlByE:Alw:pf:2} and Proposition \ref{prop:pEst4covF}, \eqref{eq:ctrlByE:Alw:pf:3} follows. 

Combining \eqref{eq:ctrlByE:Alw:pf:2} and \eqref{eq:ctrlByE:Alw:pf:3}, we obtain \eqref{eq:ctrlByE:Alw:0} with $N = 32$. \qedhere
\end{proof}

\begin{proposition} \label{prop:ctrlByE:Fs}
For any $t \in (-T_{0}, T_{0})$, we have
\begin{equation} \label{eq:ctrlByE:Fs:0}
	\calIFs(t) \leq C_{\calIAlow(t), \bfE} \cdot \sqrt{\bfE}.
\end{equation}
\end{proposition}

\begin{proof} 
Throughout the proof, the time $t \in (-T_{0}, T_{0})$ will be fixed and thus be omitted. 

Recalling the definition of $\calIFs$, establishing \eqref{eq:ctrlByE:Fs:0} reduces to proving
\begin{align} 
\nrm{\nb_{x} F_{s}}_{\calL^{5/4,p}_{s} \dot{\calH}^{k-1}_{x}} \leq & C_{k, \calIAlow, \bfE} \cdot \sqrt{\bfE}, \label{eq:ctrlByE:Fs:pf:1} \\
\nrm{\nb_{0} F_{s}}_{\calL^{5/4,p}_{s} \dot{\calH}^{k-1}_{x}} \leq & C_{k, \calIAlow, \bfE} \cdot \sqrt{\bfE} \label{eq:ctrlByE:Fs:pf:2}
\end{align}
for $1 \leq k \leq 10$ and $p=2, \infty$.

The estimate \eqref{eq:ctrlByE:Fs:pf:1} is an easy consequence of \eqref{eq:pEst4covFs:1} of Corollary \ref{cor:pEst4covFs} and Corollary \ref{cor:cov2u}. On the other hand, to prove \eqref{eq:ctrlByE:Fs:pf:2}, we use the formula
\begin{equation*} 
	\nb_{0} F_{si} = s^{-1/2} \calD^{\ell} \calD_{\ell} F_{0i} - 2 s^{1/2} \LieBr{\tensor{F}{_{0}^{\ell}}}{F_{i\ell}} + \calD_{i} F_{s0} - s^{1/2} \LieBr{A_{0}}{F_{si}},
\end{equation*}
which is an easy consequence of the Bianchi identity \eqref{eq:fullBianchi} and the parabolic equation for $\covD_{s} F_{0i}$. Taking $\covD_{x}^{(k-1)}$ of both sides and using Proposition \ref{prop:pEst4covF}, Corollary \ref{cor:pEst4covFs} and Proposition \ref{prop:intEst4A:s}, we obtain
\begin{equation} \label{eq:pEst4covD0Fsi}
	\nrm{\calD_{x}^{(k-1)} \nb_{0} F_{s i}}_{\calL^{5/4, \infty}_{s} \calL^{2}_{x}(0,1]} + \nrm{\calD_{x}^{(k-1)} \nb_{0} F_{s i}}_{\calL^{5/4, 2}_{s} \calL^{2}_{x}(0,1]} \leq C_{k, \bfE} \cdot \sqrt{\bfE}
\end{equation}
for $k \geq 1$. At this point, applying Corollary \ref{cor:cov2u}, we obtain \eqref{eq:ctrlByE:Fs:pf:2}. \qedhere 
\end{proof}

Combining Propositions \ref{prop:ctrlByE:Alw} and \ref{prop:ctrlByE:Fs}, Theorem \ref{thm:ctrlByE} follows.

%==========

\section{Short time estimates for (HPYM) in the caloric-temporal gauge: \\ Proof of Theorem \ref{thm:dynEst}} \label{sec:pfOfDynEst}
The goal of this section is to prove Theorem \ref{thm:dynEst}. As discussed in Section \ref{sec:reduction}, this theorem follows from a local-in-time analysis of the wave equations of \eqref{eq:HPYM}. As such, its proof will follow closely that of \cite[Theorem 4.9]{Oh:6stz7nRe}, which is essentially a `$H^{1}_{x}$ local well-posedness (in time)' statement for \eqref{eq:HPYM} in the caloric-temporal gauge.

To begin with, let us borrow the following definition from \cite{Oh:6stz7nRe}.
\begin{equation} \label{eq:calE}
	\calE(t) := \sum_{m=1}^{3} \bb( \nrm{\nb_{x}^{(m-1)} F_{s0}(t)}_{\calL^{1, \infty}_{s} \calL^{2}_{x}(0,1]} + \nrm{\nb_{x}^{(m)} F_{s0}(t)}_{\calL^{1, 2}_{s} \calL^{2}_{x}(0,1]} \bb).
\end{equation}

Given a time interval $I \subset \bbR$, we define $\calE(I)$ to be $\sup_{t \in I} \calE(t)$.

Moreover, also borrowing from \cite{Oh:6stz7nRe}, we assert the existence of norms $\calF(I), \calAlow(I)$ of $F_{si}, \Alow_{i}$, respectively, such that the following lemma holds.
\begin{lemma} \label{lem:ingr4dynEst}
Let $I \subset \bbR$ be a finite open interval centered at $t=0$, and $A_{\bfa}$ a regular solution to \eqref{eq:HPYM} in the caloric-temporal gauge on $I \times \bbR^{3} \times [0,1]$. Then:
\begin{enumerate}
\item The following estimates hold.
\begin{equation*}
\calA_{0}[A_{0}(s=0)](I) \leq  C_{\calF(I), \calAlow(I)} \cdot \calE(I) + C_{\calF(I), \calAlow(I)} \cdot (\calF(I) + \calAlow(I))^{2}, 
\end{equation*}
\begin{equation*}
\sup_{i} \sup_{s \in [0,1]} \nrm{\rd_{t,x} A_{i}(s)}_{L^{\infty}_{t} L^{2}_{x} (I)} \leq C_{\calF(I), \calAlow(I)} \cdot (\calF(I) + \calAlow(I)).
\end{equation*}

\item The norms $\calF{(-T, T)}$ and $\calAlow{(-T, T)}$ are continuous as functions of $T$ (where $0 < T < T_{0}$). Furthermore, we have
\begin{align*}
	\limsup_{T \to 0+} \bb( \calF{(-T, T)} + \calAlow{(-T, T)} \bb) \leq C \, \calI, \\
	\limsup_{T \to 0+} \bb( \dlt \calF{(-T, T)} + \dlt \calAlow{(-T, T)} \bb) \leq C \, \dlt \calI.
\end{align*}

\item The following estimates hold for $\calAlow$:
\begin{equation*}
	\calAlow(I) \leq C \calI(0) + \abs{I} \bb( C_{\calF(I), \calAlow(I)} \cdot \calE(I) + C_{\calE(I), \calF(I), \calAlow(I)} \cdot (\calE(I) + \calF(I) + \calAlow(I))^{2} \bb).
\end{equation*}

\item The following estimates hold for $\calF$:
\begin{equation*}
	\calF(I) \leq  C \calI(0) + \abs{I}^{1/2} C_{\calE(I), \calF(I), \calAlow(I)} \cdot (\calE(I) + \calF(I) + \calAlow(I))^{2}
\end{equation*}
\end{enumerate}
\end{lemma}
\begin{proof} 
This is essentially a summary of \cite[Propositions 7.1, 7.2, 7.4]{Oh:6stz7nRe} and \cite[Theorems 7.5, 7.6]{Oh:6stz7nRe}. \qedhere 
\end{proof}

The main reason why the analysis in \cite{Oh:6stz7nRe} is insufficient to prove Theorem \ref{thm:dynEst} is because of \cite[Proposition 7.3]{Oh:6stz7nRe}, which gives an estimate for $\calE(t)$ \emph{only under the hypothesis} that either the size of the initial data or the $s$-interval is small. We remark that the latter case is not explicitly treated in \cite{Oh:6stz7nRe}, but follows by a scaling argument. The following proposition is a replacement of \cite[Proposition 7.3]{Oh:6stz7nRe}, which utilizes the smallness of the conserved energy $\bfE(t)$ instead, based on the covariant parabolic estimates derived in Section \ref{sec:covParabolic}.

\begin{proposition} \label{prop:impEst4Fs0}
Let $I \subset \bbR$ be an open interval and $t \in I$. Consider a regular solution $A_{\bfa}$ to \eqref{eq:HPYM} in the caloric-temporal gauge on $I \times \bbR^{3} \times [0,1]$ such that
\begin{equation*}
	\bfE(t) = \bfE[\bfF(t, s=0)] < \dlt, \quad \calIAlow(t) \leq D,
\end{equation*}
where $D > 0$ is an arbitrarily large number and $\dlt > 0$ is the small constant in Proposition \ref{prop:pEst4covF}. Then the following estimate holds.
\begin{equation} \label{eq:a0First:low:0}
	\calE(t) \leq C_{D, \dlt}.
\end{equation}
\end{proposition}

In \S \ref{subsec:impEst4Fs0}, we give a proof of Proposition \ref{prop:impEst4Fs0}. Assuming Proposition \ref{prop:impEst4Fs0}, the proof of Theorem \ref{thm:dynEst} is a straightforward adaptation of \cite[Proof of Theorem 4.9]{Oh:6stz7nRe}. We present a sketch in \S \ref{subsec:pfOfDynEst}.

\subsection{Improvement of estimates for $F_{s0}$ : Proof of Proposition \ref{prop:impEst4Fs0}} \label{subsec:impEst4Fs0}
%For $F_{s0}$, on the other hand, we can actually improve \eqref{eq:pEst4covFsi} thanks to the Yang-Mills equations at $s=0$. 

The goal of this subsection is to prove Proposition \ref{prop:impEst4Fs0}. Consider a regular solution $A_{\bfa}$ to \eqref{eq:HPYM} on $(-T_{0}, T_{0}) \times \bbR^{3} \times [0,1]$ ($T_{0} > 0$). From Proposition \ref{prop:covParabolic} and the fact that $F_{s0} = \covD^{\ell} F_{\ell 0}$, it is not difficult to verify that $F_{s0}$ satisfies the covariant parabolic equation
\begin{equation} \label{eq:covParabolic4Fs0}
	(\covD_{s} - \covD^{\ell} \covD_{\ell}) F_{s0} = 2\LieBr{\tensor{F}{_{0}^{\ell}}}{F_{s\ell}}.
\end{equation}

Recall furthermore that  $\LieBr{\covD_{i}}{(\covD_{s} - \covD^{\ell} \covD_{\ell})} B = \calO(F, \covD_{x} B)$. This implies that $\covD_{x}^{(k)} F_{s0}$ for $k \geq 1$ satisfies the following schematic parabolic equation.
\begin{equation} \label{eq:covParabolic4DFs0}
	(\covD_{s} - \covD^{\ell} \covD_{\ell}) (\covD_{x}^{(k)} F_{s0}) = \sum_{j=0}^{k} \calO( \covD_{x}^{(j)} \bfF, \covD_{x}^{(k-j)} \bfF_{s}).
\end{equation}

Now recall that the hyperbolic Yang-Mills equation holds along $s=0$. In particular, the constraint equation $\covD^{\ell} F_{\ell 0}(s=0) = 0$ holds, which is equivalent to $F_{s0}(s=0) = 0$. Taking this extra ingredient into account, it follows that $F_{s0}$ obeys an \emph{improved bound} compared to the one proved in Section \ref{sec:covParabolic}, as we state below.

\begin{proposition}[Improved estimate for $F_{s0}$, with covariant derivatives] \label{prop:pEst4covFs0}
Let $T_{0} > 0$, $t \in (-T_{0}, T_{0})$, and consider a regular solution $A_{\bfa}$ to \eqref{eq:HPYM} on $(-T_{0}, T_{0}) \times \bbR^{3} \times [0,1]$. If $\bfE(t) = \bfE[\bfF(t, s=0)] < \dlt$, where $\dlt > 0$ is the small constant in Proposition \ref{prop:pEst4covF}, then the following estimate holds for each integer $k \geq 0$.
\begin{equation} \label{eq:pEst4covFs0:1}
	\nrm{\calD_{x}^{(k-1)} F_{s 0}(t)}_{\calL^{1, \infty}_{s} \calL^{2}_{x}(0,1]} + \nrm{\calD_{x}^{(k)} F_{s 0}(t)}_{\calL^{1, 2}_{s} \calL^{2}_{x}(0,1]} \leq C_{k, \bfE(t)} \cdot \bfE(t),
\end{equation}

When $k=0$, we omit the first term on the left-hand side.
\end{proposition}
\begin{proof} 
We will fix $t \in (-T_{0}, T_{0})$ and therefore omit writing $t$. Let us begin with the case $k=0$. Applying Lemma \ref{lem:est4Duhamel} to the covariant parabolic equation for $F_{s0}$, along with the fact that $F_{s0} = 0$ at $s=0$ thanks to \eqref{eq:hyperbolicYM}, it follows that
\begin{equation*}
	\nrm{F_{s0}}_{\calL^{1,2}_{s} \calL^{2}_{x}} \leq 2\nrm{\int_{0}^{s} e^{(s-\sbr) \lap} \abs{\LieBr{\tensor{F}{_{0}^{\ell}}}{F_{s\ell}}}(\sbr) \, \ud \sbr}_{\calL^{1,2}_{s} \calL^{2}_{x}}.
\end{equation*}

Using Lemma \ref{lem:est4Duhamel}, H\"older, \eqref{eq:pEst4covF:1} and \eqref{eq:pEst4covFs:1}, we have
\begin{align*}
\nrm{\int_{0}^{s} e^{(s-\sbr) \lap} \abs{\LieBr{\tensor{F}{_{0}^{\ell}}}{F_{s\ell}}}(\sbr) \, \ud \sbr}_{\calL^{1,2}_{s} \calL^{2}_{x}}
\leq & \nrm{\LieBr{\tensor{F}{_{0}^{\ell}}}{F_{s\ell}}}_{\calL^{1+1,2}_{s} \calL^{1}_{x}} \\
\leq & \nrm{\tensor{F}{_{0}^{\ell}}}_{\calL^{3/4,\infty}_{s} \calL^{2}_{x}} \nrm{F_{s\ell}}_{\calL^{5/4,2}_{s} \calL^{2}_{x}} 
\leq C_{\bfE} \cdot \bfE.
\end{align*}

Therefore, we have proved $\nrm{F_{s0}}_{\calL^{1,2}_{s} \calL^{2}_{x}} \leq C_{\bfE} \cdot \bfE$.

For $k \geq 1$, we proceed by induction. Suppose, for the purpose of induction, that the cases $0, \cdots, k-1$ have already been established. Using the energy integral estimate \eqref{eq:pEst4covHeat:1} with $\ell =1 + \frac{k-1}{2}$ to  \eqref{eq:covParabolic4DFs0} for $\covD_{x}^{(k-1)} F_{s0}$, we see that
\begin{align*}
	&\nrm{\calD_{x}^{(k-1)} F_{s0}}_{\calL^{1,\infty}_{s} \calL^{2}_{x}} + \nrm{\calD_{x}^{(k)} F_{s0}}_{\calL^{1,2}_{s} \calL^{2}_{x}} \\
	&\quad \leq C \nrm{\calD_{x}^{(k-1)}F_{s0}}_{\calL^{1,2}_{s} \calL^{2}_{x}} + C \sum_{j=0}^{k-1} \nrm{\calO( \calD_{x}^{(j)} \bfF, \calD_{x}^{(k-1-j)} \bfF_{s})}_{\calL^{1,1}_{x} \calL^{2}_{x}}.
\end{align*}

The first term on the right-hand side is acceptable by the induction hypothesis; we therefore focus on the second term. Let us use H\"older to estimate $\calD^{(j)}_{x} \bfF$ in $\calL^{3/4,2}_{s} \calL^{6}_{x}$ and $\calD^{(k-1-j)}_{x} \bfF_{s}$ in $\calL^{5/4,2}_{s} \calL^{3}_{x}$. Next, we apply Corollary \ref{cor:covSob} to each. Then using Proposition \ref{prop:pEst4covF} and Corollary \ref{cor:pEst4covFs}, the sum is estimated by
\begin{equation*}
	\sum_{j=0}^{k-1} \nrm{\calD_{x}^{(j+1)} \bfF}_{\calL^{3/4,2}_{s} \calL^{2}_{x}} \nrm{\calD_{x}^{(k-1-j)} \bfF_{s}}_{\calL^{5/4,2}_{s} \calL^{2}_{x}}^{1/2} \nrm{\calD_{x}^{(k-j)} \bfF_{s}}_{\calL^{5/4,2}_{s} \calL^{2}_{x}}^{1/2} \leq C_{k, \bfE} \cdot \bfE,
\end{equation*}

Therefore,  \eqref{eq:pEst4covFs0:1} holds for the case $k$, which completes the induction. \qedhere
\end{proof}

Suppose furthermore that $A_{\bfa}$ is in the caloric-temporal gauge, so that $A_{s} =0$ in particular. Combining Proposition \ref{prop:pEst4covFs0} and Corollary \ref{cor:cov2u}, the covariant derivative estimate \eqref{eq:pEst4covFs0:1} leads to the corresponding estimate for usual derivatives. This is the content of the following corollary, whose proof we omit.
\begin{corollary}[Improved estimate for $F_{s0}$, with usual derivatives] \label{cor:pEst4Fs0}
Assume that the hypotheses of Proposition \ref{prop:impEst4Fs0} hold. Furthermore, assume that $A_{\bfa}$ satisfies the caloric-temporal gauge condition. Then the following estimate holds for $0 \leq k \leq 29$.
\begin{equation} \label{eq:pEst4Fs0:0}
	\nrm{\nb_{x}^{(k)} F_{s 0}(t)}_{\calL^{1, \infty}_{s} \calL^{2}_{x}(0,1]} + \nrm{\nb_{x}^{(k)} F_{s 0}(t)}_{\calL^{1, 2}_{s} \calL^{2}_{x}(0,1]} \leq C_{k, \calIAlow(t), \bfE(t)} \cdot \bfE(t).
\end{equation}
\end{corollary}

The estimate \eqref{eq:pEst4Fs0:0} is more than sufficient to prove Proposition \ref{prop:impEst4Fs0}.

\subsection{Proof of Theorem \ref{thm:dynEst}} \label{subsec:pfOfDynEst} 
With Proposition \ref{prop:impEst4Fs0}, we are ready to give a proof of Theorem \ref{thm:dynEst}. We follow \cite[Proof of Theorem 4.9]{Oh:6stz7nRe}, replacing \cite[Proposition 7.3]{Oh:6stz7nRe} by Proposition \ref{prop:impEst4Fs0}. 

\begin{proof} [Proof of Theorem \ref{thm:dynEst}]
Let $A_{\bfa}$ be a regular solution to the hyperbolic-parabolic Yang-Mills equation in the caloric-temporal gauge on $(-T_{0}, T_{0}) \times \bbR^{3} \times [0,1]$ such that \eqref{eq:dynEst:hyp} is satisfied. For simplicity, we will consider the case when $I_{0}$ is centered at $t=0$, i.e., $I_{0} = (-d/2, d/2)$ for $d > 0$ to be determined. As we shall see, the proof only utilizes the hypotheses \eqref{eq:dynEst:hyp} on $I_{0}$; therefore, the same proof applies to other $I_{0} \subset (-T_{0}, T_{0})$ as well.

We claim that 
\begin{equation} \label{eq:dynEst:pf:0}
	\calF(I_{0}) + \calAlow(I_{0}) \leq B D,
\end{equation}
for a large enough absolute constant $B$, to be determined later, provided that $\abs{I_{0}} = d$ is small enough. Note that Theorem \ref{thm:dynEst} then follows immediately from the claim, thanks to Part (1) of Lemma \ref{lem:ingr4dynEst}.

We will use a bootstrap argument. The starting point is provided by Part (2) of Lemma \ref{lem:ingr4dynEst}, which implies 
\begin{equation*}
	\calF(I'_{0}) + \calAlow(I'_{0}) \leq 2 \calI(0),
\end{equation*}
for some subinterval $I'_{0} \subset I$ containing $0$ such that $\abs{I'_{0}} > 0$ is sufficiently small (by upper semi-continuity of $\calF, \calAlow$ at $0$). Note that the right-hand side is estimated by $B D$, provided we choose $B \geq 2$.

Next, let us assume the following \emph{bootstrap assumption}:
\begin{equation*}
	\calF(I'_{0}) + \calAlow(I'_{0}) \leq 2 B D
\end{equation*}
for $I'_{0} := (-T', T') \subset I$. Applying Parts (3), (4) of Lemma \ref{lem:ingr4dynEst}, and using Proposition \ref{prop:impEst4Fs0} to control $\calE(I'_{0})$, we obtain
\begin{align*}
	\calF(I'_{0}) + \calAlow(I'_{0})
	\leq & C \calI + (T')^{1/2} C_{C_{D, \bfE}, \calF(I'_{0}), \calAlow(I'_{0})} (C_{D, \bfE} + \calF(I'_{0}) + \calAlow(I'_{0}))^{2} \\
	& + (T') \bb( C_{\calF(I'_{0}), \calAlow(I'_{0})} C_{D, \bfE} + C_{C_{D, \bfE}, \calF(I'_{0}), \calAlow(I'_{0})} (C_{D, \bfE} + \calF(I'_{0}) + \calAlow(I'_{0}))^{2} \bb),
\end{align*}
where $\bfE$ is the common value of $\bfE[\bfF(t, s=0)]$ (by conservation of energy). Here, we used the hypotheses \eqref{eq:dynEst:hyp} on $I'_{0} \subset I_{0}$. Using the bootstrap assumption and choosing $d$ small enough depending on $D, \dlt$ and $B$ (note that $T' \leq d$), we can make the second and third terms on the right-hand $\leq \frac{B}{2} D$. Then choosing $B > 2 C$, we see that 
\begin{equation*}
	\calF(I'_{0}) + \calAlow(I'_{0}) \leq B D,
\end{equation*}
which beats the bootstrap assumption. By a standard continuity argument, \eqref{eq:dynEst:pf:0} then follows. \qedhere
\end{proof}

\appendix
\section{Dependence on the companion paper} \label{appendix:dependence}
%For the convenience of the reader, we provide a table explaining which result in the companion paper \cite{Oh:6stz7nRe} is used where in the present paper.

{\footnotesize
\begin{center}
\renewcommand{\arraystretch}{1.2}
\begin{tabular}{c c c p{15em}}
		\multicolumn{2}{c}{\bfseries Ref. in this paper} &   {\bfseries Ref. in \cite{Oh:6stz7nRe}} & {\bfseries Description} \\ \hline 
Section \ref{sec:introduction}
			& Theorem \ref{thm:lwp4YM} 	
			& Main Theorem 		
			& $H^{1}$ LWP of \eqref{eq:hyperbolicYM} in the temporal gauge \\
			\hline 
Section \ref{sec:reduction} 	
			& Lemma \ref{lem:regApprox}		
			& Lemma 4.5		
			& Approximation of admissible $H^{1}$ initial data by regular initial data \\			
			& Lemma \ref{lem:est4gt2temporal}		
			& Lemma 4.6		
			& Estimates for gauge transforms to the temporal gauge \\			
			\hline 
Section \ref{sec:YMHF}
			& Theorem \ref{thm:lwp4YMHF}
			& Theorem 5.1 
			& Local existence for \eqref{eq:YMHF} with $\dot{H}^{1}_{x}$ initial data \\
			& Proof of Theorem \ref{thm:impLWP4dYMHF}
			& Proposition 5.7
			& WP of the linear covariant parabolic equation for $F_{0i}$ in \eqref{eq:dYMHF}\\
			&
			& Lemma 6.1
			& Solving for $A_{0}$ in \eqref{eq:dYMHF} \\
			\hline
Section \ref{sec:pfOfIdEst}
			& Theorem \ref{thm:oldIdEst}
			& Theorem 4.8 
			& Gauge transformation into the caloric-temporal gauge, with estimates	\\
			\hline
Section \ref{sec:pfOfDynEst}
			& Lemma \ref{lem:ingr4dynEst}.(1)	
			& Proposition 7.1
			& Estimate for the norm $\calA_{0}$ for $A_{0}(s=0)$ \\
			& Lemma \ref{lem:ingr4dynEst}.(1)	
			& Proposition 7.2
			& Estimate for $\nrm{\rd_{t,x} A_{i}(s)}_{L^{\infty}_{t} L^{2}_{x}}$ uniform in $s \in [0,1]$ \\
			& Lemma \ref{lem:ingr4dynEst}.(2)	
			& Proposition 7.4
			& Continuity properties of $\calF$, $\calAlow$ \\
			& Lemma \ref{lem:ingr4dynEst}.(3)	
			& Theorem 7.5
			& Hyperbolic estimates for $\Alow_{i}$ in the caloric-temporal gauge\\
			& Lemma \ref{lem:ingr4dynEst}.(4)	
			& Theorem 7.6
			& Hyperbolic estimates for $F_{si}$ in the caloric-temporal gauge \\
\end{tabular}
\end{center}
}
\section{List of symbols} \label{appendix:LOS}
\begin{center}
\footnotesize
\begin{tabular}{c  c p{25em}}
{\bfseries Symbol} & {\bfseries Ref.} & {\bfseries Description}  \\
$\LieGrp, \LieAlg$ & \S \ref{subsec:intro:bg} & {$\LieGrp$ is the Lie group equipped with a bi-invariant inner product, $\LieAlg$ is the associated Lie algebra}\\
$\bfD_{\mu}, \bfD_{t,x}, \bfD_{x}$ & \S \ref{subsec:intro:bg} & {Covariant derivative in the $\mu$-direction, covariant space-time gradient and covariant spatial gradient, resp.} \\
$\partial_{\mu}, \rd_{t,x}, \rd_{x}$ & \S \ref{subsec:intro:bg} & {Ordinary derivative in the $\mu$-direction, ordinary space-time gradient and ordinary spatial gradient, resp.} \\
$\bfE(t)= \bfE[\bfF](t), \bfE[\overline{\bfF}] $ & \eqref{eq:YMenergy} & Energy of $\bfF(t,x) = F_{\mu \nu}(t,x)$. $\bfE[\overline{\bfF}]$ is the energy at $t=0$. \\
$\bfB(t) = \bfB[F](t) $ & \eqref{eq:Menergy} & Magnetic energy of $F(t,x) = F_{ij}(t,x)$ \\
$\Atemp$ & \S\S \ref{subsec:intro:lwp}, \ref{subsec:overview4GWP}, \ref{subsec:reduction} 
		& Solution to \eqref{eq:hyperbolicYM} in the temporal gauge $A^{\dagger}_{0} = 0$\\
$\calIini$ & \S \ref{subsec:intro:mainThm} & $\calIini := \nrm{\Aini}_{\dot{H}^{1}_{x}} + \nrm{\Eini}_{L^{2}_{x}}$ \\
$\calA_{0}[A_{0}](I)$ & \eqref{eq:calA0} & A norm for $A_{0}$, used for controlling the ODE $\rd_{0} V = V A_{0}$ \\
$\calD_{\mu}, \calD_{t,x}, \calD_{x} $ & \S \ref{subsec:prelim:pnorm} 
		& {$\calD_{\mu}(s) = s^{1/2} \covD_{\mu}$, $\calD_{t,x}(s) = s^{1/2} \covD_{t,x}$, $\calD_{x}(s) = s^{1/2} \covD_{x}$}\\
$\nabla_{\mu}, \nabla_{t,x}, \nabla_{x}$ & \S \ref{subsec:prelim:pnorm} 
		& {$\nb_{\mu}(s) = s^{1/2} \rd_{\mu}$, $\nb_{t,x}(s) = s^{1/2} \rd_{t,x}$, $\nb_{x}(s) = s^{1/2} \rd_{x}$} \\
$\calL^{k, p}_{s}$ & \S \ref{subsec:prelim:pnorm}
		& $\nrm{f(s)}_{\calL^{k, p}_{s}(J)} := (\int_{J} (s^{k} \abs{f(s)})^{p} \, \frac{\ud s}{s} )^{1/p}$ \\
$\calL^{r}_{x}$ & \S \ref{subsec:prelim:pnorm} 
		& $\nrm{\psi(s, x)}_{\calL^{r}_{x}(s)} := s^{-3/2r} \nrm{\psi(s, x)}_{L^{r}_{x}}$ \\
$\dot{\calH}^{m}_{x}$ & \S \ref{subsec:prelim:pnorm} 
		& $\nrm{\psi(s, x)}_{\dot{\calH}^{m}_{x}(s)} := s^{m/2-3/4} \nrm{\psi(s, x)}_{\dot{H}^{m}_{x}}$ \\
$\calI(t)$ & \S \ref{sec:pfOfIdEst} & A norm for data for \eqref{eq:HPYM} at time $t$ in the caloric-temporal gauge. $\calI(t) = \calIAlow(t) + \calIFs(t)$ \\
$\calIAlow(t)$ & \eqref{eq:calIAlow}, \S \ref{sec:pfOfIdEst} 
		& A norm for the $\calAlow$ part of data for \eqref{eq:HPYM} at time $t$ in the caloric-temporal gauge \\
$\calIFs(t)$ & \S \ref{sec:pfOfIdEst} & A norm for the $F_{s}$ part of data for \eqref{eq:HPYM} at time $t$ in the caloric-temporal gauge \\
$\calE(t)$ & \S \ref{sec:pfOfDynEst} & A norm for $F_{s0} = \rd_{s} A_{0}$ in the caloric-temporal gauge\\
$\calE(I)$ & \S \ref{sec:pfOfDynEst}  & $\calE(I) := \sup_{t \in I} \calE(t)$ \\
$\calF(I), \calAlow(I)$ & \S \ref{sec:pfOfDynEst} & Space-time norms for $F_{s}$ and $\Alow$, resp., adapted to the wave equation\\ 
\end{tabular}
\end{center}

%%%%%%%%%%%%

% ----------------------------------------------------------------
\bibliographystyle{amsplain}
%\bibliography{SJ-library}
%\bibliography{081212-library.bib}

\begin{thebibliography}{10}

\bibitem{Bejenaru:2011wy}
I.~Bejenaru, A.~D Ionescu, C.~E. Kenig, and D. Tataru,
 Global Schr{\"o}dinger maps in dimensions $d \geq 2$: Small data in
  the critical Sobolev spaces, \emph{Ann. of Math.} {\bf 173} (2011) 1443-1506.

\bibitem{Bleeker:2005uj}
D. Bleeker, \emph{{Gauge Theory and Variational Principles}}, (Dover
  Publications, 2005).

\bibitem{Charalambous:2010vt}
N. Charalambous and L. Gross, {{The Yang-Mills heat semigroup on
  three-manifolds with boundary}}, arXiv:1004.1639 [math.AP].

\bibitem{MR820070}
D. Christodoulou, {Global solutions of nonlinear hyperbolic
  equations for small initial data}, \emph{Comm. Pure Appl. Math.} {\bf 39} (1986) 267--282.

\bibitem{MR1604914}
P.~T. Chru{\'s}ciel and J. Shatah, {Global existence of solutions of the {Y}ang-{M}ills equations on globally hyperbolic four-dimensional {L}orentzian manifolds}, \emph{Asian J. Math.} {\bf 1} (1997) 530-548

\bibitem{MR0499948}
R.~R.~Coifman and G.~Weiss, \emph{Analyse harmonique non-commutative sur
  certains espaces homog\`enes}, Lecture Notes in Mathematics, Vol. 242,
  Springer-Verlag, Berlin-New York, 1971.

\bibitem{MR850408}
G.~David, J.-L. Journ{\'e}, and S.~Semmes, {Op\'erateurs de
  {C}alder\'on-{Z}ygmund, fonctions para-accr\'etives et interpolation}, \emph{Rev.
  Mat. Iberoamericana} \textbf{1} (1985), no.~4, 1--56. 

\bibitem{DellAntonio:1991ih}
G. Dell'Antonio and D. Zwanziger, {{Every gauge orbit passes
  inside the Gribov horizon}}, \emph{Comm. Math. Phys.} {\bf 138} (1991) 291--299.

\bibitem{DeTurck:1983ts}
D.~M DeTurck, {{Deforming metrics in the direction of their Ricci
  tensors}}, \emph{J. Diff. Geom.} {\bf 18} (1983) 157--162.

\bibitem{Dodson:2012uj}
B. Dodson, {{Bilinear Strichartz estimates for the Schr{{\"o}}dinger map problem}}, arXiv:1210.5255 [math.AP]

\bibitem{Dodson:2013vh}
B. Dodson and P. Smith, {{A controlling norm for energy-critical Schr{\"o}dinger maps}}, arXiv:1302.3879 [math.AP]

\bibitem{Donaldson:1985vh}
S.~K. Donaldson, {{Anti self-dual Yang-Mills connections over complex
  algebraic surfaces and stable vector bundles}}, \emph{P. Lond.  Math. Soc.} {\bf 50} (1985) 1--26.

\bibitem{Eardley:1982fb}
D.~M Eardley and V. Moncrief, {{The global existence of
  Yang-Mills-Higgs fields in 4-dimensional Minkowski space I--II}}, \emph{Comm. Math. Phys.} {\bf 83} (1982) 171--191, 193--212.

\bibitem{ES}
J. Eells, H. Sampson, {Harmonic mappings of Riemannian manifolds}, \emph{Amer. J. Math.} {\bf 86} (1964) 109--160.

\bibitem{Gribov:1978eh}
V. N. Gribov, {{Quantization of non-Abelian gauge theories}},
 \emph{ Nuclear Physics B} {\bf139} (1978) 1--19.

\bibitem{MR2796047}
A.~Gr\"unrock, {On the wave equation with quadratic nonlinearities in three space dimensions,} {J. Hyperbolic Diff. Eq.} (2011)

\bibitem{Klainerman:tc}
S. Klainerman, {{The Null Condition and Global Existence to Nonlinear Wave Equations}}, in \emph{Nonlinear Systems of Partial Differential Equations in Applied
Mathematics, Part 1, Santa Fe, NM, 1984} (Amer. Math. Soc, Providence, 1986), pp. 293Ð326.

\bibitem{Klainerman:ei}
S. Klainerman and M. Machedon, {{Space-time estimates for null
  forms and the local existence theorem}}, \emph{Comm. Pure Appl.  Math.} {\bf 46} (1993) 1221--1268.

\bibitem{Klainerman:1994jb}
\bysame,{{On the Maxwell-Klein-Gordon equation with finite energy}}, \emph{Duke
  Math. J.} {\bf 74} (1994) 19--44.

\bibitem{Klainerman:1995hz}
\bysame, {{Finite Energy Solutions of the Yang-Mills Equations in
  $\mathbb{R}^{3+1}$}}, \emph{Ann. of Math.}  {\bf 142} (1995) 39--119.

\bibitem{Klainerman:1995vs}
\bysame, {{Smoothing estimates for null forms and applications}}, \emph{Duke.
  Math. J.} {\bf 81} (1995) 99--133.

  \bibitem{MR2125732}
S. Klainerman and I. Rodnianski, {Causal geometry of
  {E}instein-vacuum spacetimes with finite curvature flux}, \emph{Invent. Math.} {\bf 159} (2005) 437--529.
  
  \bibitem{Klainerman:2005kj}
  \bysame, {Bilinear estimates on curved space-times}, \emph{J. Hyperbolic Diff. Eq.} {\bf 2} (2005) 279--291.

  \bibitem{MR2221254}
\bysame, {A geometric approach to the
  {L}ittlewood-{P}aley theory}, \emph{Geom. Funct. Anal.} {\bf 16} (2006) 126--163.

\bibitem{MR2221255}
\bysame, {Sharp trace theorems for null hypersurfaces on {E}instein
  metrics with finite curvature flux}, \emph{Geom. Funct. Anal.} {\bf 16} (2006) 164-229.

\bibitem{Klainerman:1999do}
S. Klainerman and D. Tataru, {On the optimal local regularity for Yang-Mills equations in $\bf R^{4+1}$}, \emph{J. Amer. Math. Soc.} {\bf 12} (1999) 93--116.

\bibitem{Kobayashi:1963uh}
S. Kobayashi and K. Nomizu, \emph{{Foundations of Differential
  Geometry: Vol 1--2}}, (Wiley Interscience, 1963), (Wiley Interscience, 1969).

\bibitem{Lindblad:1996ws}
H. Lindblad, {{Counterexamples to local existence for semi-linear wave
  equations}}, \emph{Amer. J. Math.} {\bf 118} (1996) 1--16.

\bibitem{MR1998349}
F.~Nazarov, S.~Treil, and A.~Volberg, {The {$Tb$}-theorem on
  non-homogeneous spaces}, \emph{Acta Math.} \textbf{190} (2003), no.~2, 151--239.
  \MR{1998349 (2005d:30053)}

\bibitem{Oh:6stz7nRe}
S.-J. Oh, {{Gauge choice for the Yang-Mills equations using the
  Yang-Mills heat flow and local well-posedness in $H^{1}$}}, arXiv:1210.1558 [math.AP]. \emph{J. of Hyperbolic Diff. Eq.} {\bf 11} (2014) 1669--1732.

\bibitem{thesis}
\bysame, \emph{{Finite energy global well-posedness of the (3+1)-dimensional Yang-Mills equations using a novel Yang-Mills heat flow gauge}}. Ph.D. Thesis (Princeton University, 2013).
\bibitem{Oh:uq}
\bysame, {Almost optimal local well-posedness of the (1+4)-dimensional Yang-Mills equations}, in preparation.

\bibitem{Rade:1992tu}
J. R{\aa}de, {{On the Yang-Mills heat equation in two and three
  dimensions}}, \emph{J. Reine Angew. Math.} {\bf 431} (1992) 123--164.

\bibitem{Segal:1979hg}
I. Segal, {{The Cauchy Problem for the Yang-Mills Equations}}, \emph{J. Funct. Anal.} {\bf 33} (1979) 175--194.

\bibitem{Smith:2010ui}
P. Smith, {Conditional global regularity of Schr{\"o}dinger maps: sub-threshold dispersed energy}, arXiv:1012.4048 [math.AP]

\bibitem{Smith:2011ef}
\bysame, {{Geometric Renormalization Below the Ground State}},
 {\it Int. Math. Res. Notices} {\bf 2012} (2012) 3800--3844.

\bibitem{Smith:2011ty}
\bysame, {{Global regularity of critical Schr{\"o}dinger maps:
  subthreshold dispersed energy}}, arXiv:1112.0251 [math.AP]

\bibitem{MR0252961}
E.~M.~Stein, \emph{Topics in harmonic analysis related to the
  {L}ittlewood-{P}aley theory.}, Annals of Mathematics Studies, No. 63,
  Princeton University Press, Princeton, N.J., 1970. 
  
\bibitem{Struwe:1994is}
M. Struwe, {{The Yang-Mills flow in four dimensions}}, \emph{Calc. Var. and Partial Diff. Eq.} (1994)

\bibitem{Tao:2000vba}
T. Tao, {{Local well-posedness of the Yang-Mills equation in the
  Temporal Gauge below the energy norm}}, \emph{J. Diff. Eq.} {\bf 189} (2003) 366--382.

\bibitem{Tao:2004tm}
\bysame, {{Geometric renormalization of large energy wave maps}},
  \emph{Journ{\'e}es {\'e}quations aux d{\'e}riv{\'e}es partielles}, {\bf 11} (2004) 1--32.

\bibitem{Tao:2008wn}
\bysame, Global regularity of wave maps III -- VII, 
arXiv:0805.4666 [math.AP], arXiv:0806.3592 [math.AP], arXiv:0808.0368 [math.AP], arXiv:0906.2833 [math.AP], arXiv:0908.0776 [math.AP]

\bibitem{MR3154530}
X.~Tolsa, \emph{Analytic capacity, the {C}auchy transform, and
  non-homogeneous {C}alder\'on-{Z}ygmund theory}, Progress in Mathematics, vol.
  307, Birkh\"auser/Springer, Cham, 2014. 

\bibitem{Uhlenbeck:1982vna}
K.~K. Uhlenbeck, {{Connections with $L^{p}$ bounds on curvature}},
  \emph{Comm. Math. Phys.} {\bf 83} (1982) 31--42.

\end{thebibliography}
% ----------------------------------------------------------------

\end{document}